\documentclass[12pt,reqno]{amsart}
\usepackage[margin=1in]{geometry}
\usepackage{amsmath,amssymb,amsthm,graphicx,amsxtra, setspace}
\usepackage[utf8]{inputenc}
\usepackage{mathrsfs}
\usepackage{hyperref}
\usepackage{upgreek}
\usepackage{mathtools}
\allowdisplaybreaks

\newtheorem{theorem}{Theorem}[section]
\newtheorem{lemma}[theorem]{Lemma}
\newtheorem{proposition}[theorem]{Proposition}

\newtheorem{corollary}[theorem]{Corollary}
\newtheorem{definition}[theorem]{Definition}
\newtheorem{example}[theorem]{Example}
\newtheorem{remark}[theorem]{Remark}

\let\originalleft\left
\let\originalright\right
\renewcommand{\left}{\mathopen{}\mathclose\bgroup\originalleft}
\renewcommand{\right}{\aftergroup\egroup\originalright}

\newcommand{\Tr}{\mathop{\mathrm{Tr}}}
\renewcommand{\d}{\/\mathrm{d}\/}

\def\e{\varepsilon}

\def\L{\mathrm{L}}
\def\A{\mathrm{A}}
\def\I{\mathrm{I}}
\def\2{\varsigma}
\def\1{\mathcal{O}}
\def\F{\mathrm{F}}
\def\C{\mathrm{C}}
\def\c{{c}}

\def\J{\mathrm{J}}
\def\B{\mathrm{B}}
\def\D{\mathrm{D}}
\def\y{\mathrm{y}}

\def\E{\mathbb{E}}
\def\X{\mathrm{X}}
\def\x{\mathrm{x}}

\def\z{\mathrm{z}}
\def\v{\mathfrak{v}}

\def\W{\mathrm{W}}
\def\G{\mathrm{G}}
\def\Q{\mathrm{Q}}
\def\M{\mathrm{M}}
\def\N{\mathbb{N}}

\def\wi{\widetilde}
\def\Q{\mathrm{Q}}
\def\u{\mathrm{U}}
\def\P{\mathbb{P}}
\def\u{\mathfrak{u}}
\def\H{\mathrm{H}}

\def\3{\varrho}

\newcommand{\R}{\mathbb{R}}

\renewcommand{\d}{\/\mathrm{d}\/}

\newcommand{\Addresses}{{
		\footnote{
			\noindent \textsuperscript{1,2}Department of Mathematics, Indian Institute of Technology Roorkee-IIT Roorkee,
			Haridwar Highway, Roorkee, Uttarakhand 247667, INDIA.\par\nopagebreak
			\noindent  \textit{e-mail:} \texttt{Manil T. Mohan: maniltmohan@ma.iitr.ac.in, maniltmohan@gmail.com.}
			
			\textit{e-mail:} \texttt{Ankit Kumar: akumar14@mt.iitr.ac.in.}
			
			\noindent \textsuperscript{*}Corresponding author.
			
			\textit{Key words:} Stochastic generalized Burgers-Huxley equation,  Irreducibility, Strong Feller, Invariant measures, Large deviation principle, Occupation measures.
			
			Mathematics Subject Classification (2020): Primary 60F10; Secondary 60J35; 60A10; 37L40.

}}}

\begin{document}
	
	\title[LDP for the SGBH equation]{Large deviation principle for occupation measures of stochastic generalized Burgers-Huxley equation
		\Addresses}
	
	\author[A. Kumar and M. T. Mohan]
	{Ankit Kumar\textsuperscript{1} and Manil T. Mohan\textsuperscript{2*}}

	\maketitle
	
	\begin{abstract}
		The present work deals with the global solvability as well as asymptotic analysis of  stochastic generalized Burgers-Huxley (SGBH) equation perturbed by space-time white noise in a bounded interval of $\R$. We first prove the existence of unique mild  as well as strong solution  to SGBH equation and then obtain the existence of an invariant measure. Later, we establish two major properties of the Markovian semigroup associated with the solutions of SGBH equation, that is, irreducibility and strong Feller property. These two properties guarantees the uniqueness of invariant measures and ergodicity also.  Then, under further assumptions on the noise coefficient, we discuss the ergodic behavior of the solution of SGBH equation by providing a Large Deviation Principle (LDP) for the occupation measure for large time (Donsker-Varadhan), which describes the exact rate of exponential convergence. 
	\end{abstract}
	
	\section{Introduction}
	The stochastic generalized Burgers-Huxley equation (SGBH) equation describes a prototype model for describing the interaction between reaction mechanisms, convection effects and diffusion transports (cf. \cite{MTMAK}). We consider the generalized Burgers-Huxley equation perturbed by a random forcing, which is a space-time white noise  (or Brownian sheet), as
	\begin{align}\label{1.1}
		\frac{\partial u(t,\xi)}{\partial t} &= \nu \frac{\partial^{2}u(t,\xi)}{\partial \xi^{2}} -\alpha u^{\delta}(t,\xi)\frac{\partial u(t,\xi)}{\partial \xi} + \beta u(t,\xi)(1-u^{\delta}(t,\xi))(u^{\delta}(t,\xi)-\gamma)\nonumber\\&\quad+ \G \frac{\partial^2\wi\W(t,\xi)}{\partial \xi\partial t},
	\end{align}
	for $(t,\xi)\in(0,T)\times(0,1)$,	where $\alpha>0$ is the advection coefficient, $\beta>0$, $\delta\geq 1$ and $\gamma \in (0,1)$ are parameters. The noise coefficient $\G:\L^2(0,1)\to \L^2(0,1)$ is a bounded linear operator, $\wi\W(t,\xi),$ $t\geq 0$, $x\in(0,1)$ is a zero mean Gaussian process, whose covariance function is given by 
	$$\mathbb{E}\left[\wi\W(t,\xi)\wi\W(s,\zeta)\right]=(t\wedge s)(\xi\wedge \zeta),\ t,s\geq 0,\ \xi,\zeta\in\mathbb{R}.$$ On the other hand, one can consider a cylindrical Brownian process $\W(\cdot)$ by setting 
	\begin{align}\label{1p2}
		\W(t)=\frac{\partial\wi\W(t)}{\partial \xi}=\sum_{k=1}^{\infty}e_k\beta_k(t),
	\end{align}
	where $\{e_k\}_{k=1}^{\infty}$ is an orthonormal basis of $\L^2(0,1)$ and $\{\beta_k	\} _{k=1}^{\infty}$ is a sequence of  independent real Brownian motions in a fixed probability space $(\Omega,\mathscr{F},\mathbb{P})$ adapted to a filtration $\{\mathscr{F}_t\}_{t\geq 0}$. It is well-known that the series \eqref{1p2} does not converge in $\L^2(0,1)$, but it is convergent in any Hilbert space $\mathrm{U}$ such that the embedding $\L^2(0,1)\subset\mathrm{U}$ is Hilbert-Schimdt (cf. \cite{DaZ}).   With the above formulation, we rewrite the equation \eqref{1.1} as 
	\begin{align}\label{1p1}
		\d u(t)&=\left(\nu \frac{\partial^{2}u(t,\xi)}{\partial \xi^{2}} -\alpha u^{\delta}(t,\xi)\frac{\partial u(t,\xi)}{\partial \xi} + \beta u(t,\xi)(1-u^{\delta}(t,\xi))(u^{\delta}(t,\xi)-\gamma)\right)\d t\nonumber\\&\quad+\G\d\W(t). 
	\end{align}
	The equation \eqref{1p1} is supplemented by the Dirichlet boundary condition:
	\begin{align}\label{1.2}
		u(t,0)=u(t,1)=0, \ t\geq 0,
	\end{align}
	and the initial condition 
	\begin{align}\label{1p3}
		u(0,\xi)=x(\xi). 
	\end{align}
	Let $\L^2(\1):=\L^2(0,1)$ and $\A:=-\frac{\partial^2}{\partial \xi^2}$. In order to prove the existence of strong solution, for $\Q =\G\G^*$, we assume that 
	\begin{align}\label{1.3} 
		\D(\A^{\frac{\e}{2}}) \subset \mathrm{Im}(\Q^{\frac{1}{2}}), \ \text{ for some } \ 0<\e<1,
	\end{align} 
	where $\mathrm{Im}(\Q^{\frac{1}{2}})$ denotes the range of the operator $\Q^{\frac{1}{2}}$. It is equivalent to say that the range of the definition of $\A^{-\frac{\e}{2}}$ in $\L^2(\1)$ is contained in $\mathrm{Im}(\Q^{\frac{1}{2}})$. Under the assumption \eqref{1.3}, for any $\nu,\alpha,\beta>0,\delta<p<\infty$, the existence of a unique mild solution $u\in \C([0,T];\L^p(\mathcal{O})),$ $\mathbb{P}$-a.s., to the system \eqref{1p1}-\eqref{1p3} is established in \cite{MTMSGBH}.  For $1\leq \delta<2$, we prove the existence of an invariant measure for the  system \eqref{1p1}-\eqref{1p3}. Under the following assumption: 
	\begin{align}\label{13} 
		\D(\A^{\frac{\e}{2}}) \subset \mathrm{Im}(\Q^{\frac{1}{2}}), \ \text{ for some }\  \frac{1}{2}<\e<1,
	\end{align} 
	we prove the existence of a strong solution to the system \eqref{1p1}-\eqref{1p3},  for any $\nu,\alpha,\beta>0,$ $\gamma\in(0,1)$ and $1\leq \delta<\infty$. For $\delta\in[1,2]$, the uniqueness of strong solution is established for any $\nu,\alpha,\beta>0, \gamma\in(0,1)$ and  for $2<\delta<\infty$, the uniqueness is obtained for $\beta\nu>2^{2(\delta-1)}\alpha^2$. For these cases, we  prove the existence of an invariant measure for any real $\delta\geq 1$ as well as a Large Deviation Principle (LDP) for the occupation measure for large time (Donsker-Varadhan) for $\delta\in[1,2]$. 	The assumption \eqref{13} implies $\Tr(\G\G^*)<\infty$, which tells us that the energy injected by the random force is finite. The condition \eqref{13} also indicates that the noise is not too degenerate. 
	
	The stochastic Burgers' equation perturbed by cylindrical Gaussian noise is considered in the work \cite{DDR}, where the authors established the existence and uniqueness of mild solution, along with the existence of an invariant measure. The uniqueness of invariant measure is obtained in \cite{GDJZ}, by showing that the Markov semigroup associated with the solution is irreducible and strong Feller  (Chapter 14, \cite{GDJZ}). The existence and uniqueness of invariant measures for stochastic Burgers equations perturbed by multiplicative noise is established in \cite{GDG}. For a sample literature on stochastic Burgers equations, the interested readers are referred to see \cite{AG,BJ,GDAD1,DDR,MG2}, etc and the references therein. For a comprehensive study on ergodicity for infinite dimensional systems, one may refer to \cite{GDJZ,ADe}, etc.  The irreducibility of the semigroup corresponding to the solution of stochastic real Ginzburz-Landau equation driven by $\alpha$-stable noises is proved in \cite{RWJXLX}.  The global solvability results (the existence and uniqueness of strong solutions) and asymptotic analysis (the existence and uniqueness of invariant measures) of stochastic Burgers-Huxley equation is carried out in the paper \cite{MTMSBHE}.  In the work \cite{MTMSGBH}, the author studied SGBH equation perturbed by space-time white noise and established the existence and uniqueness of mild solution with the help of fixed point and stopping time arguments. Ergodicity results for the stochastic real Ginzburg-Landau equation driven by $\alpha$-stable noises are available in \cite{RWJXLX,LXu}, etc.

	The theory of large deviations, which provides asymptotic estimates for the probabilities of the rare events, is one of the important research topics in probability theory and received the required attention after the contributions of Varadhan. One can find the theory of large deviation along with its applications in \cite{DZ,DE,ST,VA}, etc. Several authors have established the Wentzell-Freidlin type large deviation principle (LDP) for different classes of stochastic partial differential equations (SPDEs) (cf. \cite{Chow1,KX,SW}, etc). By using weak convergence approach, the Wentzell-Freidlin type LDP for 2D stochastic Navier-Stokes equations (SNSE) perturbed by a small multiplicative noise in both bounded and unbounded domains has been obtained in \cite{SSSP}. Exponential estimates for exit from a ball of radius $R$ by time $T$ for	solutions of the stochastic Burgers-Huxley equation  in the context of Freidlin-Wentzell type LDP is studied in  \cite{MTMSBHE}.

	An SPDE is ergodic means that the occupation measure of it's solution converges to a unique invariant measure. A Donsker-Varadhan type LDP provides an estimate on the probability of occupation measures deviations from the invariant measure (cf. \cite{JD,MDD}). Thus, it is quite interesting to ask whether the occupation measures satisfy the Donsker-Varadhan type LDP. Similar to Wentzell-Freidlin type LDP, a good number of works are available in the literature regarding Donsker-Varadhan type LDP (cf. \cite{MDD,MG2,MG1,JVNV,JVNV1,LW1} etc and the references therein). A criterion for LDP for occupation measures has been developed in the work \cite{LW1}, the so-called hyper-exponential recurrence for strong Feller and irreducible Markov process. However, the recurrence condition is very strong and it is not easy to verify this condition for SPDEs. The author in \cite{MG2} and \cite{MG1} verified this recurrence and proved the LDP for occupation measures for stochastic Burgers equation and 2D SNSE in bounded domains. Recently, the authors in \cite{WA} verified the hyper-exponential recurrence and proved the LDP for occupation measures for a class of non-linear monotone SPDEs including the stochastic porous medium equation, stochastic $p$-Laplace equation,  stochastic fast-diffusion equation, etc. In the context of  stochastic convective Brinkman-Forchhemier equations, the hyper-exponential recurrence is verified in \cite{AKMTM} and the authors proved LDP of occupation measures.  Authors in \cite{JVNV1} established the LDP for occupation measures for a class of dissipative PDE's perturbed by a bounded random kick force. LDP for occupation measures for a class of dissipative PDE's perturbed by an unbounded kick force is studied in \cite{JVNV} and for stochastic reaction-diffusion equations driven by subordinate Brownian motions is established in \cite{RWLX}.

	In the present work, we first prove the existence and uniqueness of mild as well as strong solutions for the system  \eqref{1p1}-\eqref{1p3} under the assumptions \eqref{1.3} and \eqref{13}, respectively. The existence of  an invariant measure, strong Feller and topological irreducibility properties of the Markov semigroup corresponding to the solution of SGBH equation \eqref{1.1} and hence  the uniqueness of invariant measure (Doob's theorem) are also obtained. Then we discuss the ergodic behavior of SGBH equation by providing an LDP for occupation measures w.r.t. the stronger $\tau$-topology and an LDP of Donsker-Varadhan. To prove LDP w.r.t. the $\tau$- topology for SGBH equation, we established the hyper-exponential recurrence given in \cite{LW1}.  Let us now summarize the results obtained in this work as a table to emphasize the dependence of noise and different parameters appearing in  \eqref{1.1}. 
	
	\begin{table}[h]
		\begin{tabular}{|p {3.5cm}|p {5.9cm}|p {5.9cm}|}
			\hline &Mild solution (Assm. \eqref{1.3})&Strong solution (Assm. \eqref{13})\\
			\hline $\delta\in[1,\infty)$&existence for any $\nu,\alpha,\beta>0, \gamma\in(0,1)$ &existence for any $\nu,\alpha,\beta>0, \gamma\in(0,1)$\\
			\hline $\delta\in[1,2]$&existence and uniqueness  for any $\nu,\alpha,\beta>0, \gamma\in(0,1)$ &existence and uniqueness for any $\nu,\alpha,\beta>0, \gamma\in(0,1)$ \\
			\hline $\delta\in(2,\infty), \beta\nu>2^{2(\delta-1)}\alpha^2$&existence and uniqueness&existence and uniqueness\\
			\hline Invariant measure &existence for $\delta\in[1,2)$&existence for $\delta\in[1,\infty)$ with $ \beta\nu>2^{2(\delta-1)}\alpha^2$ for $\delta\in(2,\infty)$  \\
			\hline Irreducibility &$\delta\in[1,2]$&$\delta\in[1,2]$  \\
			\hline Strong Feller &$\delta\in[1,\infty)$ with $ \beta\nu>2^{2(\delta-1)}\alpha^2$ for $\delta\in(2,\infty)$ &$\delta\in[1,\infty)$ with $ \beta\nu>2^{2(\delta-1)}\alpha^2$ for $\delta\in(2,\infty)$  \\
			\hline Uniqueness of invariant measure  &$\delta\in[1,2)$&$\delta\in[1,2]$  \\
			\hline LDP&-- &$\delta\in[1,2]$   \\
			\hline
		\end{tabular}
		\caption{Assumptions on noise and restrictions on $\delta$.}
	\end{table}
	
	The article is organized as follows. In section \ref{sec2}, we define linear and nonlinear operators and the necessary function space needed to obtain the solvability and LDP results for our model. Then we provide the abstract formulation of SGBH equation (see \eqref{2.2}) perturbed by the non-degenerate additive noise and discuss the existence and uniqueness of mild as well as strong solutions (Theorems \ref{thm2.6} and \ref{thm2.7}). The existence and uniqueness of invariant measures for our model is discussed in section \ref{sec3}. Under assumption \eqref{1.3} and $1\leq\delta<2$, we followed similar arguments  as in \cite{GDJZ} for the existence of   invariant measure  (Theorem \ref{thEx2}). Using energy equality (It\^o's formula), we obtained the existence of invariant measure under the assumption \eqref{13} and $1\leq\delta<\infty$ (Theorem \ref{thEx1}). Then, we discussed  two properties of the Markov semigroup associated with the solutions of SGBH equations, that is, irreducibility   and strong Feller property (Propositions \ref{prop4.1} and \ref{prop4.3}).  For the  proof of strong Feller property we followed the book \cite{GDJZ} and irreducibility we borrowed ideas from \cite{GDJZ} and  \cite{RWJXLX}. We stated our main result of  Donsker-Varadhan type LDP of occupation measures for the solution of SGBH equation in section \ref{sec5} (Theorem \ref{thrm5.1} and Corollary \ref{cor4.2}) with the help of exponential estimates for the strong solution of SGBH equation (Proposition \ref{prop4.10}) and the hyper-exponential recurrence given in \cite{LW1}.

	\section{Mathematical Formulation}\label{sec2}\setcounter{equation}{0}
	This section provides the necessary function spaces needed to obtain the major  results of this paper. 
	\subsection{Function spaces} Let us fix $\mathcal{O}=(0,1)$. Let $\C_{0}^{\infty}(\1)$ denote the space of all infinitely differentiable functions having compact support in $\1$. The Lebesgue spaces are denoted by $\L^{p}(\1)$ for $p\in [1,\infty]$, and the norm in $\L^p(\1)$ is denoted by $\|\cdot\|_{\L^p}$ and for $p=2$, the inner product in $\L^2(\1)$ is denoted by $(\cdot,\cdot)$. We denote the Sobolev spaces by $\H^{k}(\1)$. Let $\H_0^1(\1)$ denote the closure of $\C_{0}^{\infty}(\1)$ in $\H^1$-norm. As we are working in $\1$ (bounded domain) by using Poincar\'e inequality, the norm $(\|\cdot\|_{\L^2}^2+\|\partial_{\xi}\cdot\|_{\L^2}^2)^{\frac{1}{2}}$ is equivalent to the seminorm $\|\partial_{\xi}\cdot\|_{\L^2}$ and hence $\|\partial_{\xi}\cdot\|_{\L^2}$ defines a norm on $\H_0^1(\1)$. We have the continuous embedding $\H_0^1(\1) \subset \L^2(\1) \subset \H^{-1}(\1)$, where $\H^{-1}(\1)$ is the dual space of $\H_0^1(\1)$. For the bounded domain $\1$ the embedding  $\H_0^1(\1) \subset \L^2(\1)$ is compact. The duality pairing between $\H_0^1(\1)$ and its dual $\H^{-1}(\1)$, and $\L^p(\mathcal{O})$ and its dual $\mathrm{L}^{\frac{p}{p-1}}$ are   denoted by $\langle \cdot, \cdot \rangle$. In one dimension, we have the continuous embedding: $\H_0^1(\1) \subset \L^{\infty}(\1) \subset \L^p(\1)$, for $p\in [1,\infty)$. Also the embedding of $\H^{\sigma}(\1) \subset \L^p(\1)$ is compact for any $\sigma> \frac{1}{2}-\frac{1}{p}$, for $p\geq 2$. The following interpolation inequality will be used frequently in the paper. 
	Assume that $1\leq p \leq r\leq q\leq \infty$ and $\frac{1}{r}=\frac{\alpha}{p}+\frac{1-\theta}{q}$. For $u\in\L^p(\mathcal{O})\cap\L^q(\mathcal{O})$, we have  $u\in\L^r(\mathcal{O})$, and \begin{align*}
		\|u\|_{\L^r}\leq \|u\|_{\L^p}^{\theta}\|u\|_{\L^q}^{1-\theta}.
	\end{align*}The following fractional form of Gagliardo-Nirenberg inequality (see \cite{CMLP} and \cite{LN}) is also used in the sequel. 	Fix $1\leq q, \; l\leq \infty$ and a natural number $n$. Suppose also that a real number $\theta$ and a non-negative number $j$ are such that 
\begin{align*}
		\frac{1}{p} = \frac{j}{n}+\bigg(\frac{1}{l}-\frac{m}{n}\bigg)\theta + \frac{1-\theta}{q}, \;\; \frac{j}{m}\leq \theta \leq 1, 
	\end{align*} then we have 
	\begin{align*}
		\|\D^{j}u\|_{\L^{p}} \leq C\|\D^{m}u\|^{\theta}_{\L^{l}}\|u\|^{1-\theta}_{\L^{q}},
	\end{align*} 
	for all $u\in\mathbb{W}^{m,l}(\mathcal{O})\cap\H_0^1(\mathcal{O})$.

	\subsection{Linear operator}\label{subsec2.2} Let $\A$ denote the self-adjoint operator and unbounded operator on $\L^2(\1)$ defined by 
	\begin{align*}
		\A u:= -\frac{\partial^2 u }{\partial \xi^2},
	\end{align*}
	with domain $\D(\A)= \H^2(\1) \cap \H_0^1(\1)= \{u\in \H^2(\1):u(0)=u(1)=0\}$. The eigenvalues and the corresponding eigenfunctions of $\A$ are given by 
	\begin{align*}
		\lambda_k = k^2\pi^2 \ \text{ and } \ e_k(\xi) = \sqrt{\frac{2}{\pi}} \sin(k\pi x), \; k=1,2\ldots	
	\end{align*}
	As we are working in the bounded domain $\1$, the inverse of $\A$, that is, $\A^{-1}$ exists and is a compact operator on $\L^2(\1)$. Moreover, we can define the fractional powers of $\A$ and 
	\begin{align*}
		\|\A^{\frac{1}{2}}u\|_{\L^2}^2=\sum_{j=1}^{\infty}|(u,e_j)|^2 \geq \lambda_1\sum_{j=1}^{\infty}|(u,e_j)|^2 = \lambda_1\|u\|_{\L^2}^2 = \pi^2\|u\|_{\L^2}^2,
	\end{align*} 
	which is the Poincar\'e inequality. An integration by parts yields 
	\begin{align*}
		(\A u , v) = (\partial_{\xi} u , \partial_{\xi} v) =: a(u,v), \ \text{ for all } \ v \in \H_0^1 (\1),
	\end{align*} 
	so that $\A: \H_0^1(\1)\to \H^{-1}(\1)$. Let us define the operator $\A_p = -\frac{\partial^2}{\partial \xi^2}$ with $\D(\A_p)= \W_0^{1,p}(\1)\cap \W^{2,p}(\1)$, for $1<p<\infty$ and $\D(\A_1)= \{u\in \W^{1,1}(\1): u \in \L^1(\1)\}$, for $p=1$. From Proposition 4.3, Chapter 1 \cite{VB1, AP}, we know that for $1\leq p<\infty$, $\A_p$ generates an analytic semigroup of contractions in $\L^p(\1)$.
	\subsection{Nonlinear operators}\label{subsec2.3} We define two nonlinear operators in this subsection.
	\subsubsection{The operator $\B(\cdot)$.}
	Let us define $b:\H_0^1(\1)\times\H_0^1(\1)\times \H_0^1(\1) \to \R$ as 
	\begin{align*}
		b(u,v,w) = \int_{0}^{1} (u(\xi))^{\delta}\frac{\partial v(\xi)}{\partial x} w(\xi)\d \xi.
	\end{align*} 		
	Using an integration by parts and boundary conditions, it is immediate  that 
	\begin{align*}
		b(u,u,u) = (u^{\delta}\partial_{\xi} u , u ) =\int_{0}^{1} (u(\xi))^{\delta}\frac{\partial u(\xi)}{\partial \xi}u(\xi) \d \xi = \frac{1}{\delta+2}\int_{0}^{1} \frac{\partial}{\partial \xi}(u(\xi))^{\delta+2}\d \xi =0,
	\end{align*} and 
	\begin{align*}
		b(u,u,v) = -\frac{1}{\delta+1}b(u,v,u),
	\end{align*} for all $u,v \in \H_0^1(\1)$. In general, for all $p>2$ and $u \in \H_0^1(\1)$, using integration by parts one can easily show that
	\begin{align}\label{2.1}
		b(u,u,|u|^{p-2}u)= (u^{\delta}\partial_{\xi} u , |u|^{p-2}u)=0,
	\end{align} for all $p\geq 2$ and $u \in \H_0^1(\1)$. For $w \in \L^2(\1)$, we define $\B(\cdot,\cdot):\H_0^1(\1)\times \H_0^1(\1)\to \L^2(\1)$ by 
	\begin{align*}
		(\B(u,v),w) = b(u,v,w) \leq \|u\|_{\L^{\infty}}^{\delta} \|\partial_{\xi} v\|_{\L^2}\|w\|_{\L^2} \leq \|u\|_{\H_0^1}^{\delta}\|v\|_{\H_0^1}\|w\|_{\L^2},
	\end{align*} so that $\|\B(u,v)\|_{\L^2}\leq \|u\|_{\H_0^1}\|v\|_{\H_0^1}$. We set $\B(u)=\B(u,u)$, so that we can easily obtain $\|\B(u)\|_{\L^2} \leq \|u\|_{\H_0^1}^{\delta+1}$. One can  show that the operator $\B(\cdot)$ is a locally Lipschitz operator, that is, (cf. \cite{MTMAK})
	\begin{align}\label{BLL}
		\|\B(u)-\B(v)\|_{\L^2}&\leq    C\delta(1+2^{\delta})r^{\delta}\|u-v\|_{\H_0^1},
	\end{align}for $\|u\|_{\H_0^1},\|v\|_{\H_0^1}\leq r$. 
	
	\subsubsection{The operator $\c(\cdot)$.}
	Let us define the operator $\c:\H_0^1(\1)\to \L^2(\1)$ by
	$\c(u)=u(1-u^{\delta})(u^{\delta}-\gamma)$. It is easy to compute that
	\begin{align*}
		(\c(u),u)& = (u(1-u^{\delta})(u^{\delta}-\gamma),u)=((1+\gamma)u^{\delta+1}-\gamma u - u^{2\delta+1},u)
		\\& = (1+\gamma)(u^{\delta+1},u)-\gamma \|u\|_{\L^2}^2-\|u\|_{\L^{2(\delta+1)}}^{2(\delta+1)},
	\end{align*} for all $u\in \L^{2(\delta+1)}(\1) \subset \H_0^1(\1)$. The  operator $\c(\cdot)$ is also a locally Lipschitz, that is, (cf. \cite{MTMAK})
	\begin{align}\label{CLL}
		&\|\c(u)-\c(v)\|_{\L^2}  \leq
		\frac{C}{\pi}((1+\gamma)(1+\delta)2^{\delta}r^{\delta}+\gamma+(1+2\delta)2^{2\delta}r^{2\delta})\|u-v\|_{\H_0^1},
	\end{align}for $\|u\|_{\H_0^1},\|v\|_{\H_0^1} \leq r$.

	\subsection{Linear problem} Let us first consider the following stochastic heat equation: 
	\begin{equation}\label{2p2}
		\left\{
		\begin{aligned}
			\d z(t)&=-\nu\A z(t)\d t +\G\d\W(t), \; t\in(0,T),\\
			z(0)&=0,
		\end{aligned}
		\right.
	\end{equation}where $\G$ satisfies the assumption \eqref{1.3}. Then, from Chapter 5, \cite{DaZ}, we infer that the solution of \eqref{2p2} is unique and it can defined by the stochastic convolution
	\begin{align*}
		z(t)= \int_{0}^{t}R(t-s)\G\d\W(s),
	\end{align*}where $R(t)=e^{-\nu t\A}$. Also note the process $z$ is a Gaussian process and it is mean square continuous taking values in $\L^2(\1)$ and $z$ has a version, which has $\P$-a.s., $\lambda$-H\"older continuous paths w.r.t. $(t,\xi)\in [0,T]\times[0,1]$ for any $\lambda\in(0,\frac{1}{2})$	(for more details see Theorem 5.22, \cite{DaZ}). Under the assumption \eqref{13}, one can show that $z\in\C(0,T;\L^2(\mathcal{O}))\cap\mathrm{L}^2(0,T;\H_0^1(\mathcal{O}))$, $\mathbb{P}$-a.s.
	
	Since $R(t)=e^{-\nu t\A}$ is an analytic semigroup, we infer that $R(t):\L^p(\mathcal{O})\to\L^q(\mathcal{O})$ is a bounded map whenever $1\leq p\leq q\leq \infty$ and $t>0$, and there exists a constant $C$ depending on $p,q$ and $\nu$ such that (see Lemma 3, Part I, \cite{FR})
	\begin{align}
		\|R(t)f\|_{\L^q}&\leq Ct^{-\frac{1}{2}\left(\frac{1}{p}-\frac{1}{q}\right)}\|f\|_{\L^p},\label{1.6}\\
		\|\partial_{\xi} R(t)f\|_{\L^q}&\leq Ct^{-\frac{1}{2}-\frac{1}{2}\left(\frac{1}{p}-\frac{1}{q}\right)}\|f\|_{\L^p},\label{1.5}
	\end{align}
	for all $t\in(0,T]$ and $f\in\L^p(\mathcal{O})$. Moreover, we have 
	\begin{align}\label{1.7}
		\|\A^{\sigma}R(t)\|_{\mathcal{L}(\L^p)}\leq\frac{C}{t^{\sigma}},
	\end{align}
	for any $p\geq 1$	and $t\in(0,T]$.
	
	\subsection{Mild solution} In this subsection,  we provide the definition of mild solution and we state a result form \cite{MTMSGBH}, where the existence and uniqueness of a global mild solution to the SGBH equation \eqref{1p1} is established. One can re-write the abstract formulation of the problem \eqref{1p1} as
	\begin{equation}\label{2.2}
		\left\{
		\begin{aligned}
			\d  u(t)&=\{-\nu\A u(t)-\alpha\B (u(t))+\beta \c(u(t))\}\d t + \G\d \W(t), \ t\in(0,T), \\
			u(0)&= x,
		\end{aligned}
		\right.
	\end{equation}
	where $x\in\L^p(\mathcal{O})$, for $p>\delta$. 
	\begin{definition}\label{def2.1}
		An $\L^p(\1)$-valued and $\mathscr{F}_t$-adapted stochastic process $u:[0,\infty)\times[0,1]\times\Omega \to \R $ with $\P$-a.s continuous trajectories on $t\in[0,T]$, is a mild solution to \eqref{2.2}, if for any $T>0$, $u(t):=u(t,\cdot,\cdot)$ satisfies the following integral equation:\begin{align}\label{2.3}
			u(t)=R(t)x-\alpha\int_{0}^{t}R(t-s)\B(u(s))\d s +\beta\int_{0}^{t}R(t-s)\c(u(s))\d s +\int_{0}^{t}R(t-s)\G\d\W(s),
		\end{align} $\P$-a.s., for all $t\in[0,T]$.
	\end{definition} 
	
	\begin{theorem}{\cite{MTMSGBH}}\label{thrm2.2}
		Let the $\mathscr{F}_0$-measurable initial data $x$ be given and $x \in \L^p(\1)$ for $p>\delta$, $ \P$-a.s. Then there exists a unique mild solution of \eqref{2.2}, which belongs to $\C([0,T];\L^p(\1)),$ for $p>\delta$, $\P$-a.s.
	\end{theorem}

	\subsection{Strong solution}
	Let us now discuss the existence and uniqueness of strong solution to the system under the assumption \eqref{13}. 
	\begin{definition}[Strong solution]\label{def2.3}
		Let $x\in \L^2(\1)$ be given. An $\L^2(\1)$-valued $\{\mathscr{F}\}_{t\geq 0}$-adapted stochastic process $u(\cdot)$ is called strong solution to the system \eqref{2.2} if 
		\begin{itemize}
			\item the process $$u \in \L^2(\Omega;\L^{\infty}(0,T;\L^2(\1))\cap \L^2(0,T;\H_0^1(\mathcal{O})))\cap \L^{2(\delta+1)}(\Omega; \L^{2(\delta+1)}(0,T;\L^{2(\delta+1)}(\mathcal{O})))$$ and $u \in \C([0,T];\L^2(\1))\cap\L^2(0,T;\H_0^1(\mathcal{O}))\cap\L^{2(\delta+1)}(0,T;\L^{2(\delta+1)}(\mathcal{O})),\; \P$-a.s.,
			\item the process $u$ satisfies 
			\begin{align*}
				(u(t),\varphi)&=(x,\varphi)+\int_0^t\langle -\nu\A u(s)-\alpha\B (u(s))+\beta \c(u(s)) ,\varphi\rangle\d s+\int_0^t(\G\d\W(s),\varphi),
			\end{align*}
			for all $t\in[0,T]$ and $\varphi\in\H_0^1(\mathcal{O})$, $\mathbb{P}$-a.s., 
			\item the following energy equality is satisfied: 
			\begin{align*}
				&\|u(t)\|_{\L^2}^2+2\nu\int_0^t\|\partial_{\xi}u(s)\|_{\L^2}^2\d s+2\beta\gamma\int_0^t\|u(s)\|_{\L^2}^2\d s+2\beta\int_0^t\|u(s)\|_{\L^{2(\delta+1)}}^{2(\delta+1)}\d s\nonumber\\&\quad=\|x\|_{\L^2}^2+2\beta(1+\gamma)\int_0^t(u^{\delta+1}(s),u(s))\d s+\mathrm{Tr}(\G\G^*)t+2\int_0^t(\G\d\W(s),u(s))\d s,
			\end{align*}
			for all $t\in[0,T]$, $\mathbb{P}$-a.s.
		\end{itemize}
	\end{definition}
	For $x\in\L^2(\1)$, in order to prove the existence of mild solution and strong solution under the assumptions \eqref{1.3} and \eqref{13}, respectively, let us first set \begin{align*}
		v(t):=u(t)-z(t), \; t\geq0.
	\end{align*} Then $u(\cdot)$ is a solution of \eqref{2.2} if and only if $v(t)$ is a solution of 
	\begin{equation}\label{2.4}
		\left\{
		\begin{aligned}
			\frac{\d  v(t)}{\d t}&=-\nu\A v(t)-\alpha\B(v(t)+z(t))+\beta \c(v(t)+z(t)),\ t\in(0,T), \\
			v(0)&= x,
		\end{aligned}
		\right.
	\end{equation} which is, for fixed $\omega\in\Omega$, a deterministic system. One can rewrite \eqref{2.4} in the mild form as 
	\begin{align}\label{2.5} 
		v(t)= &R(t)x-\alpha\int_{0}^{t}R(t-s)\B(v(s)+z(s))\d s +\beta\int_{0}^{t}R(t-s)\c(v(s)+z(s))\d s,	
	\end{align} then if $v\in\C([0,T];\L^2(\1))\cap\L^{2}(0,T;\H_0^1(\1))\cap\L^{2(\delta+1)}(0,T;\L^{2(\delta+1)}(\1))$ satisfies \eqref{2.5}, we say that it is a weak solution of \eqref{2.4}. For each dixed $\omega\in\Omega$, the weak form of \eqref{2.4} can be written as 
	\begin{align*}
		(v(t),\varphi)&=(x,\varphi)+\int_0^t\langle -\nu\A v(s)-\alpha\B(v(s)+z(s))+\beta \c(v(s)+z(s)),\varphi\rangle \d s,
	\end{align*}
	for all $t\in[0,T]$ and $\varphi\in\H_0^1(\mathcal{O})$. The next result provides the existence of weak solution to the system \eqref{2.4}. 
	
	\begin{theorem}\label{thm2.4}
		Let $x\in\L^2(\1)$ be the given initial data. Then there exists a weak solution $v\in\C([0,T];\L^2(\1))\cap\L^{2}(0,T;\H_0^1(\1))\cap\L^{2(\delta+1)}(0,T;\L^{2(\delta+1)}(\1))$ to the system \eqref{2.4}.
		
		For $\delta\in[1,2]$, the weak solution is unique for any $\nu,\alpha,\beta>0,\gamma\in(0,1)$ and for $2<\delta<\infty$, the weak solution is unique for $\beta\nu>2^{2(\delta-1)}\alpha^2$.
	\end{theorem}
	\begin{proof} The proof is divided into the following steps.
		\vskip 0.2 cm
		\noindent\textbf{Step (1):} \emph{Faedo-Galerkin approximation}. 
		Let $\H_n= \mathrm{span}\{e_1,\ldots,e_n\}$, where  $\{e_1,\ldots,e_n,\ldots\}$ be a complete orthogonal system of eigenfunctions of the Laplacian operator $\A=-\frac{\partial^2}{\partial x^2}$ and let $\Pi_n:\L^2(\1)\to\H_n$ be the orthogonal projection operator, that is, $x_n=\Pi_n x = \sum_{j=1}^{n}(x,e_j)$, for $x\in\L^2(\1)$. We define 
		\begin{align*}
			\B_n(u)= \Pi_n \B(\Pi_n u), \; \c_n(u) = \Pi_n\c_n(\Pi_nu),\; \G_n=\Pi_n\G\Pi_n.
		\end{align*} 
		Let us define $z_n(t):=\int_0^tR(t-s)\G_n\d\W(t)$. Then, it has been shown in Lemma 3.3, \cite{MTMSGBH} that $$z_n\to z\ \text{ in }\ \mathrm{C}([0,T]\times[0,1]).$$	Let us first consider the following finite dimensional system: 
		\begin{equation}\label{2.7}
			\left\{
			\begin{aligned}
				\frac{\d  v_n(t)}{\d t}&=-\nu\A v_n(t)-\alpha\B(v_n(t)+z_n(t))+\beta \c(v_n(t)+z_n(t)), \ t\in(0,T), \\
				v_n(0)&= x_n.
			\end{aligned}
			\right.
		\end{equation} 
		Since $\B(\cdot)$ and $c(\cdot)$ satisfy the locally Lipschitz conditions (see \eqref{BLL} and \eqref{CLL}),  the above system has a unique local solution $v_n\in\C([0,T^*];\H_n)$ for some $0<T^*<T$. Now, we show that the time $T^*$ can be extended to $T$. 
		\vskip 0.2 cm
		\noindent	\textbf{Step (2):} \emph{$\L^2$-energy estimate}.
		Let us take the inner product of \eqref{2.7} with $v_n$ to get 
		\begin{align}\label{2.8}
			\frac{1}{2}\frac{\d }{\d t}\|v_n\|_{\L^2}^2+\nu\|\partial_{\xi} v_n \|_{\L^2}^2 = -\alpha \langle \B(v_n+z_n),v_n\rangle +\beta \langle \c(v_n+z_n),v_n\rangle,
		\end{align} for a.e. $t\in[0,T]$. Now we estimate the right hand side of \eqref{2.8} term by term.
		Using integration by parts, Taylor's formula,  \eqref{2.1}, H\"older's and  Young's inequalities,  we estimate the term  $-\alpha \langle \B(v_n+z_n),v_n\rangle$ as
		\begin{align}\label{2.9} \nonumber
			-&\alpha \langle \B(v_n+z_n),v_n\rangle \nonumber
			\\&=  \frac{\alpha}{\delta+1}(v_n^{\delta+1},\partial_{\xi} v_n)\nonumber
			+\alpha(z_n(\theta_1v_n+(1-\theta_1)z_n)^{\delta},\partial_{\xi} v_n) \nonumber
			\\&= \alpha(z_n(\theta_1v_n+(1-\theta_1)z_n)^{\delta},\partial_{\xi} v_n) \nonumber
			\\& \leq \frac{\nu}{2}\|\partial_{\xi} v_n\|_{\L^2}^2+\frac{\alpha^2}{2\nu}\|z_n(\theta_1v_n+(1-\theta_1)z_n)^{\delta}\|_{\L^2}^2
			\nonumber
			\\& \leq \frac{\nu}{2}\|\partial_{\xi} v_n\|_{\L^2}^2+2^{2(\delta-1)}\frac{\alpha^2}{2\nu}\|z_n\|_{\L^{2(\delta+1)}}^{2}\|v_n\|_{\L^{2(\delta+1)}}^{2\delta} 
			+2^{2(\delta-1)}\frac{\alpha^2}{2\nu}\|z_n\|_{\L^{2(\delta+1)}}^{2(\delta+1)}
			\nonumber
			\\& \leq \frac{\nu}{2}\|\partial_{\xi} v_n\|_{\L^2}^2 +\frac{\beta}{4}\|v_n\|_{\L^{2(\delta+1)}}^{2(\delta+1)}+ \left[\frac{1}{\delta+1}\left(\frac{4\delta}{\beta(\delta+1)}\right)^{\delta}+ \frac{\alpha^22^{2(\delta-1)}}{2\nu}\right]\|z_n\|_{\L^{2(\delta+1)}}^{2(\delta+1)}.
		\end{align}	
		Let us now consider the second term of right hand side of \eqref{2.8}. We estimate using Taylor's formula, H\"older's and Young's inequalities as 
		\begin{align}\label{2.10}
			&\beta(\c(v_n+z_n),v_n)\nonumber\\&= \beta((1+\gamma)(v_n+z_n)^{\delta+1}-\gamma(v_n+z_n)-(v_n+z_n)^{2\delta+1},v_n)\nonumber\\&=\beta(1+\gamma)(v_n^{\delta+1},v_n)+\beta(1+\gamma)(\delta+1)(z_n(\theta_2 v_n+(1-\theta_2)z_n)^{\delta},v_n)\nonumber\\&\quad-\beta\gamma\|v_n\|_{\L^2}^2-\beta\gamma(v_n,z_n)-\beta\|v_n\|_{\L^{2(\delta+1)}}^{2(\delta+1)}-\beta(2\delta+1)(z_n(\theta_3 v_n+(1-\theta_3)z_n)^{2\delta},v_n)\nonumber\\&\leq\beta(1+\gamma)\|v_n\|_{\L^{2(\delta+1)}}^{\delta+1}\|v_n\|_{\L^2}+\beta(1+\gamma)(\delta+1)2^{\delta-1}\|z_n\|_{\L^{2(\delta+1)}}\|v_n\|_{\L^{2(\delta+1)}}^{\delta}\|v_n\|_{\L^2}\nonumber\\&\quad+\beta(1+\gamma)(\delta+1)2^{\delta-1}\|z_n\|_{\L^{2(\delta+1)}}^{(\delta+1)}\|v_n\|_{\L^2}-\beta\gamma\|v_n\|_{\L^2}^2-\beta\gamma\|v_n\|_{\L^2}\|z_n\|_{\L^2}\nonumber\\&\quad-\beta\|v_n\|_{\L^{2(\delta+1)}}^{2(\delta+1)}+\beta(2\delta+1)2^{2\delta-1}\|z_n\|_{\L^{2(\delta+1)}}\|v_n\|_{\L^{2(\delta+1)}}^{2\delta}\nonumber\\&\quad+\beta(2\delta+1)2^{2\delta-1}\|z_n\|_{\L^{2(\delta+1)}}^{2\delta+1}\|v_n\|_{\L^{2(\delta+1)}}\nonumber\\&\leq 2\beta(1+\gamma)^2\|v_n\|_{\L^2}^2+\frac{\beta}{\delta+1}\left(\frac{8\delta}{\delta+1}\right)^{\delta}[(1+\gamma)^2(\delta+1)^22^{2\delta-3}]^{\delta+1}\|z_n\|_{\L^{2(\delta+1)}}^{2(\delta+1)}\nonumber\\&\quad+\beta\|v_n\|_{\L^2}^2+\beta(1+\gamma)^2(\delta+1)^22^{2\delta-3}\|z_n\|_{\L^{2(\delta+1)}}^{2(\delta+1)}+\frac{\beta\gamma^2}{4}\|z_n\|_{\L^2}^2-\frac{\beta}{2}\|v_{n}\|_{\L^{2(\delta+1)}}^{2(\delta+1)}\nonumber\\&\quad+\frac{\beta}{\delta+1}\left(\frac{8\delta}{\delta+1}\right)^{\delta}[(2\delta+1)2^{2\delta-1}]^{\delta+1}\|z_n\|_{\L^{2(\delta+1)}}^{2(\delta+1)}\nonumber\\&\quad+\beta\left(\frac{2\delta+1}{2(\delta+1)}\right)\left(\frac{4}{\delta+1}\right)^{\frac{1}{2\delta+1}}[(2\delta+1)2^{2\delta-1}]^{\frac{2(\delta+1)}{2\delta+1}}\|z_n\|_{\L^{2(\delta+1)}}^{2(\delta+1)}.
		\end{align}
		Using the estimates \eqref{2.9}-\eqref{2.10} in \eqref{2.8}, we obtain
		\begin{align}\label{213}\nonumber
			&\frac{\d }{\d t}\|v_n\|_{\L^2}^2+\nu\|\partial_{\xi} v_n \|_{\L^2}^2 +\frac{\beta}{2}\|v_n\|_{\L^{2(\delta+1)}}^{2(\delta+1)} \\&\leq 
			C(\beta,\gamma)\|v_n\|_{\L^2}^2+C(\alpha,\beta,\gamma,\delta,\nu)\|z_n\|_{\L^{2(\delta+1)}}^{2(\delta+1)}\nonumber\\&\leq C(\beta,\gamma)+\frac{\beta}{4}\|v_n\|_{\L^{2(\delta+1)}}^{2(\delta+1)}+C(\alpha,\beta,\gamma,\delta,\nu)\|z_n\|_{\L^{2(\delta+1)}}^{2(\delta+1)},
		\end{align}
		for a.e. $t\in[0,T]$. Integrating the above inequality from $0$ to $t$, and using the fact that $\|x_n\|_{\L^2}\leq\|x\|_{\L^2}$, we get 
		\begin{align}\label{214}
			&	\|v_n(t)\|_{\L^2}^2+\nu\int_0^t\|\partial_{\xi}v_n(s)\|_{\L^2}^2\d s+\frac{\beta}{2}\int_0^t\|v_n(s)\|_{\L^{2(\delta+1)}}^{2(\delta+1)}\d s\nonumber\\&\leq C(\alpha,\beta,\gamma,\delta,\nu)\left(\|x\|_{\L^2}^2+\int_0^t\|z(s)\|_{\L^{2(\delta+1)}}^{2(\delta+1)}\d s\right),
		\end{align}
		for all $t\in[0,T]$. 
		\vskip 0.2 cm
		\noindent 
		\textbf{Step (3):} \emph{Weak convergence along a subsequence}. For the initial data $x \in \L^2(\1)$, using the estimate \eqref{214} and an application of Banach-Alaoglu theorem  yields the existence of a subsequence $\{v_{n_k}\}$ of $\{v_n\}$ such that (for convenience, we still denote the index $n_k$ by $n$): 
		\begin{equation}\label{LP1}
			\left\{
			\begin{aligned}
				v_n&\xrightarrow{w^*}v\ \text{ in }\ \L^{\infty}(0,T;\L^2(\1)),\\
				v_n&\xrightarrow{w}v\ \text{ in }\ \L^{2}(0,T;\H_0^1(\1)),\\
				v_n&\xrightarrow{w}v\ \text{ in }\ \L^{2(\delta+1)}(0,T;\L^{2(\delta+1)}(\1)),
			\end{aligned}\right.
		\end{equation}
		as $n\to\infty$. 
		\if
		Taking inner product of \eqref{2.7} with $\A v_n$ , we obtain
		\begin{align}\label{2.17} \nonumber
			\bigg\langle \frac{\d v_n}{\d t},\A v_n\bigg\rangle +\nu\langle\A v_n,\A v_n \rangle = -\alpha \langle \B(v_n+z_n),\A v_n \rangle+\beta \langle \c(v_n+z_n),\A v_n \rangle  \\
			\frac{1}{2}\frac{\d }{\d t}\|\partial_{\xi} v_n\|_{\L^2}^2+\nu \|\A v_n\|_{\L^2}^2 =-\alpha \langle \B(v_n+z_n),\A v_n \rangle+\beta \langle \c(v_n+z_n),\A v_n \rangle. 
		\end{align}Now we estimate the right hand side of \eqref{2.17}. 
		Using
		\begin{align*}
			|-\alpha &\langle \B(v_n+z_n),\A v_n\rangle|  \\&= |-\alpha \langle(v_n+z_n)^{\delta}\partial_{\xi}(v_n+z_n),\A v_n\rangle|\\&\leq \alpha 2^{\delta-1}\langle |v_n|^{\delta}+|z_n|^{\delta}|\partial_{\xi}(v_n+z_n)|,|\A v_n|\rangle \\& \leq \alpha 2^{\delta-1}\bigg\{\|v_n^{\delta}\partial_{\xi} v_n\|_{\L^2}+\|z_n^{\delta}\partial_{\xi} v_n\|_{\L^2}+\|v^{\delta}\partial_{\xi} z_n\|_{\L^2}+\|z_n^{\delta}\partial_{\xi} z_n\|_{\L^2}\bigg\}\|\A v_n\|_{\L^2}
			\\&\leq \frac{\nu}{4}\|\A v_n\|_{\L^2}^2+\alpha^2 2^{2(\delta-1)}\bigg\{\|v_n^{\delta}\partial_{\xi} v_n\|_{\L^2}^2+\|z_n^{\delta}\partial_{\xi} v_n\|_{\L^2}^2+\|v^{\delta}\partial_{\xi} z_n\|_{\L^2}^2+\|z_n^{\delta}\partial_{\xi} z_n\|_{\L^2}^2\bigg\}.
		\end{align*}We have the Sobolev embedding $\H^s (\Omega)\subset \C^k(\Omega)$, where $\Omega \subseteq \R^n$, whenever $s>\frac{n}{2}+k$, where $n$ is the dimension. For our case if we consider $n=1$ then we get $1>\frac{1}{2}+k$, and by the embedding we can say that $\H^1(\Omega) \subset \L^{\infty}(\Omega)$.
		Using the Sobolev embedding and H\"older's inequality in the above inequality, we obtain
		\begin{align*}
			|-\alpha \langle \B(v_n+z_n),\A v_n\rangle|  &\leq \frac{\nu}{4}\|\A v_n\|_{\L^2}^2+\alpha^2 2^{2(\delta-1)}\big\{ \|v_n\|_{\L^{\infty}}^{2\delta}\|\partial_{\xi} v_n\|_{\L^2}^2+\|z_n\|_{\L^{\infty}}^{2\delta}\|\partial_{\xi} v_n\|_{\L^2}^2\\&\quad+\|v_n\|_{\L^{\infty}}^{2\delta}\|\partial_{\xi} z_n\|_{\L^2}^2+\|z_n\|_{\L^{\infty}}^{2\delta}\|\partial_{\xi} z_n\|_{\L^2}^2\big\}.
		\end{align*}
		\noindent
		Using Cauchy-Schwarz inequality, Young's inequality (with $p=2,\; q=2$), and triangle inequality, to estimate the term $\beta(1+\gamma)( (v_n+z_n)^{\delta+1},\A v_n)$ as
		\begin{align*}
			\beta(1+\gamma)( (v_n+z_n)^{\delta+1},\A v_n) \leq \frac{\nu}{4}\|\A v_n\|_{\L^2}^2+\frac{\beta^2(1+\gamma)^2}{\nu}\|v_n\|_{\L^{2(\delta+1)}}^{2(\delta+1)}+\frac{\beta^2(1+\gamma)^2}{\nu}\|z_n\|_{\L^{2(\delta+1)}}^{2(\delta+1)}.
		\end{align*}
		Similarly we can estimate the term $\gamma\beta((v_n+z_n),\A v_n)$ as
		\begin{align*}
			|-\gamma\beta((v_n+z_n),\A v_n)| \leq \frac{\nu}{4}\|\A v_n\|_{\L^2}^2+\frac{\beta^2\gamma^2}{\nu}\|v_n\|_{\L^2}^2+\frac{\beta^2\gamma^2}{\nu}\|z_n\|_{\L^2}^2.
		\end{align*} Using Taylor's formula ($\theta\in[0,1]$), we obtain
		\begin{align*}
			-\beta((v_n+z_n)^{2\delta+1},\A v_n)&= -\beta(v_n^{2\delta+1}+(2\delta+1)z_n(\theta v_n+(1-\theta)z_n)^{2\delta},\A v_n) \\& \leq -\beta(2\delta+1)\|v_n^{\delta}\partial_{\xi}v_n\|_{\L^2}^2-\beta(2\delta+1)(z_n(\theta v_n+(1-\theta)\z_n)^{2\delta},\A v_n).
		\end{align*} We take the second term of above inequality and estimate as
		\begin{align*}
			|-\beta&(2\delta+1)(z_n(\theta v_n+(1-\theta)z_n)^{2\delta},\A v_n)| \\&\leq \beta(2\delta+1)2^{2\delta-1}(|z_n|(|v_n|^{2\delta}+|z_n|^{2\delta}),\A v_n) \\& \leq \frac{\nu}{4}\|\A v_n\|_{\L^2}^2+\frac{2^{2\delta-1}(2\delta+1)^2}{\nu}\|z_n\|_{\L^{2(2\delta+1)}}^2\|v_n\|_{\L^{2(2\delta+1)}}^{4\delta}+\frac{2^{2\delta-1}(2\delta+1)^2}{\nu}\|v_n\|_{\L^{2(2\delta+1)}}^{4\delta+2}
			\\& 
			\leq \frac{\nu}{4}\|\A v_n\|_{\L^2}^2+\frac{2^{2\delta-1}(2\delta+1)^2}{\nu}\bigg\{2\|z_n\|_{\L^{2(2\delta+1)}}^2+\|v_n\|_{\L^{2(2\delta+1)}}\bigg\}
		\end{align*} By combining above two estimates we obtain
		\begin{align*}
			-\beta((v_n+z_n)^{2\delta+1},\A v_n) &\leq -\beta(2\delta+1)\|v_n^{\delta}\partial_{\xi}v_n\|_{\L^2}^2+\frac{\nu}{4}\|\A v_n\|_{\L^2}^2+\frac{2^{2\delta-1}(2\delta+1)^2}{\nu}\\&\quad \times\bigg\{2\|z_n\|_{\L^{2(2\delta+1)}}^2+\|v_n\|_{\L^{2(2\delta+1)}}^{2(2\delta+1)}\bigg\}.
		\end{align*} Using the above estimates in \eqref{2.17}, we obtain
		\begin{align}\label{2.18} \nonumber
			\frac{1}{2}\frac{\d}{\d t}&\|\partial_{\xi} v_n\|_{\L^2}^2+\frac{\nu}{4} \|\A v_n\|_{\L^2}^2+\beta(2\delta+1)\|v_n^{\delta}\partial_{\xi}v_n\|_{\L^2}^2 \nonumber\\&\leq \alpha^22^{2(\delta-1)}\big\{ \|v_n\|_{\L^{\infty}}^{2\delta}\|\partial_{\xi} v_n\|_{\L^2}^2+\|z_n\|_{\L^{\infty}}^{2\delta}\|\partial_{\xi} v_n\|_{\L^2}^2+\|v_n\|_{\L^{\infty}}^{2\delta}\|\partial_{\xi} z_n\|_{\L^2}^2 \nonumber\\&\quad+\|z_n\|_{\L^{\infty}}^{2\delta}\|\partial_{\xi} z_n\|_{\L^2}^2\big\}+\frac{\beta^2(1+\gamma)^2}{\nu}\|v_n\|_{\L^{2(\delta+1)}}^{2(\delta+1)}+\frac{\beta^2(1+\gamma)^2}{\nu}\|z_n\|_{\L^{2(\delta+1)}}^{2(\delta+1)}\nonumber\\& \quad +\frac{\beta^2\gamma^2}{\nu}\|v_n\|_{\L^2}^2+\frac{\beta^2\gamma^2}{\nu}\|z_n\|_{\L^2}^2+\frac{2^{2\delta-1}(2\delta+1)^2}{\nu} \big\{2\|z_n\|_{\L^{2(2\delta+1)}}^2+\|v_n\|_{\L^{2(2\delta+1)}}^{2(2\delta+1)}\big\}.
		\end{align}
		Using Gronwall's inequality and $\L^2$-estimate, we obtain 
		\begin{align*}
			\|\partial_{\xi} v_n\|_{\L^2}^2 & \leq \exp\bigg(\alpha^22^{2(\delta-1)}\int_{0}^{t}(\|v_n\|_{\L^{\infty}}^{2\delta}+\|z_n\|_{\L^{\infty}}^{2\delta})\d s\bigg)\bigg[\|\partial_{\xi} v_n(0)\|_{\L^2}^2\\&\quad+\int_{0}^{t}\big(\alpha^22^{2(\delta-1)}\big\{\|v_n\|_{\L^{\infty}}^{2\delta}\|\partial_{\xi} z_n\|_{\L^2}^2+\|z_n\|_{\L^{\infty}}^{2\delta}\|\partial_{\xi} z_n\|_{\L^2}^2\big\}+\frac{\beta^2(1+\gamma)^2}{\nu}\|v_n\|_{\L^{2(\delta+1)}}^{2(\delta+1)}\\& \quad+\frac{\beta^2(1+\gamma)^2}{\nu}\|z_n\|_{\L^{2(\delta+1)}}^{2(\delta+1)}  +\frac{\beta^2\gamma^2}{\nu}\|v_n\|_{\L^2}^2+\frac{\beta^2\gamma^2}{\nu}\|z_n\|_{\L^2}^2+\frac{2^{2\delta-1}(2\delta+1)^2}{\nu} \\&\quad\times\big\{2\|z_n\|_{\L^{2(2\delta+1)}}^2+\|v_n\|_{\L^{2(2\delta+1)}}^{2(2\delta+1)}\big\}\big)\bigg]\\& <\infty,
		\end{align*}provided $\delta=1$ and we are restricted our self in one dimensional so that the Sobolev embedding be applicable. Using the above inequality in \eqref{2.18} and integrating the resultant from $0$ to $t$, we find
		\begin{align*}
			\|\partial_{\xi} v_n(t)\|_{\L^2}^2+\frac{\nu}{2}\int_{0}^{t}\|\A v_n(s)\|_{\L^2}^2\d s +2\beta(2\delta+1)\int_{0}^{t}\|v_n^{\delta}\partial_{\xi}v_n(s)\|_{\L^2}^2 \d s <\infty,
		\end{align*}for any $t\in[0,T]$.\\ \fi
		\vskip 0.2 cm
		\noindent
		\textbf{Step (4):} \emph{Estimate for the derivative}. For any $\phi \in \L^{2(\delta+1)}(0,T;\H_0^1(\1))$, taking  the inner with $\phi$ to the first equation of system \eqref{2.7}, and using integration by parts, Cauchy-Schwarz inequality, Taylor's formula and H\"older's inequality, we obtain 
		\begin{align*}
			\int_{0}^{T}&\bigg|\bigg\langle \frac{\partial v_n}{\partial t}, \phi\bigg\rangle \bigg|\d t  \\& \leq \nu \int_{0}^{T}|\langle \A v_n, \phi\rangle |\d t+\alpha\int_{0}^{T}|\langle  \B_n(v_n+z_n), \phi\rangle |\d t+\beta \int_{0}^{T}|\langle \c(v_n+z_n), \phi\rangle |\d t \\& \leq  \nu\int_{0}^{T}\|v_n\|_{\H_0^1}\|\phi\|_{\H_0^1}\d t +\frac{\alpha}{\delta+1}\int_{0}^{T}\|v_n+z_n\|_{\L^{2(\delta+1)}}^{\delta+1}\|\phi\|_{\H_0^1}\d t \\& \quad +\beta(1+\gamma) \int_{0}^{T}\|v_n+z_n\|_{\L^{2(\delta+1)}}^{\delta+1}\|\phi\|_{\L^2}\d t+ \beta\gamma \int_{0}^{T}\|v_n+z_n\|_{\L^2}\|\phi\|_{\L^2}\d t \\& \quad +\beta\int_{0}^{T}\|v_n+z_n\|_{\L^{2(\delta+1)}}^{2\delta+1}\|\phi\|_{\L^{2(\delta+1)}}\d t
			\\& \leq \left[\nu\bigg(\int_{0}^{T}\|v_n\|_{\H_0^1}^2\d t\bigg)^{\frac{1}{2}}+\frac{\alpha}{\delta+1}\bigg(\int_{0}^{T}\|v_n+z_n\|_{\L^{2(\delta+1)}}^{2(\delta+1)}\d t\bigg)^{\frac{1}{2}}\right]\bigg(\int_{0}^{T}\|\phi\|_{\H_0^1}^2\d t \bigg)^{\frac{1}{2}} \\& \quad +\beta\left[(1+\gamma)\bigg(\int_{0}^{T}\|v_n+z_n\|_{\L^{2(\delta+1)}}^{2(\delta+1)}\d t\bigg)^{\frac{1}{2}}+\gamma \bigg(\int_{0}^{T}\|v_n+z_n\|_{\L^2}^2\d t\bigg)^{\frac{1}{2}}\right]\bigg(\int_{0}^{T}\|\phi\|_{\L^2}^2\d t \bigg)^{\frac{1}{2}} \\& \quad +\beta\bigg(\int_{0}^{T}\|v_n+z_n\|_{\L^{2(\delta+1)}}^{2(\delta+1)}\d t \bigg)^{\frac{2\delta+1}{2(\delta+1)}}\bigg(\int_{0}^{T}\|\phi\|_{\L^{2(\delta+1)}}^{2(\delta+1)}\d t \bigg)^{\frac{1}{2(\delta+1)}},
		\end{align*}
		which is finite.	Since it is true for any $\phi \in \L^{2(\delta+1)}(0,T;\H_0^1(\1))$, we get  \begin{align*}\frac{\d v_n}{\d t } \xrightarrow{w} \frac{\d v}{\d t} \ \text{ in }\  \ \L^{\frac{2(\delta+1)}{2\delta+1}}(0,T;\H^{-1}(\1)).
		\end{align*} 
		Since the embedding $\H_0^1(\1)\subset \L^2(\1)$ is compact, using Aubin-Lions compactness lemma (Theorem 1, \cite{JM}), we deduce  the following strong convergence (along a subsequence):
		\begin{align}\label{SC}
			v_{n} \to v \ \text{ in }\ \L^2(0,T;\L^2(\1)), \ \text{ as } \ n\to \infty. 
		\end{align}
		\vskip 0.2 cm
		\noindent 
		\textbf{Step (5):} \emph{Passing to limit  in $\B_n(\cdot)$ and $\c_n(\cdot)$}.
		Now our aim is to pass the limit in the terms $\B_n(\cdot)$ and $\c_n(\cdot)$ by using $\L^2$-estimate.
		Let us choose a smooth function $\phi \in \C([0,T];\C^1([0,1]))$, and using Taylor's formula and H\"older's inequality to obtain 
		\begin{align}\label{2.19} \nonumber
			\int_{0}^{T}\big|&\langle \B_n(v_n+z_n)-\B(v+z),\phi \rangle\big|\d t \nonumber\\&= \int_{0}^{T}\big| \langle (v_n+z_n)^{\delta}\partial_{\xi}(v_n+z_n)-(v+z)^{\delta}\partial_{\xi} (v+z), \phi \rangle\big|\d t \nonumber\\& \leq\frac{1}{\delta+1} \int_{0}^{T}\big|\langle (v_n+z_n)^{\delta+1}-(v+z)^{\delta+1},\partial_{\xi} \phi \rangle\big|\d t 
			\nonumber \\& \leq  \int_{0}^{T}\|((v_n+z_n)-(v+z))(\theta(v_n+z_n)+(1-\theta)(v+z))^{\delta}\|_{\L^1}\|\partial_{\xi}\phi\|_{\L^{\infty}} \d t
			\nonumber \\& \leq \sup_{t\in[0,T]}\|\partial_{\xi}\phi\|_{\L^{\infty}} \int_{0}^{T}\big(\|v_n\|_{\L^{2\delta}}+\|z_n\|_{\L^{2\delta}}+\|v\|_{\L^{2\delta}}+\|z\|_{\L^{2\delta}}\big)^{\delta}\big(\|v_n-v\|_{\L^2}+\|z_n-z\|_{\L^2}\big)\d t \nonumber \\& \leq \sup_{t\in[0,T]}\|\partial_{\xi}\phi\|_{\L^{\infty}} \bigg(\int_{0}^{T}\big(\|v_n\|_{\L^{2\delta}}+\|z_n\|_{\L^{2\delta}}+\|v\|_{\L^{2\delta}}+\|z\|_{\L^{2\delta}}\big)^{2\delta}\d t \bigg)^\frac{1}{2} \nonumber\\& \quad\times \bigg(\int_{0}^{T}\big(\|v_n-v\|_{\L^2}+\|z_n-z\|_{\L^2}\big)^2\d t\bigg)^{\frac{1}{2}}\nonumber \\& \to 0, \ \text{ as }\  n\to \infty,
		\end{align} since we have the convergence $z_n \to z \in \C([0,T]\times [0,1])$. By using the density of $\C([0,T];\C^1([0,1]))$ in $\mathrm{L}^2(0,T;\L^{2}(\mathcal{O}))$, we obtain \begin{align}\label{2p18}\B(v_n+z_n)\xrightarrow{w}\B(v+z)\ \text{ in }\ \mathrm{L}^2(0,T;\L^2(\mathcal{O}))\ \text{ as }\ n\to\infty.\end{align} 
		Again choosing a function $\phi \in \C([0,T];\C([0,1]))$, using Taylor's formula, H\"older's inequality and \eqref{SC}, we obtain
		\begin{align}\label{CLP}\nonumber
			&\int_{0}^{T}\big|(\c_n(v_n+z_n)-\c(v+z),\phi)\big|\d t \nonumber\\& \leq (1+\gamma)\int_{0}^{T}\big|((v_n+z_n)^{\delta+1}-(v+z)^{\delta+1},\phi)\big|\d t+\gamma\int_{0}^{T}\big| ((v_n+z_n)-(v+z),\phi)\big|\d t 
			\nonumber\\& \quad + \int_{0}^{T}\big|((v_n+z_n)^{2\delta+1}-(v+z)^{2\delta+1},\phi)\big|\d t 
			\nonumber\\& \leq (1+\gamma)(\delta+1)\sup_{t\in[0,T]}\|\phi\|_{\L^{\infty}}\nonumber\\&\qquad\times\int_{0}^{T}(\|v_n\|_{\L^{2\delta}}+\|z_n\|_{\L^{2\delta}}+\|v\|_{\L^{2\delta}}+\|z\|_{\L^{2\delta}})^{\delta}
			(\|v_n-v\|_{\L^2}+\|z_n-z\|_{\L^2})\d t \nonumber\\&\quad+\gamma \bigg[\bigg(\int_{0}^{T}\|v_n-v\|_{\L^2}^2\d t\bigg)^{\frac{1}{2}}+\bigg(\int_{0}^{T}\|z_n-z\|_{\L^2}^2\d t\bigg)^{\frac{1}{2}}\bigg]
			\bigg(\int_{0}^{T}\|\phi\|_{\L^2}^2\d t \bigg)^{\frac{1}{2}}\nonumber\\&\quad+(2\delta+1)\sup_{t\in[0,T]}\|\phi\|_{\L^{\infty}}\bigg[ \int_{0}^{T}\|v_n-v\|_{\L^2}^{\frac{1}{\delta}}\|v_n-v\|_{\L^{2(\delta+1)}}^{\frac{\delta-1}{\delta}}\big(\|v_n+z_n\|_{\L^{2(\delta+1)}}^{2\delta}+\|v+z\|_{\L^{2(\delta+1)}}^{2\delta}\big)\d t
			\nonumber\\&\quad+\bigg(\int_{0}^{T}\|z_n-z\|_{\L^{\delta+1}}^{\delta+1}\d t\bigg)^{\frac{1}{\delta+1}}\bigg(\int_{0}^{T}\big(\|v_n+z_n\|_{\L^{2(\delta+1)}}^{2\delta}+\|v+z\|_{\L^{2(\delta+1)}}^{2\delta}\big)\d t\bigg)^{\frac{\delta}{\delta+1}}\bigg]\nonumber\\&\leq (1+\gamma)(\delta+1)\sup_{t\in[0,T]}\|\phi\|_{\L^{\infty}}\nonumber\\&\qquad\times\left(\int_{0}^{T}(\|v_n\|_{\L^{2\delta}}+\|z_n\|_{\L^{2\delta}}+\|v\|_{\L^{2\delta}}+\|z\|_{\L^{2\delta}})^{2\delta}\d t\right)^{1/2}
			\nonumber\\&\qquad\times\left(\int_0^T(\|v_n-v\|_{\L^2}+\|z_n-z\|_{\L^2})^2\d t\right)^{1/2} \nonumber\\&\quad+\gamma \bigg[\bigg(\int_{0}^{T}\|v_n-v\|_{\L^2}^2\d t\bigg)^{\frac{1}{2}}+\bigg(\int_{0}^{T}\|z_n-z\|_{\L^2}^2\d t\bigg)^{\frac{1}{2}}\bigg]
			\bigg(\int_{0}^{T}\|\phi\|_{\L^2}^2\d t \bigg)^{\frac{1}{2}}\nonumber\\&\quad+(2\delta+1)\sup_{t\in[0,T]}\|\phi\|_{\L^{\infty}}\bigg[\bigg(\int_{0}^{T}\|v_n-v\|_{\L^2}^2\d t\bigg)^{\frac{1}{2\delta}}\bigg(\int_{0}^{T}\|v_n-v\|_{\L^{2(\delta+1)}}^{2(\delta+1)}\bigg)^{\frac{\delta-1}{2\delta(\delta-1)}}
			\nonumber\\&\qquad\times\bigg(\int_{0}^{T}\big(\|v_n\|_{\L^{2(\delta+1)}}^{2(\delta+1)}+\|v\|_{\L^{2(\delta+1)}}^{2(\delta+1)}\big)\d t \bigg)^{\frac{\delta}{\delta+1}}\nonumber\\&\quad +T^{\frac{2\delta^2}{(2\delta-1)(\delta+1)}}\sup_{t\in[0,T]}\big(\|z_n\|_{\L^{2(\delta+1)}}^{2\delta}+\|z\|_{\L^{2(\delta+1)}}^{2\delta}\big)\bigg(\int_{0}^{T}\|v_n-v\|_{\L^2}^2\d t\bigg)^{\frac{1}{2\delta}}\nonumber\\&\qquad\times\bigg(\int_{0}^{T}\|v_n-v\|_{\L^{2(\delta+1)}}^{2(\delta+1)}\d t\bigg)^{\frac{\delta-1}{2\delta(\delta+1)}}
			\nonumber\\&\quad+\bigg(\int_{0}^{T}\|z_n-z\|_{\L^{\delta+1}}^{\delta+1}\d t\bigg)^{\frac{1}{\delta+1}}\bigg(\int_{0}^{T}\big(\|v_n+z_n\|_{\L^{2(\delta+1)}}^{2\delta}+\|v+z\|_{\L^{2(\delta+1)}}^{2\delta}\big)\d t\bigg)^{\frac{\delta}{\delta+1}}\bigg]\nonumber\\&\to 0\ \text{ as } n\to\infty.
		\end{align}
		By using the density of $\C([0,T];\C([0,1]))$ in $\mathrm{L}^2(0,T;\L^{2}(\mathcal{O}))$, we get  \begin{align}\label{2p20}
			\c_n(v_n+z_n)\xrightarrow{w}\c_n(v+z)\ \text{ in }\ \mathrm{L}^2(0,T;\L^2(\mathcal{O}))\ \text{ as }\ n\to\infty.
		\end{align}
		Thus one can pass to limit in the equation \eqref{2.7}, and use the convergences \eqref{LP1}, \eqref{2p18} and \eqref{2p20} to deduce that $v(\cdot)$ is a weak solution. 
		
		The steps (2), (3)  and (4) yield $v \in W^{1,\frac{2(\delta+1)}{2\delta+1}}(0,T;\H^{-1}(\1))$ and an application of Theorem 2, page 302, \cite{LCE} implies  $v \in \C([0,T];\H^{-1}(\1))$. Moreover, from Proposition 1.7.1, \cite{PCAM}, we have $v \in \C_w([0,T];\L^{2}(\1))$, where $\C_w([0,T];\L^{2}(\1))$ is the space of functions $u:[0,T]\to\L^{2}(\1)$ which are weakly continuous.
		\vskip 0.2 cm
	\noindent	\textbf{Step (6):} \emph{Energy equality}. 
		We need to show that $v \in \C([0,T];\L^{2}(\1))$ and it satisfies energy equality. Let us define $\mathcal{D}_T= \{\phi:\phi\in\C_0^{\infty}(\1_T)\}$, where $\1_T = \1\times[0,T)$. Note that $u(\cdot,T)=0$ for all $u\in\mathcal{D}_T$ and $\mathcal{D}_T$ is dense in $\L^q(0,T;\H^1(\1))$ (Lemma 2.6, \cite{GGP} for $q=2$). For $u \in\L^q(0,T;\X)$, $1\leq q<\infty$ and $T>h>0$, the mollifier $u_h$ (in the sense of Friederichs) of $u$ is defined by \begin{align*}
			u_h(t)=\int_{0}^{T}j_h(t-s)u(s)\d s,
		\end{align*}	where $j_h(\cdot)$ is an infinitely differentiable function having support in $(-h,h),$ which is even and positive, such that $\int_{-\infty}^{+\infty}j_h(s)\d s=1$. From Lemma 2.5, \cite{GGP}, we infer that for $u\in \L^q(0,T;\X),\; 1\leq q <\infty$, we have $u_h \in \C^k([0,T];\X)$ for all $k\geq 0$. Moreover 
		\begin{align}\label{2.21}
			\lim_{h\to0}\|u_h-u\|_{\L^q(0,T;\X)}=0.
		\end{align}Finally, if ${u_k}\in\L^q(0,T;\X)$ converges to $u$ in the norm of $\L^q(0,T;\X),$ then 
		\begin{align}\label{2.22}
			\lim_{k\to \infty}\|(u_k)_h-u_h\|_{\L^q(0,T;\X)}=0.
		\end{align}
		The weak solution of \eqref{2.4} can be written as 
		\begin{align}\label{2p21}
			\int_s^t\left\{\left(v,\frac{\partial\varphi}{\partial t}\right)+\langle\nu\A v+\alpha\B(v+z)-\beta c(v+z),\varphi\rangle\right\}\d\tau =(v(t),\varphi(t))-(v(s),\varphi(s)),
		\end{align}
		for all $s\in[0,t]$, $t<T$ and all $\varphi\in\mathcal{D}_T$.
		Let $\{v_k\}\in\mathcal{D}_T$ be a sequence converging to $v \in \L^{2\delta}(0,T;\H^1(\1))$. Let us choose $\varphi=(v_k)_h=:v_{k,h}$  in \eqref{2p21} ($s=0$), where $(\cdot)_{h}$ is mollification operator defined above,  for $0\leq t<T$, we obtain
		\begin{align}\label{1a}\nonumber
			&	\int_{0}^{t}\bigg(v,\frac{\partial v_{k,h}}{\partial t}\bigg)+\nu (\partial_{\xi} v,\partial_{\xi} v_{k,h})\bigg)\d s\nonumber\\&=(v(t),v_{k,h}(t))-(v(0),v_{k,h}(0))+\int_{0}^{t}\bigg(-\alpha \langle (v+z)^{\delta}\partial_{\xi} (v+z),v_{k,h} \rangle \nonumber\\& \quad+\beta(1+\gamma)\langle (v+z)^{\delta+1},v_{k,h}\rangle-\beta\gamma\langle (v+z),v_{k,h}\rangle-\beta\langle(v+z)^{2\delta+1},v_{k,h}\rangle\bigg)\d s. 	
		\end{align}
		Now with the help of \eqref{2.22}, one can take the limit as $k\to \infty$. To pass the limit in the first term of the right hand side of \eqref{1a} we use integration by parts and H\"older's inequality to  find 
		\begin{align*}
			\bigg|\int_{0}^{t}(( (v+z)^{\delta}&\partial_{\xi}(v+z),v_{k,h})-( (v+z)^{\delta}\partial_{\xi}(v+z),v_h)) \d s \bigg|  \\& \leq \frac{1}{\delta+1} \int_{0}^{t}\|(v+z)\|_{\L^{2(\delta+1)}}^{\delta+1}\|\partial_{\xi}(v_{k,h}-v_h)\|_{\L^2}\d s \\&\leq \frac{1}{\delta+1}\|v_{k,h}-v_h\|_{\L^2(0,T;\H^1)}\bigg(\int_{0}^{t}\|(v+z)\|_{\L^{2(\delta+1)}}^{2(\delta+1)}\d s\bigg)^{\frac{1}{2}} 
			\\& \to 0 \text{ as } k \to\infty.
		\end{align*} Using H\"older's inequality, 
		the second term in the right hand side of \eqref{1a} can be estimated as 
		\begin{align*}
			\bigg|\int_{0}^{t}(( (v+z)^{\delta+1},v_{k,h}&) -( (v+z)^{\delta+1},v_h))\d s \bigg|\\&\leq \int_{0}^{t} \|v+z\|^{\delta+1}_{\L^{2(\delta+1)}}\|v_{k,h}-v_h\|_{\L^2}\d s \\&\leq \|v_{k,h}-v_h\|_{\L^2(0,T;\L^2)}\bigg(\int_{0}^{t}\|(v+z)\|_{\L^{2(\delta+1)}}^{2(\delta+1)}\d s\bigg)^{\frac{1}{2}} \\& \to 0 \text{ as } k \to\infty.
		\end{align*}For the third term in the right hand side of \eqref{1a}, using H\"older's inequality, we have  
		\begin{align*}
			\bigg|\int_{0}^{t}(( v+z,v_{k,h}) -( v+z,v_h))\d s\bigg|  &\leq C\|v_{k,h}-v_h\|_{\L^2(0,T;\L^2)}\bigg(\int_{0}^{t}\|v+z\|_{\L^2}^2\d s\bigg)^{\frac{1}{2}}
			\\& \to 0 \text{ as } k \to\infty.
		\end{align*}The final term in the right hand side of \eqref{1a} can be estimated using the Cauchy-Schwarz inequality, the embedding of $\H^1(\1)\subset \L^{\infty}(\1)$, interpolation inequality, H\"older's inequality 
		and \eqref{2.22} as
		\begin{align*}
			\bigg|\int_{0}^{t}&( ((v+z)^{2\delta+1},v_{k,h})- ((v+z)^{2\delta+1},v_h))\d s \bigg| \\& \leq  \int_{0}^{t}\|v+z\|_{\L^{2\delta+1}}^{2\delta+1}\|v_{k,h}-v_h\|_{\L^{\infty}}\d s \\&\leq\int_0^t\|v+z\|_{\L^2}^{\frac{1}{\delta}}\|v+z\|_{\L^{2(\delta+1)}}^{\frac{(\delta+1)(2\delta-1)}{\delta}}\|v_{k,h}-v_h\|_{\H^1}\d s\\& \leq\sup_{s\in[0,t]} \|v+z\|_{\L^2}^{\frac{1}{\delta}}\left(\int_0^t\|v+z\|_{\L^{2(\delta+1)}}^{2(\delta+1)}\d s\right)^{\frac{2\delta-1}{2\delta}}\left(\int_0^t\|v_{k,h}-v_h\|_{\H^1}^{2\delta}\d s\right)^{\frac{1}{2\delta}}\\& \to 0,\; \text{ as } k\to \infty.
		\end{align*}	Passing limit $k\to\infty$ in \eqref{1a}, we obtain
		\begin{align}\label{1b}\nonumber
			-\int_{0}^{t}&\bigg(v,\frac{\partial v_h}{\partial t}\bigg)\d s+\nu\int_{0}^{t} (\partial_{\xi} v,\partial_{\xi}v_h)\d s+(v(t),v_h(t))\nonumber\\&=(v(0),v_h(0))+\int_{0}^{t}\bigg(-\alpha \langle (v+z)^{\delta}\partial_{\xi} (v+z),v_h \rangle +\beta(1+\gamma)\langle (v+z)^{\delta+1},v_h\rangle\nonumber\\& \quad-\beta\gamma\langle (v+z),v_h\rangle-\beta\langle(v+z)^{2\delta+1},v_h\rangle\bigg)\d s. 	
		\end{align}Since the kernel $j_h(s)$ in the definition of mollifier even in $(-h,h)$, we get 
		\begin{align*}
			\int_{0}^{t}\bigg(v,\frac{\partial v_{k,h}}{\partial t}\bigg)\d t =\int_{0}^{t}\int_{0}^{t}\frac{\d j_h(t-t')}{\d t}(v(t),v(t'))\d t\d t'=0.
		\end{align*}By using  \eqref{2.21} and similar arguments as above yield 
		\begin{align*}
			\lim_{h\to 0}\int_{0}^{t}(\partial_{\xi} v,\partial_{\xi} v_h)\d s&= \int_{0}^{t}(\partial_{\xi} v,\partial_{\xi} v)\d s,
			\\ \lim_{h\to 0}\int_{0}^{t}\langle\B(v+z),v_h\rangle\d s&=\int_{0}^{t}\langle\B(v+z),v\rangle\d s,
			\\ \lim_{h\to 0}\int_{0}^{t}\langle\c(v+z),v_h\rangle\d s &=\int_{0}^{t}\langle\c(v+z),v\rangle\d s.
		\end{align*} As $h\to 0$ in \eqref{1b}, we obtain 
		\begin{align*}
			&\|v(t)\|_{\L^2}^2+\nu\int_{0}^{t}\|\partial_{\xi}v(s)\|_{\L^2}^2\d s+\beta\gamma\int_{0}^{t}\|v(s)+z(s)\|_{\L^2}^2\d s+\beta\int_{0}^{t}\|v(s)+z(s)\|_{\L^{2(\delta+1)}}^{2(\delta+1)}\d s  \\&
			=\|v(0)\|_{\L^2}^2+\beta(1+\gamma)\int_{0}^{t}((v(s)+z(s))^{\delta+1},v(s))\d s+\frac{\alpha }{1+\delta}\int_{0}^{t}((v(s)+z(s))^{\delta+1},\partial_{\xi} v(s))\d s\\&\quad+\beta\gamma \int_{0}^{t}(v(s)+z(s),z(s))\d s+\beta\int_{0}^{t}((v(s)+z(s))^{2\delta+1},z(s))\d s,
		\end{align*} and hence the energy equality is satisfied. 	We know that every weak solution is $\L^2$-weakly continuous in time, and all weak solutions satisfy the energy equality, and so all weak solutions belongs to $\C([0,T];\L^2(\mathcal{O}))$. 
		\vskip 0.2 cm
	\noindent	\textbf{Step (7):} \emph{Uniqueness}. Let $v_1$ and $v_2$ be the two weak solutions of the system \eqref{2.7} and the initial data $x$. Note that $w=v_1-v_2$ satisfies:
		\begin{equation}\label{U1}
			\left\{
			\begin{aligned}
				\frac{\partial w}{\partial t} - \nu \A w&= -\alpha\big[(v_1+z)^{\delta}\partial_{\xi}(v_1+z)^{\delta}-(v_2+z)^{\delta}\partial_{\xi}(v_2+z)^{\delta}\big]+\beta\big[(v_1+z)\\&\quad\times(1-(v_1+z)^{\delta})((v_1+z)^{\delta}-\gamma)-(v_2+z)(1-(v_2+z)^{\delta})((v_2+z)^{\delta}-\gamma)], \\
				w(0)&= 0,
			\end{aligned}
			\right.
		\end{equation}in $\H^{-1}(\1)$, for a.e. $t\in[0,T]$. Taking the inner product of first equation of above system \eqref{U1} with $w$, we obtain
		\begin{align}\label{U2}\nonumber
			\frac{1}{2}\frac{\d}{\d t} \|w\|_{\L^2}^2+\nu \|\partial_{\xi}w\|_{\L^2}^2&= \frac{\alpha}{\delta+1}((v_1+z)^{\delta+1}-(v_2+z)^{\delta+1},\partial_{\xi} w)+\beta ((v_1+z)(1-(v_1+z)^{\delta})\\&\quad\times((v_1+z)^{\delta}-\gamma)-(v_2+z)(1-(v_2+z)^{\delta})((v_2+z)^{\delta}-\gamma),w).
		\end{align}For the term $\frac{\alpha}{\delta+1}((v_1+z)^{\delta+1}-(v_2+z)^{\delta+1},\partial_{\xi} w)$, we use  Taylor's formula, H\"older's inequality, interpolation inequality, Gagliardo-Nirenberg inequality and Young's inequality to estimate it as 
		\begin{align}\label{U3}\nonumber
			&\frac{\alpha}{\delta+1}((v_1+z)^{\delta+1}-(v_2+z)^{\delta+1},\partial_{\xi} w) \nonumber\\&= \alpha (w(\theta(v_1+z)+(1-\theta)(v_2+z))^{\delta}, \partial_{\xi} w) \nonumber\\& \leq \alpha2^{\delta-1}(\|v_1+z\|_{\L^{2(\delta+1)}}^{\delta}+\|v_2+z\|_{\L^{2(\delta+1)}}^{\delta})\|w\|_{\L^{2(\delta+1)}}\|\partial_{\xi}w\|_{\L^2}\nonumber\\& \leq  \alpha2^{\delta-1}(\|v_1+z\|_{\L^{2(\delta+1)}}^{\delta}+\|v_2+z\|_{\L^{2(\delta+1)}}^{\delta})\|w\|_{\L^2}^{\frac{\delta+2}{2(\delta+1)}}\|\partial_{\xi}w\|_{\L^2}^{\frac{3\delta+2}{2(\delta+1)}}\nonumber\\& \leq 
			\frac{\nu}{2}\|\partial_{\xi}w\|_{\L^2}^2+c2^{\frac{4\delta^2+3\delta-2}{\delta+2}}\alpha^{\frac{4(\delta+1)}{\delta+2}}\bigg(\frac{\delta+2}{4(\delta+1)}\bigg)\bigg(\frac{3\delta+2}{2\nu(\delta+1)}\bigg)^{\frac{3\delta+2}{\delta+2}}\nonumber\\&\quad\times\big(\|v_1+z\|_{\L^{2(\delta+1)}}^{\frac{4\delta(\delta+1)}{\delta+2}}+\|v_2+z\|_{\L^{2(\delta+1)}}^{\frac{4\delta(\delta+1)}{\delta+2}}\big)\|w\|_{\L^2}^2. 
		\end{align}Let us take the term $\beta(1+\gamma)((v_1+z)^{\delta+1}-(v_2+z)^{\delta+1},w)$ and estimate it using  Taylor's formula and H\"older's inequality as
		\begin{align}\label{U4}\nonumber
			&\beta(1+\gamma)((v_1+z)^{\delta+1}-(v_2+z)^{\delta+1},w)\nonumber\\&= \beta(1+\gamma)(\delta+1)(w(\theta(v_1+z)+(1-\theta)(v_2+z))^{\delta},w) \nonumber\\& \leq \beta(1+\gamma)(\delta+1)2^{\delta-1}(\|(|v_1+z|^{\delta}+|v_2+z|^{\delta})w\|_{\L^2})\|w\|_{\L^2}\nonumber\\& \leq  \frac{\beta}{4} \||v_1+z|^{\delta}w\|_{\L^2}^2+\frac{\beta}{4} \||v_2+z|^{\delta}w\|_{\L^2}^2+\beta 2^{2\delta-1}(1+\delta)^2(1+\gamma)^2\|w\|_{\L^2}^2.
		\end{align}
		Also, we have
		\begin{align}\label{U5}
			-\beta\gamma((v_1+z)-(v_2+z),w) =-\beta\gamma(w,w)=-\beta\gamma\|w\|_{\L^2}^2.
		\end{align}In order to estimate the final term of right hand side of \eqref{U2}, we use the following formula (cf. \cite{MTM1})
		\begin{align}\label{U6}
			(x|x|^{2\delta}-y|y|^{2\delta},x-y) \geq \frac{1}{2}\||x|^{\delta}(x-y)\|_{\L^2}^2+\frac{1}{2}\||y|^{\delta}(x-y)\|_{\L^2}^2.
		\end{align} 
		Let us take the term $-\beta((v_1+z)^{2\delta+1}-(v_2+z)^{2\delta+1},w)$ and using the above formula \eqref{U6} to estimate it as
		\begin{align}\label{U7}
			-\beta((v_1+z)^{2\delta+1}-(v_2+z)^{2\delta+1},w) \leq -\frac{\beta}{2} \||v_1+z|^{\delta}w\|_{\L^2}^2-\frac{\beta}{2}\||v_2+z|^{\delta}w\|_{\L^2}^2.
		\end{align}Combining \eqref{U4}, \eqref{U5} and \eqref{U7}, we get 
		\begin{align}\label{U8}\nonumber
			&\beta(((1+\gamma)((v_1+z)^{\delta+1}-(v_2+z)^{\delta+1}))-\gamma((v_1+z)-(v_2+z))-((v_1+z)^{2\delta+1}\\&\quad-(v_2+z)^{2\delta+1}),w) \nonumber\\&\leq  -\frac{\beta}{4} \||v_1+z|^{\delta}w\|_{\L^2}^2-\frac{\beta}{4}\||v_2+z|^{\delta}w\|_{\L^2}^2-\beta\gamma\|w\|_{\L^2}^2+\beta 2^{2\delta-1}(1+\delta)^2(1+\gamma)^2\|w\|_{\L^2}^2.
		\end{align}
		Using \eqref{U3} and \eqref{U8} in \eqref{U2}, we obtain
		\begin{align}\label{U9}\nonumber
			&\frac{\d}{\d t} \|w\|_{\L^2}^2+\nu \|\partial_{\xi}w\|_{\L^2}^2+\frac{\beta}{2} \||v_1+z|^{\delta}w\|_{\L^2}^2+\frac{\beta}{2} \||v_2+z|^{\delta}w\|_{\L^2}^2+2\beta\gamma\|w\|_{\L^2}^2\nonumber\\& \leq c2^{\frac{2\delta(\delta+1)}{\delta+2}}\alpha^{\frac{4(\delta+1)}{\delta+2}}\bigg(\frac{\delta+2}{4(\delta+1)}\bigg)\bigg(\frac{3\delta+2}{2\nu(\delta+1)}\bigg)^{\frac{3\delta+2}{\delta+2}}\big(\|v_1+z\|_{\L^{2(\delta+1)}}^{\frac{4\delta(\delta+1)}{\delta+2}}\nonumber+\|v_2+z\|_{\L^{2(\delta+1)}}^{\frac{4\delta(\delta+1)}{\delta+2}}\big)\|w\|_{\L^2}^2\nonumber\\&\quad+\beta 2^{2\delta}(1+\delta)^2(1+\gamma)^2\|w\|_{\L^2}^2.
		\end{align}
		An application of Gronwall's inequality in \eqref{U9} yields 
		\begin{align}\label{U10}
			\|w(t)\|_{\L^2}^2 &\leq \|w(0)\|_{\L^2}^2\exp(2^{\delta}\beta(1+\delta)^2(1+\gamma)^2T)\nonumber\\&\quad \times \exp\bigg\{C(\alpha,\delta,\nu) \int_{0}^{T}\bigg(\|v_1(t)+z(t)\|_{\L^{2(\delta+1)}}^{\frac{4\delta(\delta+1)}{\delta+2}}+\|v_2(t)+z(t)\|_{\L^{2(\delta+1)}}^{\frac{4\delta(\delta+1)}{\delta+2}}\bigg)\d t\bigg\},
		\end{align} for all $t\in[0,T]$. For $1\leq \delta\leq 2$, the term appearing inside the exponential is finite and since $v_1(\cdot) \text{ and } v_2(\cdot)$ are weak solutions of the system.  Since $w(0)=0$, the uniqueness follows from \eqref{U10} for any $\nu,\delta,\alpha$ and $\delta\in[1,2]$. 
		
		From \eqref{U10}, we obtain the uniqueness of the weak solution with a restriction on $\delta\in[1,2]$. To remove this restriction, we estimate the  term $\frac{\alpha}{\delta+1}((v_1+z)^{\delta+1}-(v_2+z)^{\delta+1},\partial_{\xi} w)$, with the help of Taylor's formula, H\"older's and Young's inequalities as
		\begin{align}\label{U11}\nonumber
			&\frac{\alpha}{\delta+1}((v_1+z)^{\delta+1}-(v_2+z)^{\delta+1},\partial_{\xi} w) \\&= \alpha (w(\theta (v_1+z)+(1-\theta)(v_2+z))^{\delta},\partial_{\xi} w ) \nonumber\\& \leq 2^{\delta-1}\alpha(\||v_1+z|^{\delta}w\|_{\L^2}+\||v_2+z|^{\delta}w\|_{\L^2})\|\partial_{\xi}w\|_{\L^2}
			\nonumber\\& \leq \theta\nu\|\partial_{\xi}w\|_{\L^2}^2+\frac{2^{2\delta-2}\alpha^2}{2\theta\nu}\||v_1+z|^{\delta}w\|_{\L^2}^2+\frac{2^{2\delta-2}\alpha^2}{2\theta\nu}\||v_2+z|^{\delta}w\|_{\L^2}^2.	\end{align}
		Using \eqref{U4} and \eqref{U10} in \eqref{U2}, we obtain 
		\begin{align}\label{U12}\nonumber
			\frac{\d}{\d t} \|w\|_{\L^2}^2+\nu(1-\theta) \|\partial_{\xi}w\|_{\L^2}^2&+\bigg(\beta(1-\theta)-\frac{2^{2\delta-2}\alpha^2}{\theta\nu}\bigg)\bigg\{ \||v_1+z|^{\delta}w\|_{\L^2}^2+ \||v_2+z|^{\delta}w\|_{\L^2}^2\bigg\}\\+2\beta\gamma\|w\|_{\L^2}^2& \leq C(\beta,\delta,\gamma)\|w\|_{\L^2}^2.
		\end{align}For $\beta\nu>2^{2\delta-2}\alpha^2$, an application of Gronwall's inequality in \eqref{U12} gives
		\begin{align}\label{U13}
			\|w(t)\|_{\L^2}^2\leq \|w(0)\|_{\L^2}^2\exp\{C(\beta,\delta,\gamma)T\},
		\end{align}for all $t\in[0,T]$. Hence the uniqueness follows from \eqref{U13}, provided $\beta\nu>2^{2(\delta-1)}\alpha^2$.
	\end{proof}
	\begin{remark}
		Under additional regularity assumptions on $z$ and $x\in\H_0^1(\mathcal{O})$, one can prove the existence of a strong solution $v\in\mathrm{L}^{\infty}(0,T;\H_0^1(\mathcal{O}))\cap\mathrm{L}^2(0,T;\H^2(\mathcal{O}))$ to the problem \eqref{2.4} and hence from the estimate \eqref{U10}, it follows that the strong solution (in the deterministic sense) is unique for any $\nu,\alpha,\beta>0$, $\gamma\in(0,1)$ and $1\leq \delta<\infty$.
	\end{remark}
	Since $u=v+z$ and $z\in\C([0,T];\C([0,1]))$, $\mathbb{P}$-a.s.,  we have the following results:
	\begin{theorem}\label{thm2.6}
		For $x\in\L^2(\1)$,	under assumption \eqref{13}, there exists a strong solution $u$ to the equation \eqref{2.2} in the sense of Definition \ref{def2.3} such that 
		\begin{align}\label{2.42}
			u\in\C([0,T];\L^2(\mathcal{O}))\cap\mathrm{L}^2(0,T;\C([0,1]))\cap\L^{2(\delta+1)}(0,T;\L^{2(\delta+1)}(\mathcal{O})),\  \mathbb{P}\text{-a.s.}
		\end{align}
		For $\delta\in[1,2]$, the strong solution is unique for any $\nu,\alpha,\beta>0, \gamma\in(0,1)$ and for $2<\delta<\infty$, the strong solution is unique for $\beta\nu>2^{2(\delta-1)}\alpha^2$.
	\end{theorem}
	\begin{theorem}\label{thm2.7}
		For $x\in\L^2(\1)$,	under assumption \eqref{1.3}, there exists a mild solution $u$ to the equation \eqref{2.2} in the sense of Definition \ref{def2.1} such that \eqref{2.42} is satisfied. 
	\end{theorem}

	\section{Existence of an Invariant Measure}\label{sec3}\setcounter{equation}{0}
	In this section, we discuss the existence of  an invariant measures for  solutions of the equation \eqref{2.2}. Let us first introduce some notations, which will be used in the upcoming sections. We denote the space of probability measures on $\L^2(\1)$ equipped with the Borel $\sigma$-field $\mathcal{B}$ by $\M_{1}(\L^2(\1))$, the space of signed $\sigma$-additive measures of bounded variation on $\L^2(\1)$ by  $\M_{b}(\L^2(\1))$,  the space of all bounded Borel measurable functions on $\L^2(\1)$ by $\mathrm{B}_b(\L^2(\1))$ and the space of all bounded continuous functions on $\L^2(\1)$ by $\C_b(\L^2(\1))$. On the space $\M_{b}(\L^2(\1))$, we consider $\sigma(\M_{b}(\L^2(\1)), \mathrm{B}_{b}(\L^2(\1)))$, the  $\tau$-topology of convergence against measurable and bounded functions which is much stronger than the usual weak convergence topology $\sigma(\M_{b}(\L^2(\1)),\C_{b}(\L^2(\1)))$ (\cite{MDD}, Section 6.2, \cite{ADOZ}). Let us denote	$\|\cdot\|_{\sup}$ for the supremum norm in $\C_{b}\ (\text{or } \mathrm{B}_b)$. We denote the duality relation between $\varrho \in \M_{b}(\L^2(\1))$ and $\psi \in \mathrm{B}_b(\L^2(\1))$  by $ \varrho(\psi) := \int_{E} \psi \d \varrho$.   In the sequel, we denote the law on $\C(\R^+;\L^2(\1))$ of the Markov process with $x\in\L^2(\1)$ as initial state by $\P_{x}$. We define $\P_{\varrho}(\cdot)=\int_{\L^2(\1)}\P_{x}\varrho(\d x)$, where $\varrho$ be any initial measure on $\L^2(\1)$.

	Let  $\mathrm{E}$ be any Borel subset of $\L^2(\1)$, and the transition probability measure $\mathrm{P}(t,x,\cdot)$ be  defined by $\mathrm{P}(t,x,B)=\mathbb{P}\{u(t,x)\in B\}$, for all $t>0, x \in \mathrm{E}$ and all Borelian sets $B \in \mathcal{B}(\mathrm{E})$, where $u(t,x)$ is the solution of the SGBH equation \eqref{2.2} with the initial condition $x\in\L^2(\1)$. Such a process is shown to exists and Markovian. We define  $\{\mathrm{P}_{t}\}_{t\geq 0},$ a Markov semigroup in the space $\C_{b}(\mathrm{E})$ corresponding to the strong solution of SGBH equation \eqref{2.2}, as 
	\begin{align*}
		(\mathrm{P}_{t}\varphi )(x) = \mathbb{E}[\varphi(u(t,x))], \ \text{ for all }\  \varphi \in \C_{b}(\mathrm{E}).
	\end{align*}
	A Markov semigroup $\mathrm{P}_t,\ t\geq 0$ is  \emph{Feller} if $\mathrm{P}_{t}:\C_{b}(\mathrm{E}) \to \C_{b}(\mathrm{E})$ for arbitrary $t>0 $.	Let us first consider the dual semigroup $\{\mathrm{P}^{*}_{t}\}_{t\geq0}$ in the space $\M_1(\mathrm{E})$, which is defined as 
	\begin{align*}
		\int_{\mathrm{E}} \varphi \d(\mathrm{P}^{*}_{t}\varrho) = \int_{\mathrm{E}}\mathrm{P}_{t}\varphi\d\varrho,
	\end{align*}
	for all $\varphi \in \C_{b}(\mathrm{E}) \text{ and } \varrho \in \M_1(\mathrm{E})$.
	A measure $\varrho \in \M_1(\mathrm{E})$ is called \emph{invariant} if $\mathrm{P}^{*}_{t}\varrho = \varrho $ for all $t\geq 0 $.	Under assumptions \eqref{1.3} and \eqref{13},  we prove the existence of  an invariant measure for the SGBH equation \eqref{2.2} for the following two different cases:
	\begin{itemize}
		\item[(i)] The noise coefficient has finite trace  (assumption \eqref{13})  and without any restriction on $\delta$ ($\beta\nu>2^{2(\delta-1)}\alpha^2$ for $2<\delta<\infty$).
		\item[(ii)] The general case (assumption \eqref{1.3}) with the restriction on $\delta\in[1,2)$.
	\end{itemize} 
	\vskip 0.2 cm 
	\noindent 
	\textbf{Case (i):} \emph{$\Tr(\G\G^*)<\infty$.}  We state and prove the existence of the invariant measure for this case in the following Theorem:
	\begin{theorem}\label{thEx1}
		Let us take the initial data $x \in \L^2(\1)$.  Then there exists an invariant measure for the system \eqref{2.2} with support in $\H_0^1(\1)$ ($\beta\nu>2^{2(\delta-1)}\alpha^2,$ for $2<\delta<\infty$). 
	\end{theorem}
	\begin{proof}
		Applying infinite dimensional It\^o's formula to the process $\|u(\cdot)\|_{\L^2}^2$, we find
		\begin{align}\label{5.14}\nonumber
			&\|u(t)\|_{\L^2}^2+2\nu\int_{0}^{t}\|\partial_{\xi} u(s)\|_{\L^2}^2\d s +\beta\int_{0}^{t} \|u(s)\|_{\L^{2(\delta+1)}}^{2(\delta+1)}\d s \\&=\|x\|^{2}_{\L^2}  +\Tr(\G\G^{*})t+ \beta(1+\gamma^2)\int_{0}^{t}\|u(s)\|_{\L^2}^2\d s+2\int_{0}^{t}	( \G \d \W(s), u(s)) ,
		\end{align} for all $t\in[0,T]$, $\P$-a.s., where we have used \eqref{2.1}. Using $\L^{2(\delta+1)}(\1)\subset \L^2(\1)$, H\"older's and Young's inequalities  to estimate  the term $\beta(1+\gamma^2)\int_{0}^{t}\|u(s)\|_{\L^2}^2\d s$, we obtain
		\begin{align*}
			\beta(1+\gamma^2)\int_{0}^{t}\|u(s)\|_{\L^2}^2\d s &\leq \beta(1+\gamma^2)\bigg(\int_{0}^{t}\|u(s)\|_{\L^{2(\delta+1)}}^{2(\delta+1)}\d s\bigg)^{\frac{1}{\delta+1}}t^{\frac{\delta}{\delta+1}}
			\\& \leq \frac{\beta}{2}\int_{0}^{t}\|u(s)\|_{\L^{2(\delta+1)}}^{2(\delta+1)}\d s+ C (\beta,\delta)t,
		\end{align*}where the constant $C(\beta,\delta)=\big(\frac{2}{\beta(\delta+1)}\big)^{\frac{1}{\delta}}\frac{\delta}{\delta+1}$. Using the above estimate in \eqref{5.14}, we have 
		\begin{align}\label{5.15}\nonumber
			&\|u(t)\|_{\L^2}^2+2\nu\int_{0}^{t}\|\partial_{\xi} u(s)\|_{\L^2}^2\d s +\frac{\beta}{2}\int_{0}^{t} \|u(s)\|_{\L^{2(\delta+1)}}^{2(\delta+1)}\d s \\&\leq \|x\|^{2}_{\L^2}  +(\Tr(\G\G^{*})+C)t+ 2\int_{0}^{t}	( \G_{n} \d \W(s), u(s)) ,
		\end{align}for all $t\in[0,T]$, $\P$-a.s.
		Taking expectation in  \eqref{5.15}, we obtain
		\begin{align}\label{Expectation}
			\E\left[\|u(t)\|_{\L^2}^2+2\nu\int_{0}^{t}\|\partial_{\xi} u(s)\|_{\L^2}^2\d s +\frac{\beta}{2}\int_{0}^{t} \|u(s)\|_{\L^{2(\delta+1)}}^{2(\delta+1)}\d s \right]\leq \|x\|_{\L^2}^2+(\Tr(\G\G^*)+C)t,
		\end{align}where we have used the fact that the final term is a martingale. Thus, for all $t>T_0$, we have
		\begin{align*}
			\frac{2\nu}{t}\E\bigg[\int_{0}^{t}\|\partial_{\xi} u(s)\|_{\L^2}^2\d s\bigg]\leq \frac{1}{T_0}\|x\|_{\L^2}^2+\Tr(\G\G^*)+C.
		\end{align*}An application of Markov's inequality yields
		\begin{align}\label{3.1}\nonumber
			\lim_{M\to\infty}\sup_{T>T_0}\bigg[\frac{1}{T}\int_{0}^{T}\P\{\|\partial_{\xi}u(t)\|_{\L^2}^2>M\}\d t\bigg] &\leq 	\lim_{M\to\infty}\sup_{T>T_0}\frac{1}{M^2}\E\bigg[\frac{1}{T}\int_{0}^{T}\|\partial_{\xi} u(s)\|_{\L^2}^2\d s\bigg] \nonumber\\& \leq \lim_{M\to\infty}\sup_{T>T_0}\frac{1}{M^2}\bigg[\frac{1}{T_0}\|x\|_{\L^2}^2+\Tr(\G\G^*)+C\bigg] \nonumber\\& = 0.
		\end{align}
		Using the estimate \eqref{3.1} and  the compact embedding $\H_0^1 (\1)\subset \L^2(\1)$, it is clear from the standard argument that the sequence of probability measures $\mu_{t,x}(\cdot)=\frac{1}{t}\int_{0}^{t}\mathrm{P}_s(0,\cdot)\d s$ is tight. That is, for each $\epsilon>0$, there exist a compact subset $K \subset \L^2(\1)$ such that $\mu_{\epsilon}(K^c)\leq \epsilon$ for all $t>0$ and hence by Krylov-Bogoliubov theorem (see \cite{CHKH}), $\mu_{t,x} \to \mu$, weakly for $n\to\infty$ and the measure $\mu$ is an invariant measure for the transition semigroup $(\mathrm{P}_t)_{t\geq 0}$, and is defined as \begin{align*}
			\mathrm{P}_t \varphi(x)=\E[\varphi(u(t,x))],
		\end{align*}for all $\varphi \in \C_b(\L^2(\1))$, where $u(\cdot)$ is the unique strong solution of \eqref{2.2} with the initial data $x\in \L^2(\1)$.
	\end{proof}\noindent
	\vskip 0.2 cm 
	\noindent 
	\textbf{Case (ii):} \emph{The general assumption \eqref{1.3}.}  
	
	Let us  now write the problem in a different form. For any $\kappa>0$, we set $R_{\kappa}(t)=e^{-\kappa t}R(t), t\geq0$. Then the mild solution of \eqref{2.2} with the initial data $x=0$, is given by the integral form 
	\begin{align}\label{3.2}
		u(t)= -\alpha\int_{0}^{t}R_{\kappa}(t-s)\B(u(s))\d s +\beta\int_{0}^{t}R_{\kappa}(t-s)\c(u(s))\d s + z_{\kappa}(t),
	\end{align}where $z_{\kappa}(t)= \int_{0}^{t}R_{\kappa}(t-s)\G\d\W(s)$, and $\kappa>0$ will be fixed later.
	Now setting $v(t)=u(t)-z_{\kappa}$ and transform the problem associated with \eqref{3.2} into the initial value problem
	
	\begin{equation}\label{3.3}
		\left\{
		\begin{aligned} 
			v'(t)&=-\nu\A v(t)-\alpha \B(v(t)+z_{\kappa}(t))+\beta\c(v(t)+z_{\kappa}(t))+\kappa z_{\kappa}(t), \\
			v(0)&=0.
		\end{aligned}
		\right.
	\end{equation}Also note that the system \eqref{3.3} defines a transition semigroup $\mathrm{P}_t, t\geq 0$, on $\mathrm{B}_b(\L^2(\1)),$ which holds the Feller property, since the solution of \eqref{3.3} depends continuously on the initial data. With the help of general theory developed in Chapter 6, \cite{GDJZ}, in order to prove the existence of an invariant measure for \eqref{3.3}, it is sufficient to show that the family of measures 
	\begin{align*}
		\bigg\{\frac{1}{T}\int_{0}^{T}\mathrm{P}_s(0,\cdot)\d s \bigg\},
	\end{align*}is tight on $\L^2(\1)$. Before going to prove the tightness of the family of measures defined above, we require the following Lemma.
	\begin{lemma}\label{lemEx1}
		For any $\epsilon>0$, there exists $K_{\epsilon}$ such that for all $T>0$, \begin{align}\label{3.4}
			\frac{1}{T}\int_{0}^{T} \P(\|v(s)\|_{\L^2}^2>K_{\epsilon})\d s <\epsilon.
		\end{align}
	\end{lemma}
	\begin{proof}
		Let  $v$ be the solution of \eqref{3.3}. Taking the inner product of first equation of the system  \eqref{3.3} with $v(\cdot)$, we get 
		\begin{align}\label{3.5}\nonumber
			\frac{1}{2}\frac{\d}{\d t}\|v\|_{\L^2}^2+\nu\|\partial_{\xi} v\|_{\L^2}^2&= \frac{\alpha}{\delta+1}((v+z_{\kappa})^{\delta+1},\partial_{\xi} v)+\beta(1+\gamma)((v+z_{\kappa})^{\delta+1},v)\\&\quad-\beta\gamma(v+z_{\kappa},v)-\beta((v+z_{\kappa})^{2\delta+1},v)+\kappa(z_{\kappa},v).
		\end{align} Using \eqref{2.1}, Taylor's formula, H\"older's 
		and	Young's inequalities, 
		we estimate the term $\frac{\alpha}{\delta+1}((v+z_{\kappa})^{\delta+1},\partial_{\xi} v)$ as 
		
		\begin{align}\label{3.6}\nonumber
			&\frac{\alpha}{\delta+1}((v+z_{\kappa})^{\delta+1},\partial_{\xi} v) \\&\leq\frac{\alpha}{\delta+1}(v^{\delta+1},\partial_{\xi} v)+\alpha(z_{\kappa}(\theta v+(1-\theta)z_{\kappa})^{\delta},\partial_{\xi} v) \nonumber\\& 
			\leq \frac{\nu}{2}\|\partial_{\xi} v\|_{\L^2}^2+\frac{\alpha^2}{2\nu}\|z_{\kappa}\|_{\L^{2(\delta+1)}}^2\|v+z_{\kappa}\|_{\L^{2(\delta+1)}}^{2\delta} \nonumber\\& 
			\leq \frac{\nu}{2}\|\partial_{\xi} v\|_{\L^2}^2+ \frac{\beta}{8}\|v+z_{\kappa}\|_{\L^{2(\delta+1)}}^{2(\delta+1)}+\frac{1}{\delta+1}\bigg(
			\frac{\alpha^2}{2\nu}\bigg)^{\delta+1}\bigg(\frac{8\delta}{\beta(\delta+1)}\bigg)^{\delta}\|z_{\kappa}\|_{\L^{2(\delta+1)}}^{2(\delta+1)}.
		\end{align}
		Next, we take the term $\beta(1+\gamma)((v+z_{\kappa})^{\delta+1},v)-\beta\gamma(v+z_{\kappa},v)-\beta((v+z_{\kappa})^{2\delta+1},v)$ and estimate it using the embedding $\L^2(\1)\subset \L^{2(\delta+1)} (\1)$, H\"older's and Young's inequalities as
		\begin{align}\label{3.7}\nonumber
			&\beta(1+\gamma)((v+z_{\kappa})^{\delta+1},v+z_{\kappa}-z_{\kappa})-\beta\gamma(v+z_{\kappa},v)-\beta((v+z_{\kappa})^{2\delta+1},v+z_{\kappa}-z_{\kappa})\nonumber \\& \leq
			\beta(1+\gamma)\|v+z_{\kappa}\|_{\L^{2(\delta+1)}}^{\delta+1}\|v+z_{\kappa}-z_{\kappa}\|_{\L^2}-\frac{\beta\gamma}{2}\|v\|_{\L^2}^2+\frac{\beta\gamma}{2}\|z_{\kappa}\|_{\L^2}^2-\beta\|v+z_{\kappa}\|_{\L^{2(\delta+1)}}^{2(\delta+1)}\nonumber\\&\quad+\beta\|v+z_{\kappa}\|_{\L^{2(\delta+1)}}^{2\delta+1}\|z_{\kappa}\|_{\L^{2(\delta+1)}}
			\nonumber\\& 
			\leq -\frac{3\beta}{8}\|v+z_{\kappa}\|_{\L^{2(\delta+1)}}^{2(\delta+1)}+4\beta(1+\gamma)^2\|v+z_{\kappa}\|_{\L^2}^2+\bigg(4\beta(1+\gamma)^2+\frac{\beta\gamma}{2}\bigg)\|z_{\kappa}\|_{\L^2}^2\nonumber\\&\quad-\frac{\beta\gamma}{2}\|v\|_{\L^2}^2+\frac{1}{2(\delta+1)}\bigg(\frac{2\delta+1}{\beta(\delta+1)}\bigg)^{\frac{1}{2\delta+1}}\|z_{\kappa}\|_{\L^{2(\delta+1)}}^{2(\delta+1)} \nonumber\\& 
			\leq 
			-\frac{3\beta}{8}\|v+z_{\kappa}\|_{\L^{2(\delta+1)}}^{2(\delta+1)}+4\beta(1+\gamma)^2\|v+z_{\kappa}\|_{\L^{2(\delta+1)}}^2+\bigg(4\beta(1+\gamma)^2+\frac{\beta\gamma}{2}\bigg)\|z_{\kappa}\|_{\L^2}^2\nonumber\\&\quad-\frac{\beta\gamma}{2}\|v\|_{\L^2}^2+\frac{1}{2(\delta+1)}\bigg(\frac{2\delta+1}{\beta(\delta+1)}\bigg)^{\frac{1}{2\delta+1}}\|z_{\kappa}\|_{\L^{2(\delta+1)}}^{2(\delta+1)}
			\nonumber\\&	\leq 
			-\frac{\beta}{4}\|v+z_{\kappa}\|_{\L^{2(\delta+1)}}^{2(\delta+1)}+\frac{(2\beta(1+\gamma)^2)^{\delta\frac{\delta+1}{\delta}}}{\delta+1}\bigg(\frac{8}{\beta(\delta+1)}\bigg)^{\frac{1}{\delta}}+\bigg(4\beta(1+\gamma)^2+\frac{\beta\gamma}{2}\bigg)\|z_{\kappa}\|_{\L^2}^2\nonumber\\&\quad-\frac{\beta\gamma}{2}\|v\|_{\L^2}^2+\frac{1}{2(\delta+1)}\bigg(\frac{2\delta+1}{\beta(\delta+1)}\bigg)^{\frac{1}{2\delta+1}}\|z_{\kappa}\|_{\L^{2(\delta+1)}}^{2(\delta+1)}.
		\end{align}Once again using Cauchy-Schwarz inequality and H\"older's inequality, we estimate the term $	\kappa(z_{\kappa},v)$ as
		\begin{align}\label{3.8}
			\kappa(z_{\kappa},v)\leq \frac{\beta\gamma}{4}\|v\|_{\L^2}^2+\frac{\kappa^2}{\beta\gamma}\|z_{\kappa}\|_{\L^2}^2.	
		\end{align}Substituting \eqref{3.6}-\eqref{3.8} in \eqref{3.5}, and using Poincar\'e inequality $\pi^2\|v\|_{\L^2}^2\leq \|\partial_{\xi} v\|_{\L^2}^2$, we find
		\begin{align}\label{3.9}\nonumber
			&\frac{\d}{\d t}\|v\|_{\L^2}^2+\bigg(\nu\pi^2+\frac{\beta\gamma}{2}\bigg)\| v\|_{\L^2}^2+\frac{\beta}{4}\|v+z_{\kappa}\|_{\L^{2(\delta+1)}}^{2(\delta+1)} \nonumber\\&\leq 
			C(\beta,\gamma,\kappa)\|z_{\kappa}\|_{\L^2}^2+C(\nu,\alpha,\beta,\delta)\|z_{\kappa}\|_{\L^{2(\delta+1)}}^{2(\delta+1)}+C(\beta,\gamma,\delta),
		\end{align}where $C(\beta,\gamma,\kappa)=\bigg(8\beta(1+\gamma)^2+\beta\gamma+\frac{2\kappa^2}{\beta\gamma}\bigg)$, $C(\beta,\gamma,\delta)=\frac{2\delta(2\beta(1+\gamma)^2)^{\frac{\delta+1}{\delta}}}{\delta+1}\bigg(\frac{8}{\beta(\delta+1)}\bigg)^{\frac{1}{\delta}}$, and $C(\nu,\alpha,\beta,\delta)=\bigg(\frac{1}{\delta+1}\bigg(
		\frac{\alpha^2}{\nu}\bigg)^{\delta+1}\bigg(\frac{8\delta}{2\beta(\delta+1)}\bigg)^{\delta}+\frac{1}{(\delta+1)}\bigg(\frac{2\delta+1}{\beta(\delta+1)}\bigg)^{\frac{1}{2\delta+1}}\bigg)$.
		
		Finally, we proceed with a similar argument as in the work Chapter 14, \cite{GDG}. We fix $K>1$ and define 
		\begin{align*}
			\zeta(t) = \log(\|v(t)\|_{\L^2}^2\vee K). 
		\end{align*}
		Then we have 
		\begin{align*}
			\zeta'(t) = \frac{1}{\|v(t)\|_{\L^2}^2}\chi_{\{\|v(t)\|_{\L^2}^2\geq K\}}\frac{\d}{\d t}\|v(t)\|_{\L^2}^2.
		\end{align*}
		On multiplying both sides of \eqref{3.9}, with  $\frac{1}{\|v(t)\|_{\L^2}^2}\chi_{\{\|v(t)\|_{\L^2}^2\geq K\}}$, we get
		\begin{align*}
			&\zeta'(t)+\bigg(\nu\pi^2+\frac{\beta\gamma}{2}\bigg)\chi_{\{\|v(t)\|_{\L^2}^2\geq K\}}\\&\leq \frac{1}{K}\big\{C(\beta,\gamma,\kappa)\|z_{\kappa}\|_{\L^2}^2+C(\nu,\alpha,\beta,\delta)\|z_{\kappa}\|_{\L^{2(\delta+1)}}^{2(\delta+1)}+C(\beta,\gamma,\delta) \big\}.
		\end{align*}Integrating the above inequality from $0$ to $t$ and taking the expectation on both sides, we find
		\begin{align*}
			&\E[\zeta(t)-\zeta(0)]+\bigg(\nu\pi^2+\frac{\beta\gamma}{2}\bigg)\int_{0}^{t}\P(\|v(s)\|_{\L^2}^2\geq K)\d s\\& \leq \frac{1}{K}\int_{0}^{t}\left\{C(\beta,\gamma,\kappa)\E\left[\|z_{\kappa}(s)\|_{\L^2}^2\right]+C(\nu,\alpha,\beta,\delta)\E\left[\|z_{\kappa}(s)\|_{\L^{2(\delta+1)}}^{2(\delta+1)}\right]\nonumber+C(\beta,\gamma,\delta) \right\}\d s.
		\end{align*}Using the fact $\zeta(t)-\zeta(0)= \zeta(t)-\log K=\log\left(\frac{\|v(t)\|_{\L^2}^2\vee K}{K}\right)\geq 0$, we obtain 
		\begin{align}\label{3.10}\nonumber
			&\bigg(\nu\pi^2+\frac{\beta\gamma}{2}\bigg)\frac{1}{t}\int_{0}^{t}\P(\|v(s)\|_{\L^2}^2\geq K)\d s \\& \leq \frac{1}{Kt}\int_{0}^{t}\big\{C(\beta,\gamma,\kappa)\E(\|z_{\kappa}(s)\|_{\L^2}^2)+C(\nu,\alpha,\beta,\delta)
			\E(\|z_{\kappa}(s)\|_{\L^{2(\delta+1)}}^{2(\delta+1)})+C(\beta,\gamma,\delta) \big\}\d s,
		\end{align}hence \eqref{3.4} holds by choosing $K$ sufficiently large.
	\end{proof}
	Now we state our main result of this section, that is, the existence of  the invariant measure for the noise with general  assumption \eqref{1.3}.
	\begin{theorem}\label{thEx2}
		Let $\mathrm{P}_{t}, t\geq 0$ be the transition semigroup corresponding to the solutions of the system \eqref{2.2}. Then there exists an invariant measure for the semigroup   $\mathrm{P}_{t}, t\geq 0$ for $\delta \in[1,2)$.
	\end{theorem}
	\begin{proof}
		First we fix $\kappa>0$, for which the Lemma \ref{lemEx1} holds. We have the embedding $\D(\A^{\sigma})\subset \L^2(\1)$, is compact for any $\sigma>0$. In order to prove the tightness property, it is sufficient to show that,
		for any $\epsilon>0$ there exists $K>0$, such that for all $T>1$
		\begin{align}\label{3.11}
			\frac{1}{T}\int_{0}^{T}\P\left\{\|\A^{\sigma}u(t)\|_{\L^2}^2\geq K\right\}\d t <\epsilon.
		\end{align}
		Now onward we fix both $\kappa>0$ and $\sigma <\frac{2-\delta}{4\delta}$ for $\delta\in[1,2)$. For $\sigma\in[0,\frac{1}{4})$, it has been shown in  Lemma 14.4.1, \cite{GDJZ} that 
		\begin{align*}
			M=\sup_{t\geq1}\E[\|\A^{\sigma}z_{\kappa}(t)\|_{\L^2}^2]<+\infty.
		\end{align*}An application of Markov's inequality yields 
		\begin{align*}
			I_T&=\frac{1}{T}\int_{1}^{T}\P\left\{\|\A^{\sigma}z_{\kappa}(t)\|_{\L^2}^2\geq K\right\}\d t  \leq \frac{1}{TK}\int_{1}^{T}\E\left[\|\A^{\sigma}z_{\kappa}(t)\|_{\L^2}^2\right]\d t  \leq \frac{M}{TK}(T-1).
		\end{align*}For the chosen sufficiently large $K$, we can made $I_T$ small uniformly for $T\geq1$. To prove \eqref{3.11}, it is sufficient to show it for the process $u(\cdot)$ replaced by $v(\cdot)$, since $u(\cdot)= v(\cdot)+z(\cdot)$. This will be derived as  did in Theorem 6.1.2, \cite{GDJZ}  by exploiting the regularizing effect of \eqref{2.2}. 	For the mild solution $v(\cdot)$ of \eqref{3.3} and any $t>0$, we have	
		\begin{align*}
			\A^{\sigma}v(t+1)&=\A^{\sigma}R(1)v(t)-\alpha\A^{\sigma}\int_{t}^{t+1}R(t+1-s)\B(v(s)+z_{\kappa}(s))\d s\\&\quad+\beta\A^{\sigma}\int_{t}^{t+1}R(t+1-s)\c(v(s)+z_{\kappa}(s))\d s+\kappa\A^{\sigma}\int_{t}^{t+1}R(t+1-s)z_{\kappa}(s)\d s.
		\end{align*}Taking the $\L^2$-norm of above expression, we obtain
		\begin{align}\label{3.12}\nonumber
			\|\A^{\sigma}v(t+1)\|_{\L^2}&=\|\A^{\sigma}R(1)v(t)\|_{\L^2}+\alpha\bigg\|\A^{\sigma}\int_{t}^{t+1}R(t+1-s)\B(v(s)+z_{\kappa}(s))\d s\bigg\|_{\L^2}\nonumber\\&\quad+\beta\bigg\|\A^{\sigma}\int_{t}^{t+1} R(t+1-s)\c(v(s)+z_{\kappa}(s))\d s\bigg\|_{\L^2}\nonumber\\&\quad+\kappa\bigg\|\A^{\sigma}\int_{t}^{t+1}R(t+1-s)z_{\kappa}(s)\d s\bigg\|_{\L^2}.
		\end{align}
		Using semigroup property (see \eqref{1.5} and \eqref{1.7}), interpolation 
		and Young's inequalities, 
		we estimate the term $\big\|\A^{\sigma}\int_{0}^{t}R(s)\partial_{\xi} u^{\delta+1}(s)\d s\big\|_{\L^2}$ as 
		\begin{align*}
			\bigg\|\A^{\sigma}\int_{0}^{t}R(s)\partial_{\xi} u^{\delta+1}(s)\d s\bigg\|_{\L^2}&\leq C\int_{0}^{t}s^{-(\sigma+\frac{3}{4})}\|u(s)\|_{\L^{\delta+1}}^{\delta+1}\d s\nonumber\\&\leq C\int_{0}^{t}s^{-(\sigma+\frac{3}{4})}\|u(s)\|_{\L^2}^{\frac{\delta+1}{\delta}}\|u(s)\|_{\L^{2(\delta+1)}}^{\frac{(\delta+1)(\delta-1)}{\delta}}\d s \\&\leq C\bigg( \int_{0}^{t}s^{-(\sigma+\frac{3}{4})\frac{2\delta}{\delta+1}}\|u(s)\|_{\L^2}^2\d s+\int_{0}^{t}\|u(s)\|_{\L^{2(\delta+1)}}^{2(\delta+1)}\d s \bigg)\\& \leq   C\bigg[t^{-(\sigma+\frac{3}{4})\frac{2\delta}{\delta+1}+1}\sup_{s\in[0,t]}\|u(s)\|_{\L^2}^2+\int_{0}^{t}\|u(s)\|_{\L^{2(\delta+1)}}^{2(\delta+1)}\d s\bigg].
		\end{align*}With the help of semigroup property (see \eqref{1.6} and \eqref{1.7}), the embedding $\L^{2(\delta+1)}(\1)\subset\L^{\delta+1}(\1)$, interpolation and Young's inequalities, 
		we estimate the term $\big\|\A^{\sigma}\int_{0}^{t}R(s)c(u(s))\d s \big\|_{\L^2}$ as 
		\begin{align*}
			&\bigg\|\A^{\sigma}\int_{0}^{t}R(s)c(u(s))\d s \bigg\|_{\L^2} \nonumber\\&\leq C\left((1+\gamma)\int_0^ts^{-(\sigma+\frac{1}{4})}\|u(s)\|_{\L^{\delta+1}}^{\delta+1}\d s+\gamma\int_0^ts^{-\sigma}\|u(s)\|_{\L^2}\d s+\int_0^ts^{-(\sigma+\frac{1}{4})}\|u(s)\|_{\L^{2\delta+1}}^{2\delta+1}\right)\\&\leq  C\bigg( \int_{0}^{t}s^{-(\sigma+\frac{1}{4})}\|u(s)\|_{\L^{2(\delta+1)}}^{\delta+1}\d s+t^{-2\sigma+\frac{1}{2}}+t\sup_{s\in[0,t]}\|u(s)\|_{\L^2}^2\\&\quad+\int_{0}^{t}s^{-(\sigma+\frac{1}{4})}\|u(s)\|_{\L^2}^\frac{1}{\delta}\|u(s)\|_{\L^{2(\delta+1)}}^{\frac{(\delta+1)(2\delta-1)}{\delta}}\d s\bigg)
			\\&\leq C\bigg( t^{-2\sigma+\frac{1}{2}}+\int_{0}^{t}\|u(s)\|_{\L^{2(\delta+1)}}^{\L^{2(\delta+1)}}\d s+(t+t^{-(\sigma+\frac{1}{4})2\delta+1})\sup_{s\in[0,t]}\|u(s)\|_{\L^2}^2\bigg).
		\end{align*}
		Let us consider the final term $\big\|\A^{\sigma}\int_{0}^{t}R(s)z_{\kappa}(s)\d s\big\|_{\L^2}$. We estimate using semigroup property (see \eqref{1.7}) and Young's inequality as 
		
		\begin{align*}
			\bigg\|\A^{\sigma}\int_{0}^{t}R(s)z_{\kappa}(s)\d s\bigg\|_{\L^2}
			& \leq C \bigg[\int_{0}^{t}\{s^{-2\sigma}+\|z_{\kappa}(s)\|_{\L^2}^2\}\d s\bigg] \\& \leq  C \bigg(t^{-2\sigma+1}+ t\sup_{s\in[0,t]}\|z_{\kappa}(s)\|_{\L^2}^2\bigg).
		\end{align*}In a similar fashion, we estimate the terms of \eqref{3.12} as
		\begin{align*}
			&	\|\A^{\sigma} R(1)v(t)\|_{\L^2} \leq K_1 \|v(t)\|_{\L^2},
			\\ 
			&\alpha\bigg\|\A^{\sigma}\int_{t}^{t+1}R(t+1-s)\B(v(s)+z_{\kappa}(s))\d s\bigg\|_{\L^2}\nonumber\\&\quad \leq C \bigg(\sup_{s\in[0,1]}\|(v+z_{\kappa})(t+s)\|_{\L^2}^2+\int_{t}^{t+1}\|(v+z_{\kappa})(s)\|_{\L^{2(\delta+1)}}^{2(\delta+1)}\d s\bigg), 
			\\ &\beta\bigg\|\A^{\sigma}\int_{t}^{t+1}R(t+1-s)\c(v(s)+z_{\kappa}(s))\d s\bigg\|_{\L^2} \nonumber\\&\quad\leq C\bigg(1+\int_{t}^{t+1}\|(v+z_{\kappa})(s)\|_{\L^{2(\delta+1)}}^{2(\delta+1)}\d s+\sup_{s\in[0,1]}\|(v+z_{\kappa})(t+s)\|_{\L^2}^2\bigg),\\
			&	\kappa\bigg\|\A^{\sigma}\int_{t}^{t+1}R(t+1-s)z_{\kappa}(s)\d s\bigg\|_{\L^2} \leq C\bigg( 1+\sup_{s\in[0,1]}\|z_{\kappa}(t+s)\|_{\L^2}^2\bigg),
		\end{align*} and the right hand sides of the above estimates are  finite provided $\sigma <\frac{2-\delta}{4\delta}$ and $\delta\in[1,2)$. 
		Using the above estimates in \eqref{3.12}, we obtain
		\begin{align}\label{3.13}\nonumber
			\|\A^{\sigma}v(t+1)\|_{\L^2} &\leq  K_1 \|v(t)\|_{\L^2}+C\sup_{s\in[0,1]}\|v(t+s)\|_{\L^2}^2+K_3\sup_{s\in[0,1]}\|z_{\kappa}(t+s)\|_{\L^2}^2\\&\quad+C\int_{t}^{t+1}\|(v+z_{\kappa})(s)\|_{\L^{2(\delta+1)}}^{2(\delta+1)}\d s+C,
		\end{align}
		for $\sigma <\frac{2-\delta}{4\delta}$ and $\delta\in[1,2)$. 
		Integrating \eqref{3.9} from $t$ to $t+s$ and taking  supremum over $s\in[0,1]$, we find
		\begin{align*}
			\sup_{s\in[0,1]}\|v(t+s)\|_{\L^2}^2 &\leq \|v(t)\|_{\L^2}^2+ C(\beta,\gamma,\kappa)\int_{t}^{t+1}\|z_{\kappa}(r)\|_{\L^2}^2\d r\\&\quad+C(\nu,\alpha,\beta,\delta)\int_{t}^{t+1}\|z_{\kappa}(r)\|_{\L^{2(\delta+1)}}^{2(\delta+1)}\d r+C(\beta,\gamma,\delta).
		\end{align*}Again integrating \eqref{3.9} from $t$ to $t+1$, we obtain 
		\begin{align*}
			\int_{t}^{t+1}\|(v+z_{\kappa})(r)\|_{\L^{2(\delta+1)}}^{2(\delta+1)}\d r &\leq  \frac{4}{\beta}\bigg\{\|v(t)\|_{\L^2}^2+ C(\beta,\gamma,\kappa)\int_{t}^{t+1}\|z_{\kappa}(r)\|_{\L^2}^2\d r\\&\quad+C(\nu,\alpha,\beta,\delta)\int_{t}^{t+1}\|z_{\kappa}(r)\|_{\L^{2(\delta+1)}}^{2(\delta+1)}\d r+C(\beta,\gamma,\delta)\bigg\}. 
		\end{align*}Using the above two estimates in \eqref{3.13}, we get
		\begin{align}\label{3.14}\nonumber
			\|\A^{\sigma}v(t+1)\|_{\L^2} &\leq  K_1 \|v(t)\|_{\L^2}+K_2\|v(t)\|_{\L^2}^2+K_3\sup_{s\in[0,1]}\|z_{\kappa}(t+s)\|_{\L^2}^2\nonumber\\&\quad+K_4\int_{t}^{t+1}\|z_{\kappa}(r)\|_{\L^2}^2\d r+K_5\int_{t}^{t+1}\|z_{\kappa}(r)\|_{\L^{2(\delta+1)}}^{2(\delta+1)}\d r+K_6,
		\end{align}where $K_i>0, \text{ for } i\in\{1,\ldots,6\}$ are constants. Using the fact $\P \big\{K_6\geq \frac{K}{6}\big\}=0$, since $K$ has been chosen sufficiently large so the left hand can't exceed to the right hand value.	Finally, we have 
		\begin{align*}
			&	\frac{1}{T}\int_{0}^{T}\P\left\{\|\A^{\sigma}v(t+1)\|_{\L^2}\geq K\right\}\d t \\&\leq  \frac{1}{T}\int_{0}^{T}\P\left\{K_1 \|v(t)\|_{\L^2}\geq \frac{K}{6}\right\}\d t+\frac{1}{T}\int_{0}^{T}\P\left\{K_2 \|v(t)\|_{\L^2}^2\geq \frac{K}{6}\right\}\d t \\&\quad+\frac{1}{T}\int_{0}^{T}\P\left\{K_3\sup_{s\in[0,1]}\|z_{\kappa}(t+s)\|_{\L^2}^2\geq \frac{K}{6}\right\}\d t+\frac{1}{T}\int_{0}^{T}\P\left\{K_4\int_{t}^{t+1}\|z_{\kappa}(r)\|_{\L^2}^2\d r\geq \frac{K}{6}\right\}\d t\\&\quad+\frac{1}{T}\int_{0}^{T}\P\left\{K_5\int_{t}^{t+1}\|z_{\kappa}(r)\|_{\L^{2(\delta+1)}}^{2(\delta+1)}\d r\geq \frac{K}{6}\right\}\d t.
		\end{align*}
		The first two terms in the right hand side of above inequality can be made arbitrarily small uniform in time $T$ with the help of Lemma \ref{lemEx1} for sufficiently large $K$. For the remaining terms of the above inequality we use Markov's inequality and the fact that
		\begin{align*}
			\sup_{t>0}\E\left\{ \|z_{\kappa}(t)\|_{\L^2}^2+\|z_{\kappa}(t)\|_{\L^{2(\delta+1)}}^{2(\delta+1)}\right\}<\infty,
		\end{align*} and hence \eqref{3.11} follows. 
	\end{proof}

	\section{Irreducibility and Strong Feller}\label{sec4}\setcounter{equation}{0}
	In this section, we discuss two properties of the Markov semigroup $\mathrm{P}_t, t\geq 0$ associated with the solution of the SGBH equation \eqref{2.2}, namely irreducibility and strong Feller.  For any Borel subset  $\mathrm{E}$ of $\L^2(\1)$, a Markov semigroup $\mathrm{P}_t,\ t\geq 0$ is
	\begin{itemize}
		\item   \emph{irreducible } if $\mathrm{P}(t,x,B) >0, \text{ for all }\ t>0, \ x \in \mathrm{E}$ and any non-empty open subset $B\subset \mathrm{E}$,
		\item \emph{strong Feller} if $\mathrm{P}_{t}$ can be extended to the space $\mathrm{B}_b(\mathrm{E}) \text{ for any } t>0$, that is, $\mathrm{P}_{t}\varphi$ is continuous and bounded in $\mathrm{E}$ for all Borel bounded functions $\varphi$ in $\mathrm{E}$.
	\end{itemize}
	The above properties are essentially related to the uniqueness of  invariant measures. 
	\subsection{Irreducibility} Let us first show that the Markov semigroup $\mathrm{P}_t,\ t\geq 0$ is irreducible by using the ideas in \cite{GDJZ,EPJZ,RWJXLX}, etc. 
	\begin{proposition}\label{prop4.1}
		The transition semigroup $\mathrm{P}_t,\; t\geq 0$ on the space $\mathrm{B}_b(\L^2(\1))$ corresponding to the solution of SGBH equation \eqref{2.2} is irreducible for $\delta \in [1,2]$. 
	\end{proposition} 
	\begin{proof}
		Let us prove the irreducibility of the transition semigroup $\mathrm{P}_t,\ t\geq 0$ in the following steps:
		\vskip 0.2 cm
		\noindent \textbf{Step 1:} \emph{Exact controllability result.} 
		Let us fix $T>0,\; a\in \L^2(\1)$ and $b\in\H_0^1(\1)$. We first show that there exist $\u \in \L^2(0,T;\L^2(\1))$ such that for the solution $x(t), t\in[0,T]$ of the control problem

		\begin{equation}\label{4.1}
			\left\{
			\begin{aligned} 
				\partial_t x(t,\xi)&= -\nu\A x(t,\xi)-\alpha \B(x(t,\xi))+\beta\c(x(t,\xi))+\u(t,\xi),\; t>0, \;\xi \in[0,1],
				\\ x(t,0) &= x(t,1)=0,\; t>0, \\
				x(0,\xi)&=a(\xi), \; \xi \in[0,1],
			\end{aligned}
			\right.
		\end{equation}
		one has \begin{align*}
			x(T,\xi) = b(\xi), \; \xi\in[0,1].
		\end{align*}
		Assume further that $a\in\H_0^1(\mathcal{O})$. Then there exists $\v\in\L^2(0,T;\L^2(\1))$, such that for the solution $z(t), \; t\in[0,T]$, of the linear problem associated to \eqref{4.1} (see Proposition 14.4.3, \cite{GDJZ}), 
		\begin{equation}\label{4.2}
			\left\{
			\begin{aligned} 
				\partial_t z(t,\xi)&= -\nu\A z(t,\xi)+\v(t,\xi),\; t>0,\; \xi \in[0,1],		\\  z(t,0) &= z(t,1)=0,\; t>0, \\
				z(0,\xi)&=a(\xi), \; \xi \in[0,1],
			\end{aligned}
			\right.
		\end{equation}one has \begin{align*}
			z(T,\xi) = b(\xi), \; \xi\in[0,1],
		\end{align*} and one can easily verify that \eqref{4.2}, $z\in\C([0,T];\H_0^1(\1))$.  Since $a\in\H_0^1(\mathcal{O})$, it is immediate that $z\in\L^{\infty}(0,T;\H_0^1(\mathcal{O}))\cap\mathrm{L}^2(0,T;\D(\A))$. Now, we define a control 
		\begin{equation*}
			\v (t)=\left\{
			\begin{aligned}
				&0,  &&\text{ for } 0\leq t\leq t_0,\\
				&\frac{b-z(t_0)}{T-t_0}-\A z(t), &&\text{ for } t_0 < t\leq T,
			\end{aligned}
			\right.
		\end{equation*} 
		where $0<t_0<T$. Since $z\in \L^2(0,T;\D(\A)) $, we obtain $\v \in \L^2(0,T;\L^2(\1))$. 
		For the above control, the solution of the problem \eqref{4.2} is given by 
		\begin{equation*}
			z(t) =\left\{
			\begin{aligned}
				&e^{-\nu\A t}a,  &&\text{ for } 0\leq t\leq t_0,\\
				&\frac{t-t_0}{T-t_0}b+\frac{T-t}{T-t_0}z(t_0), &&\text{ for } t_0 < t\leq T,		
			\end{aligned}
			\right.
		\end{equation*}and it can be easily seen that $z(T)=b$. Thus  the linear problem \eqref{4.2} is exactly controllable. 
		Let us  now define a control $\u(\cdot)$ as 
		\begin{align*}
			\u = -\frac{\alpha}{\delta+1}\partial_{\xi} z^{\delta+1}-\beta(1+\gamma)z^{\delta+1}+\beta\gamma z+\beta z^{2\delta+1}+\v, \; t>0,\; \xi\in[0,1].
		\end{align*}
		By direct substitution, one can prove that $z(\cdot)$ is the solution of \eqref{4.1}.
		
		Now we  show that the control $\u(\cdot) \in \L^2(0,T;\L^2(\1))$. Taking the $\L^2$-norm of the control $\u$, we get 
		\begin{align}\label{4.3}
			\|\u\|^2_{\L^2} &\leq C\left( \frac{\alpha^2}{(\delta+1)^2}\|\partial_{\xi}z^{\delta+1}\|_{\L^2}^2+\beta^2(1+\gamma)^2\|z\|_{\L^{2(\delta+1)}}^{2(\delta+1)}+\beta^2\gamma^2\|z\|_{\L^2}^2+\beta^2\|z\|_{\L^{2(2\delta+1)}}^{2(2\delta+1)}+\|\v\|_{\L^2}^2\right). 
		\end{align}
		Using H\"older's inequality, we estimate the term $\frac{\alpha^2}{(\delta+1)^2}\|\partial_{\xi}z^{\delta+1}\|_{\L^2}^2$ as
		\begin{align*}
			\frac{\alpha^2}{(\delta+1)^2}\|\partial_{\xi}z^{\delta+1}\|_{\L^2}^2 = \alpha^2(z^{\delta}\partial_{\xi} z,z^{\delta}\partial_{\xi} z) \leq\alpha^2\|z\|_{\L^{\infty}}^{2\delta}\|\partial_{\xi} z\|_{\L^2}^2.
		\end{align*}Substituting it in \eqref{4.3} and integrating the resultant from $0$ to $T$, we obtain 
		\begin{align*}
			\int_{0}^{T}	\|\u(s)\|^2_{\L^2}\d s  &\leq  C\bigg(
			\alpha^2\sup_{s\in[0,T]}\|z(s)\|_{\L^{\infty}}^{2\delta}\int_{0}^{T}\|\partial_{\xi} z(s)\|_{\L^2}^2\d s+\beta^2(1+\gamma)^2\int_{0}^{T}\|z(s)\|_{\L^{2(\delta+1)}}^{2(\delta+1)}\d s\\&\quad+\beta^2\gamma^2\int_{0}^{T}\|z(s)\|_{\L^2}^2\d s+\beta^2\int_{0}^{T}\|z(s)\|_{\L^{2(2\delta+1)}}^{2(2\delta+1)}\d s+\int_{0}^{T}\|\v(s)\|_{\L^2}^2\d s\bigg) \\& <\infty,
		\end{align*} so that  $\u\in\L^2(0,T;\L^2(\1))$. Thus  the exact controllability of the system \eqref{4.1} follows. 
		Taking the inner product with $x(\cdot)$ to the first equation in \eqref{4.1}, we find 
		\begin{align*}
			&\frac{1}{2}\frac{\d}{\d t}\|x(t)\|_{\L^2}^2+\nu\|\partial_{\xi}x(t)\|_{\L^2}^2+\beta\gamma\|x(t)\|_{\L^2}^2+\beta\|x(t)\|_{\L^{2(\delta+1)}}^{2(\delta+1)}\nonumber\\&\quad=\beta(1+\gamma)(x^{\delta+1}(t),x(t))+(x(t),\v(t)).
		\end{align*}
		It follows that, if $\v=0$, the for a.e. $t\in[0,T]$, we have 
		\begin{align*}
			\frac{1}{2}\frac{\d}{\d t}\|x(t)\|_{\L^2}^2+\nu\|\partial_{\xi}x(t)\|_{\L^2}^2+\frac{\beta}{2}\|x(t)\|_{\L^{2(\delta+1)}}^{2(\delta+1)}\leq \frac{\beta}{2}(1+\gamma^2)\|x(t)\|_{\L^2}^2.
		\end{align*}
		From the above estimate it is immediate that for a.e. $t\in[0,T]$, we get $x(t)\in\H_0^1(\mathcal{O})$. Thus, it is enough to define $\v(s)=0$ in $[0,\tilde{t}]$, where $\tilde{t}\in(0,T)$ is a moment such that $x(\tilde{t})\in\H_0^1(\mathcal{O})$ and for the remaining part of the interval $[\tilde{t},T],$ we use the first part of the proof. 
		
		\vskip 0.2 cm
		\noindent \textbf{Step 2:} \emph{Irreducibility.} 	In order to prove our result, we need to estimate the $\L^2$-distance between the solution $u(\cdot)$ to equation  \eqref{2.2} with $x(0)=a\in\L^2(\1)$, and the function $x(\cdot)$.	From Theorems \ref{thm2.6} and \ref{thm2.7}, for any $a\in\L^2(\1)$,  we infer that $u(t,a)\in\H_0^1(\1)$, for a.e. $t\in[0,T]$, $\mathbb{P}$-a.s. Since $u(\cdot)$ is a Markov process in $\L^2(\1)$, for any $b\in\L^2(\1)$, $T>0$, $\eta>0$, (cf. Theorem 2.3, \cite{RWJXLX})
		\begin{align}
			\mathbb{P}\left\{\|u(T,a)-b\|_{\L^2}<\eta \right\}&=\int_{\H_0^1}\mathbb{P}\left\{\|u(T,a)-b\|_{\L^2}<\eta\big|u(t,a)=v\right\}\mathbb{P}\left\{u(t,a)\in\d v\right\}\nonumber\\&=\int_{\H_0^1}\mathbb{P}\left\{\|u(T-t,v)-b\|_{\L^2}<\eta\right\}\mathbb{P}\left\{u(t,a)\in\d v\right\},
		\end{align}
		for a.e. $t\in[0,T]$. In order to prove that $	\mathbb{P}\left\{\|u(T,a)-b\|_{\L^2}<\eta\right\}>0$, it is sufficient to show that $	\mathbb{P}\left\{\|u(T,a)-b\|_{\L^2}<\eta\right\}>0,$ for any $T>0$ and $a\in\H_0^1(\1)$.  
		
		Let us rewrite the control problem \eqref{4.1} as 
		\begin{equation}\label{4.4}
			\left\{
			\begin{aligned}
				\d \mathfrak{z}(t)+\nu\A \mathfrak{z}(t)\d t&= \u\d t,\; \mathfrak{z}(0)=0,
				\\ \d y(t)+\nu\A y(t)\d t&= \{-\alpha \B(y(t)+\mathfrak{z}(t))+\beta \c(y(t)+\mathfrak{z}(t))\}\d t,\; y(0)=a,		
			\end{aligned}
			\right.
		\end{equation}and the stochastic problem \eqref{2.2} as 
		\begin{equation}\label{4.5}
			\left\{
			\begin{aligned}
				\d z(t)+\nu\A z(t)\d t &=\G\d\W, \; z(0)=0,
				\\ \d v(t)+\nu\A v(t)\d t& =\{-\alpha\B(v(t)+z(t))+\beta\c(v(t)+z(t))\}\d t, \; v(0)=a,
			\end{aligned}
			\right.
		\end{equation}
		where we have  set $x(t)=y(t)+\mathfrak{z}(t)$ and $u(t)=v(t)+z(t)$. Now we subtract second equation of the system \eqref{4.4} from second equation of system \eqref{4.5} to obtain
		\begin{align*}
			\frac{\d}{\d t } (v(t)-y(t))+\nu\A(v(t)-y(t))&= -\alpha (\B(v(t)+z(t)))-\B(y(t)+\mathfrak{z}(t))\\&\quad+\beta (\c(v(t)+z(t))-\c(y(t)+\mathfrak{z}(t))),
		\end{align*} for a.e., $0\leq t \leq T$. Taking the inner product with $v(t)-y(t)$, we get
		\begin{align}\label{4.6}\nonumber
			\|v(t)-y(t)\|_{\L^2}^2&+2\nu\int_{0}^{t}\|\partial_{\xi}(v(s)-y(s))\|_{\L^2}^2\d s \nonumber\\&=-2\alpha \int_{0}^{t}(\B(v(s)+z(s))-\B(y(s)+\mathfrak{z}(s)),v(s)-y(s))\d s\nonumber\\&\quad+2\beta\int_{0}^{t}(\c(v(s)+z(s))-\c(y(s)+\mathfrak{z}(s)),v(s)-y(s))\d s.
		\end{align} Using integration by parts, Taylor's formula, Gagliardo-Nirenberg's inequality, interpolation inequality, Young's and H\"older's inequalities, we estimate the term $-\alpha \int_{0}^{t}(\B(v(s)+z(s))-\B(y(s)+\mathfrak{z}(s)),v(s)-y(s))\d s$ as	
		\begin{align}\label{4.7}\nonumber
				-&\alpha \int_{0}^{t}(\B(v(s)+z(s))-\B(y(s)+\mathfrak{z}(s)),v(s)-y(s))\d s \nonumber\\& \leq 
			2^{\delta-1}\alpha  \int_{0}^{t}\|v(s)+z(s)-y(s)-\mathfrak{z}(s)\|_{\L^{2(\delta+1)}}(\|v(s)+z(s)\|_{\L^{2(\delta+1)}}^\delta+\|y(s)+\mathfrak{z}(s)\|_{\L^{2(\delta+1)}}^\delta)\nonumber\\&\quad\times\|\partial_{\xi}(v(s)-y(s))\|_{\L^2}\d s 
			\nonumber\\& \leq 
			2^{\delta-1}\alpha  \int_{0}^{t}\|v(s)-y(s)\|_{\L^{2(\delta+1)}}(\|v(s)+z(s)\|_{\L^{2(\delta+1)}}^\delta+\|y(s)+\mathfrak{z}(s)\|_{\L^{2(\delta+1)}}^\delta)\nonumber\\&\qquad\times \|\partial_{\xi}(v(s)-y(s))\|_{\L^2}\d s+2^{\delta-1}\alpha  \int_{0}^{t}\|z(s)-\mathfrak{z}(s)\|_{\L^{2(\delta+1)}}(\|v(s)+z(s)\|_{\L^{2(\delta+1)}}^\delta\nonumber\\&\quad+\|y(s)+\mathfrak{z}(s)\|_{\L^{2(\delta+1)}}^\delta)\|\partial_{\xi}(v(s)-y(s))\|_{\L^2}\d s \nonumber\\& 
			\leq 2^{\delta-1}\alpha\int_{0}^{t}\|\partial_{\xi}(v(s)-y(s))\|_{\L^2}^{\frac{3\delta+2}{2(\delta+1)}} \|v(s)-y(s)\|_{\L^2}^{\frac{\delta+2}{2(\delta+1)}} (\|v(s)+z(s)\|_{\L^{2(\delta+1)}}^\delta\nonumber\\&\quad+\|y(s)+\mathfrak{z}(s)\|_{\L^{2(\delta+1)}}^\delta)\d s 
			+
			\frac{\nu}{3}\int_{0}^{t}\|\partial_{\xi}(v(s)-y(s))\|_{\L^2}^2\d s+\frac{2^{2(\delta-1)}3\nu\alpha^2}{4}\nonumber	\\& \qquad\times\int_{0}^{t}\|z(s)-\mathfrak{z}(s)\|_{\L^{2(\delta+1)}}^2(\|v(s)+z(s)\|_{\L^{2(\delta+1)}}^{2\delta}+\|y(s)+\mathfrak{z}(s)\|_{\L^{2(\delta+1)}}^{2\delta})\d s 
			\nonumber\\&\leq \frac{2\nu}{3}\int_{0}^{t}\|\partial_{\xi}(v(s)-y(s))\|_{\L^2}^2\d s+\frac{\delta+2}{4(\delta+1)}\bigg(\frac{3(3\delta+2)}{4\nu(\delta+1)}\bigg)^{\frac{\delta+2}{3\delta+2}}\int_{0}^{t}\big(\|v(s)+z(s)\|_{\L^{2(\delta+1)}}^{\frac{4\delta(\delta+1)}{\delta+2}}\nonumber\\&\quad+\|y(s)+\mathfrak{z}(s)\|_{\L^{2(\delta+1)}}^{\frac{4\delta(\delta+1)}{\delta+2}}\big)\|v(s)-y(s)\|_{\L^2}^2\d s+\frac{2^{2(\delta-1)}3\nu\alpha^2}{4}\int_{0}^{t}(\|v(s)+z(s)\|_{\L^{2(\delta+1)}}^{2\delta}\nonumber\\&\quad+\|y(s)+\mathfrak{z}(s)\|_{\L^{2(\delta+1)}}^{2\delta})\|z(s)-\mathfrak{z}(s)\|_{\L^{2(\delta+1)}}^2\d s .
		\end{align}
		Let us take the second term in the right hand side of \eqref{4.6} and rewrite it as 
		\begin{align}\label{4.8}\nonumber
			&	\beta\int_{0}^{t}(\c(v(s)+z(s))-\c(y(s)+\mathfrak{z}(s)),v(s)-y(s))\d s\nonumber\\& = \beta\int_{0}^{t}((1+\gamma)(v(s)+z(s))^{\delta+1}-\gamma(v(s)+z(s))-(v(s)+z(s))^{2\delta+1}\nonumber\\&\quad-\{(1+\gamma)(y(s)+\mathfrak{z}(s))^{\delta+1}-\gamma(y(s)+\mathfrak{z}(s))-(y(s)+\mathfrak{z}(s))^{2\delta+1}\},v(s)-y(s))\d s. 
		\end{align}
		Using Taylor's formula, Young's and H\"older's inequalities, we estimate the term $\beta(1+\gamma)\int_{0}^{t}((v(s)+z(s))^{\delta+1}-(y(s)+\mathfrak{z}(s))^{\delta+1},v(s)-y(s))\d s$ as
		\begin{align}\label{4.9}\nonumber
			&\beta(1+\gamma)\int_{0}^{t}((v(s)+z(s))^{\delta+1}-(y(s)+\mathfrak{z}(s))^{\delta+1},v(s)-y(s))\d s\nonumber\\&  = \beta(1+\gamma)\bigg(\int_{0}^{t}((v(s)+z(s))^{\delta+1}-(y(s)+\mathfrak{z}(s))^{\delta+1},v(s)+z(s)-(y(s)+\mathfrak{z}(s))\d s\nonumber\\&\quad-\int_{0}^{t}((v(s)+z(s))^{\delta+1}-(y(s)+\mathfrak{z}(s))^{\delta+1}, z(s)-\mathfrak{z}(s))\d s\bigg)\nonumber\\& \leq\beta(1+\gamma)(\delta+1)2^{\delta-1}\int_{0}^{t}\big(\||v(s)+z(s)|^{\delta}(v(s)+z(s)-y(s)-\mathfrak{z}(s))\|_{\L^2}\nonumber\\&\quad+\||y(s)+\mathfrak{z}(s)|^{\delta}(v(s)+z(s)-y(s)-\mathfrak{z}(s))\|_{\L^2}\big))\big(\|v(s)+z(s)-y(s)-\mathfrak{z}(s)\|_{\L^2}\nonumber\\&\quad+\|z(s)-\mathfrak{z}(s)\|_{\L^2}\big)\d s
			\nonumber \\&\leq\frac{\beta}{4} \int_{0}^{t}\big(
			\||v(s)+z(s)|^{\delta}(v(s)+z(s)-y(s)-\mathfrak{z}(s))\|_{\L^2}^2\nonumber\\&\quad+\||y(s)+\mathfrak{z}(s)|^{\delta}(v(s)+z(s)-y(s)-\mathfrak{z}(s))\|_{\L^2}^2\big)\d s+2^{2\delta}\beta(1+\gamma)^2(\delta+1)^2\nonumber\\&\qquad\times\int_{0}^{t}\big(\|v(s)+z(s)-y(s)-\mathfrak{z}(s))\|_{\L^2}^2+\|z(s)-\mathfrak{z}(s)\|_{\L^2}^2\big)\d s.
		\end{align}
		Let us consider the second term of the right hand side of \eqref{4.8}. Using Cauchy-Schwarz and Young's inequalities, we obtain  
		\begin{align}\label{4.10}\nonumber
			\bigg|&-\beta\gamma\int_{0}^{t}(v(s)+z(s)-y(s)-\mathfrak{z}(s),v(s)-y(s))\d s\bigg|    \nonumber\\&\leq \beta\gamma\int_{0}^{t}(\|v(s)-y(s)\|_{\L^2}^2+\|z(s)-\mathfrak{z}(s)\|_{\L^2}\|v(s)-y(s)\|_{\L^2})\d s \nonumber\\& \leq 
			\frac{3\beta\gamma}{2}\int_{0}^{t}\|v(s)-y(s)\|_{\L^2}^2\d s +\frac{2}{\beta\gamma} \int_{0}^{t}\|z(s)-\mathfrak{z}(s)\|_{\L^2}^2\d s.
		\end{align}Using Taylor's formula, \eqref{U6}, H\"older's and Young's   inequalities, we estimate the term $-\beta\int_{0}^{t}((v(s)+z(s))^{2\delta+1}-(y(s)+\mathfrak{z}(s))^{2\delta+1},v(s)-y(s))\d s$ as
		\begin{align}\label{4.11}\nonumber
			-&\beta\int_{0}^{t}((v(s)+z(s))^{2\delta+1}-(y(s)+\mathfrak{z}(s))^{2\delta+1},v(s)+z(s)-y(s)-\mathfrak{z}(s)-z(s)+\mathfrak{z}(s))\d s
			\nonumber\\& \leq -\frac{\beta}{2}\int_{0}^{t}\||v(s)+z(s)|^{\delta}(v(s)+z(s)-y(s)-\mathfrak{z}(s))\|_{\L^2}^2\d s\nonumber\\&\quad-\frac{\beta}{2}\int_{0}^{t}\||y(s)+\mathfrak{z}(s)|^{\delta}(v(s)+z(s)-y(s)-\mathfrak{z}(s))\|_{\L^2}^2\d s\nonumber\\&\quad+2^{2\delta-1}\beta(2\delta+1)\int_{0}^{t}\|v(s)+z(s)-y(s)-\mathfrak{z}(s)\|_{\L^{2(\delta+1)}}\big(\|v(s)+z(s)\|_{\L^{2(\delta+1)}}^{2\delta}\nonumber\\&\qquad+|y(s)+\mathfrak{z}(s)\|_{\L^{2(\delta+1)}}^{2\delta}\big)\|z(s)-\mathfrak{z}(s)\|_{\L^2}\d s 
			\nonumber\\&\leq -\frac{\beta}{2}\int_{0}^{t}\||v(s)+z(s)|^{\delta}(v(s)+z(s)-y(s)-\mathfrak{z}(s))\|_{\L^2}^2\d s\nonumber\\&\quad-\frac{\beta}{2}\int_{0}^{t}\||y(s)+\mathfrak{z}(s)|^{\delta}(v(s)+z(s)-y(s)-\mathfrak{z}(s))\|_{\L^2}^2\d s\nonumber\\&\quad +\frac{\beta}{2}\int_{0}^{t}\|v(s)+z(s)-y(s)-\mathfrak{z}(s))\|_{\L^{2(\delta+1)}}^{2(\delta+1)}\d s+\frac{2(\delta+1)}{2\delta+1}\big(2^{2\delta-1}\beta(2\delta+1)\big)^{\frac{2(\delta+1}{2\delta+1})}\nonumber\\&\qquad\times\frac{1}{(\beta(\delta+1))^{\frac{1}{\delta+1}}}\int_{0}^{t}\big(\|v(s)+z(s)\|_{\L^{2(\delta+1)}}^{\frac{4\delta(\delta+1)}{2\delta+1}}+\|y(s)+\mathfrak{z}(s)\|_{\L^{2(\delta+1)}}^{\frac{4\delta(\delta+1)}{2\delta+1}}\big)\|z(s)-\mathfrak{z}(s)\|_{\L^{2(\delta+1)}}^{\frac{2(\delta+1)}{2\delta+1}}\d s 	\nonumber\\&\leq -\frac{\beta}{2}\int_{0}^{t}\||v(s)+z(s)|^{\delta}(v(s)+z(s)-y(s)-\mathfrak{z}(s))\|_{\L^2}^2\d s\nonumber\\&\quad-\frac{\beta}{2}\int_{0}^{t}\||y(s)+\mathfrak{z}(s)|^{\delta}(v(s)+z(s)-y(s)-\mathfrak{z}(s))\|_{\L^2}^2\d s\nonumber\\&\quad +\frac{\beta}{2}\int_{0}^{t}\|v(s)+z(s)-y(s)-\mathfrak{z}(s))\|_{\L^{2(\delta+1)}}^{2(\delta+1)}\d s+ \frac{\beta}{2}\int_{0}^{t}\|z(s)-\mathfrak{z}(s)\|_{\L^{2(\delta+1)}}^{2(\delta+1)}\d s\nonumber\\&\quad +C(\beta,\delta)\int_{0}^{t}\big(\|v(s)+z(s)\|_{\L^{2(\delta+1)}}^{2(\delta+1)}+\|y(s)+\mathfrak{z}(s)\|_{\L^{2(\delta+1)}}^{2(\delta+1)}\big)\d s.
		\end{align}
		Using the estimates \eqref{4.7}-\eqref{4.11} in \eqref{4.6}, we get
		\begin{align}\label{4.12}\nonumber
			\|&v(t)-y(t)\|_{\L^2}^2+\frac{2\nu}{3}\int_{0}^{t}\|\partial_{\xi}(v(s)-y(s))\|_{\L^2}^2\d s+\frac{\beta}{2}\int_{0}^{t}\big(\||v(s)+z(s)|^{\delta}\nonumber\\&\qquad\times(v(s)+z(s)-y(s)-\mathfrak{z}(s))\|_{\L^2}^2+\||y(s)+\mathfrak{z}(s)|^{\delta}(v(s)+z(s)-y(s)-\mathfrak{z}(s))\|_{\L^2}^2\big)\d s
			\nonumber\\&\leq \int_{0}^{t}\big\{ C(\nu,\delta)  \big(\|v(s)+z(s)\|_{\L^{2(\delta+1)}}^{\frac{4\delta(\delta+1)}{\delta+2}}+\|y(s)+\mathfrak{z}(s)\|_{\L^{2(\delta+1)}}^{\frac{4\delta(\delta+1)}{\delta+2}}\big) +C(\beta,\gamma,\delta) \big\} \|v(s)-y(s)\|_{\L^2}^2\d s\nonumber\\&\quad +\int_{0}^{t}\big\{C(\beta,\gamma,\delta)\|z(s)-\mathfrak{z}(s)\|_{\L^2}^2+C(\beta)\|z(s)-\mathfrak{z}(s)\|_{\L^{2(\delta+1)}}^{2(\delta+1)}\big\}\d s\nonumber\\&\quad+C(\nu,\alpha,\delta)\int_{0}^{t}\big(\|v(s)+z(s)\|_{\L^{2(\delta+1)}}^{2\delta}+\|y(s)+\mathfrak{z}(s)\|_{\L^{2(\delta+1)}}^{2\delta}\big)\|z(s)-\mathfrak{z}(s)\|_{\L^{2(\delta+1)}}^2\d s\nonumber\\&\quad+C(\beta,\delta)\int_{0}^{t}\big(\|v(s)+z(s)\|_{\L^{2(\delta+1)}}^{2(\delta+1)}+\|y(s)+\mathfrak{z}(s)\|_{\L^{2(\delta+1)}}^{2(\delta+1)}\big)\d s.
		\end{align}
		Assuming 
		\begin{align*}
			\sup_{s\in[0,T]}\|z(s)\|_{\L^{2(\delta+1)}}\leq\upgamma, 
		\end{align*}
		an application of Gronwall's inequality in \eqref{4.12} yields
		\begin{align}\label{4.13}\nonumber
			&\|v(t)-y(t)\|_{\L^2}^2 \\&\leq C(\nu,\alpha,\beta,\gamma,\upgamma,\delta,T)\bigg\{\sup_{t\in[0,T]}\|z(t)-\mathfrak{z}(t)\|_{\L^{2(\delta+1)}}^2+\sup_{s\in[0,T]}\|z(t)-\mathfrak{z}(t)\|_{\L^{2(\delta+1)}}^{2(\delta+1)}\bigg\},
		\end{align}provided $2\delta^2\leq 2\delta+4$, that is, $\delta\in[1,2]$. 
		
		Since $\A$ is an analytic semigroup satisfying  \eqref{1.7},  from Theorem 5.25, \cite{DaZ}, for $0<\vartheta<\frac{\e}{4}+\frac{\delta-1}{4(\delta+1)}<\frac{\delta}{2(\delta+1)}$, we infer that  $z\in\C([0,T];\mathrm{W}^{\vartheta,2(\delta+1)}(\1))$, $\mathbb{P}$-a.s., and since $\mathrm{W}^{\vartheta,2(\delta+1)}(\1)\subset\H^{\frac{\delta}{2(\delta+1)}+\vartheta}(\1)$ (Sobolev's inequality) implies that  $z\in\C([0,T];\mathrm{H}^{\frac{\delta}{2(\delta+1)}+\vartheta}(\1))$, $\mathbb{P}$-a.s, for $0<\vartheta<\frac{\e}{4}+\frac{\delta-1}{4(\delta+1)}$. Since $a\in\H_0^1(\1)$, it is immediate that $a\in\mathrm{H}^{\frac{\delta}{2(\delta+1)}+\vartheta}(\1)$.	Note that the support of the distribution of the processes $z\in\C([0,T];\H^{\frac{\delta}{2(\delta+1)}+\vartheta}(\1))$ is the closure of the set of functions $\int_0^tR(t-s)w(s)\d s, \ t\in[0,T],\ w\in\L^2(0,T;\H^{\frac{\delta}{2(\delta+1)}+\vartheta}(\1))$. Since $\H^{\frac{\delta}{2(\delta+1)}+\vartheta}(\1)\subset \L^{2(\delta+1)}(\1)$,   for arbitrary $\eta>0$, we have $$0<\mathbb{P}\left\{\sup_{t\in[0,T]}\|z(t)-\mathfrak{z}(t)\|_{\H^{\frac{\delta}{2(\delta+1)}+\vartheta}}<\eta\right\}\leq \mathbb{P}\left\{\sup_{t\in[0,T]}\|z(t)-\mathfrak{z}(t)\|_{\L^{2(\delta+1)}}<\eta\right\}.$$ Let us fix $$\upgamma=\eta+\sup_{t\in[0,T]}\|\mathfrak{z}(t)\|_{\L^{2(\delta+1)}}.$$ Then, we have 
		\begin{align*}
			0&<\mathbb{P}\left\{\sup_{t\in[0,T]}\|z(t)-\mathfrak{z}(t)\|_{\L^{2(\delta+1)}}<\eta\right\}\nonumber\\&\leq \mathbb{P}\left\{\sup_{t\in[0,T]}\|z(t)-\mathfrak{z}(t)\|_{\L^{2(\delta+1)}}<\eta\ \text{ and }\ \sup_{t\in[0,T]}\|z(t)\|_{\L^{2(\delta+1)}} \leq\upgamma \right\}.
		\end{align*}
		Let us now consider 
		\begin{align}\label{4.14}\nonumber
			&\P\left\{\|u(T)-b\|_{\L^2}^2<\eta\right\} \nonumber\\&= \P\left\{\|v(T)+z(T)-y(T)+y(T)-\mathfrak{z}(t)+\mathfrak{z}(t)-b\|_{\L^2}^2<\eta\right\} \nonumber\\& = \P\left\{\|v(T)-y(T)+z(T)-\mathfrak{z}(t)+x(T)-b\|_{\L^2}^2 <\eta\right\} \nonumber \\& \geq \P\left\{\|v(T)-y(T)\|_{\L^2}^2<\frac{\eta}{4}, \; \|z(T)-\mathfrak{z}(t)\|_{\L^2}^2 <\frac{\eta}{4},\;\|x(T)-b\|_{\L^2}^2 <\frac{\eta}{2} \right\} \nonumber\\& = 
			\P\left\{\|v(T)-y(T)\|_{\L^2}^2<\frac{\eta}{4},\; \|z(T)-\mathfrak{z}(t)\|_{\L^2}^2 <\frac{\eta}{4}\right\} \nonumber \\& \geq \P\left\{\sum_{j\in D}\sup_{t\in[0,T]}\|z(t)-\mathfrak{z}(t)\|_{\L^{2(\delta+1)}}^{j}<C(\nu,\alpha,\beta,\gamma,\upgamma,\delta,T,\eta)\right\} \nonumber\\&\geq \P\left\{\sum_{j\in D}\sup_{t\in[0,T]}\|z(t)-\mathfrak{z}(t)\|_{\L^{2(\delta+1)}}^{j}<C(\nu,\alpha,\beta,\gamma,\upgamma,\delta,T,\eta)\ \text{ and }\ \sup_{t\in[0,T]}\|z(t)\|_{\L^{2(\delta+1)}} \leq\upgamma  \right\} \nonumber \\& >0, 
		\end{align}where $ j\in D=\big\{ 2, 2(\delta+1) \big\}$,
		since $\eta>0$ can be chosen arbitrary so that the required result holds. 	Hence the transition semigroup $\mathrm{P}_t, t\geq0$, is  irreducible.
	\end{proof}

	\subsection{Strong Feller property}
	Let $u(\cdot,x)$ be the mild solution of \eqref{2.2}, which has been established in \cite{MTMSGBH}. We show that the corresponding transition semigroup $\mathrm{P}_t,\; t\geq0$, has the strong Feller property on the space $\L^2(\1)$. To prove this property, we take a modified version of SGBH equation. For this purpose, we define a cut-off function and a mollifier as, for any $R>0$ 
	\begin{equation}\label{4.15} \Phi(r)=
		\left\{
		\begin{aligned} 
			&1 \ \text{ for }\  r\in[0,R], 
			\\& 0 \ \text{ for } \ r\in [R+1,\infty),
		\end{aligned}
		\right.
	\end{equation}which is a $\C^1$ function defined on $[0,\infty)$. Now, we define a mollifier 
	\begin{align}\label{4.16}
		M_R(x)= x \Phi(\|x\|_{\L^2}),\; x\in\L^2(\1).
	\end{align}Also note that for any $x\in\L^2(\1)$, $M_R \in \C_b^1(\L^2(\1))$ and \begin{align*}
		\D_xM_R(x) = \Phi(\|x\|_{\L^2})\I+\frac{\Phi'(\|x\|_{\L^2})}{\|x\|_{\L^2}}x \otimes x, \; x\in \L^2(\1).
	\end{align*}
	\begin{proposition}\label{prop3.4}
		If $M_R$ is a mollifier defined by \eqref{4.16}, then the modified SGBH equation
		\begin{equation}\label{4.17}
			\left\{
			\begin{aligned}
				\d u&= \big\{-\nu\A u-\alpha\B(M_R(u))+\beta \c(M_R(u))\}\d t +\G\d\W(t),\\
				u(0)&=x,
			\end{aligned}
			\right.
		\end{equation}has a unique mild solution on  the time interval $[0,T]$. Moreover the transition semigroup $\mathrm{P}_t^R,\; t\geq 0$, has the strong Feller property.
	\end{proposition}\noindent
	\begin{proof} We prove  Proposition \ref{prop3.4} in two steps. In step 1, we put a remark on the existence and uniqueness of the mild solution to the system \eqref{4.17}. In step 2, we discuss the proof of strong Feller property of the corresponding transition semigroup.
		\vskip 0.2 cm
		\noindent
		\textbf{Step 1}: \emph{The existence and uniqueness of mild solution:} The proof of existence and uniqueness of solutions to \eqref{4.17}	is similar to that of Theorem \ref{thm2.6}, provided one can derive an a-priori bound of the form \eqref{214} (see \cite{MTMSGBH} also).
		\vskip 0.2 cm
		\noindent
		\textbf{Step 2}: \emph{Strong Feller property:}
		For $0<R<\infty$, we denote the directional derivative by $U^R(t)$ at $x$ in the direction of $h$ of the mapping $x\mapsto u^R(t,x)$ (where $u^R(\cdot)$ is the solution of the system \eqref{4.18}), that is, \begin{align*}
			U^R(t) = \big[\D_x u^R(t,x)\big]\cdot h ,
		\end{align*}for given $x,h \in \L^2(\1)$. Note that it is also the derivative of the mapping $x\mapsto u^R(t,x)= v^R(t,x)+z(t)$. Thus, $U^R$ is the solution of the first variation equation associated with the system \eqref{4.17} and is given by 
		\begin{equation}\label{4.18}
			\left\{
			\begin{aligned}
				\frac{\d U^R}{\d t}&=-\nu \A U^R-\alpha \partial_{\xi} \big(M_R^{\delta}(u)M_R'(u)U^R\big)+\beta(1+\gamma)(1+\delta)M_R^{\delta}(u)M_R'(u)U^R\\&\;\quad -\beta\gamma M_R'(u)U^R-\beta(2\delta+1)M_R^{2\delta}(u)M_R'(u)U^R ,\\
				U^R(0)&=h.
			\end{aligned}
			\right.
		\end{equation}We consider the mild  solution of \eqref{4.18}, and taking the $\L^2$-norm to find 
		\begin{align}\label{4.19}
			\|U^R(t)\|_{\L^2} &\leq \|R(t)h\|_{\L^2}+\alpha\int_{0}^{t}\|R(t-s)\partial_\xi\big(M_R^{\delta}(u)M_R'(u)U^R\big)\|_{\L^2}\d s \nonumber\\&\quad+\beta(1+\gamma)(1+\delta)\int_{0}^{t}\|R(t-s)M_R^{\delta}(u)M_R'(u)U^R\|_{\L^2}\d s \nonumber\\&\quad+\beta\gamma\int_{0}^{t}\|R(t-s)M_R'(u)U^R\|_{\L^2}\d s\nonumber\\&\quad+\beta(2\delta+1)\int_{0}^{t}\|R(t-s)M_R^{2\delta}(u)M_R'(u)U^R\|_{\L^2}\d s.
		\end{align} Applying the semigroup property (see \eqref{1.6} and \eqref{1.5}) on the terms of the right hand side of \eqref{4.19}, we get
		\begin{align*}
			\|R(t-s)\partial_\xi\big(M_R^{\delta}(u)M_R'(u)U^R\big)\|_{\L^2} &\leq C(t-s)^{-\frac{3}{4}}\|M_R^{\delta}(u)M'_R(u)U^R\|_{\L^1},\\
			\|R(t-s)M_R^{\delta}(u)M_R'(u)U^R\|_{\L^2} &\leq C(t-s)^{-\frac{1}{4}}\|M_R^{\delta}(u)M_R'(u)U^R\|_{\L^1},\\
			\|R(t-s)M_R'(u)U^R\|_{\L^2} &\leq C\|M_R'(u)U^R\|_{\L^2},\\
			\|R(t-s)M_R^{2\delta}(u)M_R'(u)U^R\|_{\L^2}& \leq C(t-s)^{-\frac{1}{4}}\|M_R^{2\delta}(u)M_R'(u)U^R\|_{\L^1}.
		\end{align*}Using the above estimates in \eqref{4.19}, we obtain 
		\begin{align}\label{4.20}\nonumber
			\|U^R(t)\|_{\L^2} &\leq \|h\|_{\L^2}+\alpha C\int_{0}^{t}(t-s)^{-\frac{3}{4}}\|M_R^{\delta}(u)M'_R(u)U^R\|_{\L^1}\d s \\&\quad+C\beta(1+\gamma)(1+\delta)\int_{0}^{t}(t-s)^{-\frac{1}{4}}\|M_R^{\delta}(u)M_R'(u)U^R\|_{\L^1}\d s \nonumber\\&\quad+C\beta\gamma \int_{0}^{t}\|M_R'(u)U^R\|_{\L^2}\d s +\beta(2\delta+1)\int_{0}^{t}(t-s)^{-\frac{1}{4}}\|M_R^{2\delta}(u)M_R'(u)U^R\|_{\L^1}\d s.
		\end{align} Now, using H\"older's inequality, for $x,z\in\L^2(\1)$, we have 
		\begin{equation}\label{4.21}
			\left\{
			\begin{aligned}
				\|M_R'(x)z\|_{\L^2} &\leq C\|z\|_{\L^2}, \\
				\|M_R^{\delta}(x)M'_R(x)z\|_{\L^1}&\leq \|M_R^{\delta}(x)\|_{\L^2}\|M_R'(x)z\|_{\L^2} \leq C\|z\|_{\L^2}, \\
				\|M_R^{2\delta}(x)M'_R(x)z\|_{\L^1} &\leq \|M_R^{2\delta}(x)\|_{\L^2}\|M_R'(x)z\|_{\L^2}\leq C\|z\|_{\L^2},
			\end{aligned}
			\right.
		\end{equation}where the constant $C$ depends on the choice of the cut-off function $\Phi_R$. Thus, one can conclude that the solution of \eqref{4.18} exists and belongs to $\C([0,T];\L^2(\1)),$ provided $T$ is sufficiently small. Substituting the bounds \eqref{4.21} in \eqref{4.20} and an application of Gronwall's lemma gives
		\begin{align*}
			\sup_{t\in[0,T]} \|U^R(t,x)\|_{\L^2} \leq C_T\|h\|_{\L^2}, \; \text{ for } h,x \in \L^2(\1),
		\end{align*}where $C_T$ is a non-random constant.
		From Bismut-Elworthy formula (Lemma 7.1.3, \cite{GDJZ}) and Theorem 7.1.1, \cite{GDJZ}, we obtain that the strong Feller property holds for a short time  and then with the help of semigroup property, we can extend that interval to $[0,+\infty)$.
	\end{proof}
	
	\begin{proposition}\label{prop4.3}
		The transition semigroup $\mathrm{P}_t,\; t\geq 0$, associated to the solution $u(\cdot)$ of \eqref{2.2} has the strong Feller property.
	\end{proposition}
	\begin{proof}
		For $0<R<\infty$, let $\mathrm{P}_t^R,\; t\geq 0$, be the associated semigroup to the solution $u^R(t,x)$ of the system \eqref{4.17}. Let us define 
		\begin{align*}
			\tau_{x}^R = \inf\{t\geq 0:\|u^R(t,x)\|_{\L^2}>R\}.
		\end{align*} It is clear form the definition of cut-off function \eqref{4.15} and \eqref{4.16} that 
		\begin{align*}
			u(t,x)=u^R(t,x),\  \text{ for all } \ t\leq \tau_{x}^R, \; x\in\L^2(\1).
		\end{align*}Let $\psi$ be any arbitrary function in $\B_b(\L^2(\1))$, then 
		\begin{align*}
			\big|\mathrm{P}_t^R\psi(x)-\mathrm{P}_t\psi(x)\big| &= \big|\E[\psi(u^R(t,x))]-\E[\psi(u(t,x))]\big|
			\\& = \big|\E[\psi(u^R(t,x))-\psi(u(t,x))]\chi_{\{\tau_{x}^R\leq t\}}\big|
			\\& \leq 2\sup_{z\in\L^2(\1)}|\psi(z)|\P\{\tau_{x}^R\leq t\}.
		\end{align*}We know that the functions $\mathrm{P}_t^R\psi$ are continuous for all $R>0 \text{ and } t>0$. Therefore it is enough to prove that for any $M>0$ and $t>0$,
		\begin{align}\label{4.22}
			\lim_{R\to\infty}\sup_{\|x\|_{\L^2}\leq M}\P \{\tau_{x}^R\leq t\}=0.
		\end{align}We have already set in section \ref{sec2} that $u(t)=v(t)+z(t),\; t\geq 0$, and a similar formulation corresponding to the system \eqref{4.17}  gives $u^R(t)=v^R(t)+z(t),\; t\geq 0$ and\begin{align*}
			\sup_{0\leq s\leq t}\|u^R(s,x)\|_{\L^2}^2 &\leq 2\sup_{0\leq s\leq t}\|v^R(s,x)\|_{\L^2}^2+2\sup_{0\leq s\leq t}\|z(s,x)\|_{\L^2}^2.
		\end{align*}
		Note that the estimate \eqref{214} is valid for the process $u^R(\cdot)$ also. Thus, if  $\|x\|_{\L^2}\leq M$, then we get
		\begin{align*}
			\|u^R(s,x)\|_{\L^2}^2 &\leq 
			C(\alpha,\beta,\gamma,\delta,\nu)\left(M^2+\int_0^t\|z(s)\|_{\L^{2(\delta+1)}}^{2(\delta+1)}\d s\right)+2\sup_{0\leq s\leq t}\|z(s,x)\|_{\L^2}^2.
		\end{align*}Since the right hand side of the above inequality is finite and independent of $x$, the equality \eqref{4.22} holds.
	\end{proof}

	\begin{theorem}\label{thm4.4}
		There exists a unique invariant measure $\mu$  for the transition semigroup $\mathrm{P}_t,\ t\geq 0$, corresponding to solutions of  \eqref{2.2} ($\delta\in[1,2)$ under assumption \eqref{1.3} and $\delta\in[1,2]$ under assumption \eqref{13}). Moreover $\mu$ is ergodic and strongly mixing.
	\end{theorem}
	\begin{proof}
		From Theorems \ref{thEx1} and  \ref{thEx2}, we infer the existence of an invariant measure $\mu$  for the transition semigroup $\mathrm{P}_t,\ t\geq 0$, corresponding to solutions of  \eqref{2.2}. Since the semigroup is strong Feller (Proposition \ref{prop4.3}) and irreducible (Proposition \ref{prop4.1}), the uniqueness of invariant measures follows by Doob's theorem (Theorem 4.2.1, \cite{GDJZ}). Since $\mu$ is a unique invariant measure, ergodicity follows from Theorem 3.2.6, \cite{GDJZ} and strongly mixing is an immediate consequence of Theorem 4.2.1, \cite{GDJZ}. 
	\end{proof}

	\section{Large Deviation Principle}\label{sec5}\setcounter{equation}{0}
	In this section, we prove the LDP w.r.t. the topology $\tau$ and Donsker-Varadhan LDP of the occupation measure for the SGBH equation \eqref{1p1} with $x\in\L^2(\1)$ under the assumption \eqref{13}.  Our goal is to derive the LDP of the occupation measure $L_t$ of the solution $u(\cdot)$ to the system \eqref{1p1}, where the occupation measure is defined as
	\begin{align*}
		L_t(A):= \frac{1}{t}\int_{0}^{t}\delta_{u(s)}(A)\d s, \text{ for all } A \in \mathcal{B}(\L^2(\1)),
	\end{align*}where $\delta_a$ denotes the Dirac measure concentrated at point $a$, and $\mathcal{B}(\L^2(\1))$ represents the Borelian $\sigma$-field in $\L^2(\1)$. The Donsker-Varadhan LDP for stochastic Burgers equation, 2D stochastic Navier-Stokes equations, stochastic convective Brinkman-Forchheimer equations,  is obtained in  \cite{MG2,MG1,AKMTM}, respectively, and we follow these works to obtain Donsker-Varadhan LDP  for our problem. All the results obtained in this section are true  for $\delta\in[1,2]$ with  any $\nu,\alpha,\beta>0, \gamma\in(0,1)$  (see Theorem \ref{thm2.6}). Let us state our main result of this section:
	\begin{theorem}\label{thrm5.1}
		Assume that $\Tr(\G\G^*)<\infty$ and \eqref{13} holds. Let $ 0<\lambda_0 <\frac{\pi^2\nu}{2\|\Q\|_{\mathcal{L}(\L^2(\1))}} $, where $\|\Q\|_{\mathcal{L}(\L^2(\1))}$ is the norm of $\Q:=\G\G^*$ as an operator in $\L^2(\1)$ and 
		\begin{align}\label{5.1}
			\Psi(x)=e^{\lambda_0\|x\|_{\L^2}^2},\;\; \mathcal{M}_{\lambda_0,R}:= \bigg\{\varrho \in \M_1(\L^2(\1)): \int_{\L^2(\1)}\Psi(x)\varrho(\d x)\leq R \bigg\}.
		\end{align}
		The family $ \P_{\varrho}(L_{T} \in \cdot) \text{ as } T \rightarrow +\infty  $ satisfies the LDP w.r.t. the topology $\tau$, with speed $T$ and rate function $\J$ uniformly for any initial measure $\varrho$ in $\mathcal{M}_{ \lambda_{0},R}, \text{ where } R >1$ is any fixed number. Here the rate function $ \J:\M_{1}(\L^2(\1)) \rightarrow [0,+\infty] $ is the level-2 entropy of Donsker-Varadhan (defined in \eqref{5.7} below). Moreover, we have 
		\begin{itemize}
			\item[(i)] $\J$ is a good rate function on $\M_{1}(\L^2(\1))$ equipped with the topology $\tau$ of the convergence against bounded and Borelian functions, that is, $[J\leq a]$ is $\tau$-compact for every $a\in\R^+$.
			\item[(ii)] For all open sets $\G \text{ in } \M_{1}(\L^2(\1))$ w.r.t. the topology $\tau$,
			\begin{align}\label{5.2}
				\liminf_{T\to\infty} \frac{1}{T} \log \inf_{\varrho \in \mathcal{M}_{ \lambda_{0},R}}\P_{\varrho}\left\{L_{T}\in \G\right\} \geq -\inf_{\G} \J.\end{align}
			\item[(iii)] For all closed sets $\F \text{ in } \M_{1}(\L^2(\1))$ w.r.t. the topology $\tau$,
			\begin{align}\label{5.3}
				\limsup_{T\to\infty} \frac{1}{T} \log \sup_{\varrho \in \mathcal{M}_{ \lambda_{0},R}}\P_{\varrho}\left\{L_{T}\in \F\right\}\leq -\inf_{\F} \J.\end{align} 
		\end{itemize}
		Furthermore, we have 
		\begin{align}\label{5.4}
			\J(\varrho) < +\infty \Rightarrow \varrho\ll \mu, \quad \varrho(\H_0^1(\mathcal{O})) =1\  \text{ and }\  \int_{\H_0^1} \|\partial_{\xi}x\|_{\L^2}^{2}\d\varrho < +\infty, 
		\end{align} 
		where $\mu$ is the unique invariant probability measure of $u(t,\cdot)$.
	\end{theorem}
	\begin{corollary}\label{cor4.2}
		Let $(\mathbb{B}, \|\cdot\|_{\mathbb{B}})$ be a separable Banach space, and $\psi: \H_0^1(\1) \to \mathbb{B}$ be a measurable function, bounded on balls $\{ x:\|\A^{\frac{1}{2}}x\|_{\L^2(\1)} \leq R\}$ and satisfying 
		\begin{align}\label{5.5} \lim_{\|\A^{\frac{1}{2}}x\|_{\L^2} \to \infty}  \frac{\|\psi(x)\|_{\mathbb{B}}}{\|\A^{\frac{1}{2}}x\|^{2}_{\L^2}} = 0.
		\end{align}
		Then $\mathbb{P}_{\varrho}(L_{T}(\psi) \in \cdot)$ satisfies the Donsker-Varadhan LDP on $\mathbb{B}$ with speed $T$ and the rate function $\I_{\psi}$ given by $$ \I_{\psi}(y) = \inf \{\J(\varrho): \J(\varrho) < +\infty,\ \varrho(\psi) = y\}, \ \text{ for all }\ y \in \mathbb{B},$$
		uniformly over initial distributions $\varrho \in \mathcal{M}_{\lambda_{0}, R}$ (for any $R >1$).
	\end{corollary}
	\begin{example}Let us provide some examples of the assumptions $\Tr(\Q)<\infty$ and \eqref{13} (cf. \label{rem4.2}\cite{MG2,MG1}). 
		\begin{itemize}	
			\item[(i)]  We know that an $\L^2(\1)$-valued  cylindrical Wiener process $\{\W(t) : 0\leq t\leq T\}$  on $(\Omega,\mathscr{F},\{\mathscr{F}_t\}_{t\geq 0},\mathbb{P})$  can be expressed as  $\W(t)=\sum\limits_{k=1}^\infty  \beta_k(t)e_k,$ where $\beta_k(t), k\in \mathbb{N}$ are independent, one dimensional Brownian motions on the space $(\Omega,\mathscr{F},\{\mathscr{F}_t\}_{t\geq 0},\mathbb{P})$ (cf. \cite{DaZ}). Let us define $\G e_k=\sigma_k e_k$, for $k=1,2,\ldots$, so that $$\G\W(t)=\sum_{k=1}^{\infty}\sigma_k\beta_k(t)e_k.$$ We also know that the eigenvalues of  the Laplacian $\lambda_k \sim k^{2}$. Thus the condition given in \eqref{13} becomes  $$\frac{c}{k}\leq \sigma_k\leq \frac{C}{k^{\frac{1}{2}+\e}},$$ for any two positive constants $c \text{ and } C$ and some $\e>0$.
			\item[(ii)] An another example of noise for which our assumptions hold is $\G:=\A^{-\beta}\F,$ where $\F$ is any linear bounded and invertible operator on $\L^2(\1)$ and $\frac{1}{4}<\beta<\frac{1}{2}$. 
		\end{itemize}
	\end{example}
	\begin{remark}
		1. The class \eqref{5.1} of permissible initial distributions for the uniform LDP is sufficiently rich. For example, choosing $R$ large enough, it accommodates all the Dirac probability measure $\delta_{x}$ with $x$ in any ball of $\L^2(\1)$. 
		
		2. The LDP w.r.t. the topology $\tau$ is stronger than that w.r.t. the usual weak convergence topology as in Donsker-Varadhan \cite{MDD}.
		
		3. The assumption \eqref{13} plays an important role in Theorem \ref{thrm5.1}. If the noise acts only a finite number of modes (that is, $\sigma_k=0$ after a finite number $N$), in the first part of Example \ref{rem4.2} as in Kolmogorov's turbulence theory, we believe that the LDP w.r.t. the $\tau$-topology is false.
	\end{remark}

	The existence  of mild solution of our problem \eqref{2.2} is proved in Theorem \ref{thm2.7} and strong solution is established in Theorem \ref{thm2.6}.  We have already defined the transition semigroup corresponding to the solution of \eqref{2.2} by 
	\begin{align*} 
		\mathrm{P}_t\psi(x)=\E[\psi(u(t,x))]=\E^{x}[\psi(u(t))],\ \text{ for all }\ \psi\in\mathrm{B}_b(\L^2(\1)).
	\end{align*}
	We have already proved that the transition semigroup $\mathrm{P}_t,\; t\geq 0$ is irreducible, satisfies the strong Feller property (Propositions \ref{prop4.1}, \ref{prop4.3}), and it admits a unique invariant measure $\mu$ (Theorems \ref{thEx1}, \ref{thm4.4}). 	
	\subsection{General results on LDP}
	In this section, we provide some necessary notations, basic definitions and give some results from \cite{LW2} on large deviations for the Markov process.
	Consider the $\L^2(\1)$-valued continuous Markov process, $$(\Omega,\{\mathscr{F}_{t}\}_{t\geq 0},\mathscr{F},\{u_{t}\}_{t\geq 0}, \{\mathbb{P}_{x}\}_{x \in \L^2(\1)}),$$ 
	whose semigroup of Markov transition kernel is denoted by $\{\P_{t}(x,\d \y)\}_{t\geq 0}$, where 
	\begin{itemize}
		\item  $\Omega = \C(\R^{+};\L^2(\1))$ is the space of continuous functions from $\R^{+}$ to $\L^2(\1)$ equipped with the compact convergence topology,
		\item  the natural filtration is $\mathscr{F}_{t} = \sigma\{u(s): 0 \leq s \leq t\}$ for any $t\geq 0$ and $\mathscr{F} = \sigma\{u(s): 0 \leq s\}$.
		\item $\P_{x}\{u(0)=x\}=1$,
	\end{itemize}
	As usual, we denote the law of Markov process with the initial state $x \in \L^2(\1) $ by $ \mathbb{P}_{x}$, and for any initial measure $\varrho $ on $\L^2(\1)$, we define $\mathbb{P}_{\varrho}(\cdot)=\int_{\L^2(\1)}\mathbb{P}_{x}(\cdot)\varrho (\d x)$. 
	The empirical measure of level-3 is given by $$ R_{t} := \frac{1}{t}\int_{0}^{t} \delta_{\theta_{s}u} \d s, $$ where $(\theta_{s}u)(t)= u(s+t),$ for all $t,s \geq 0 $ are the shifts on $\Omega$. Therefore, $R_{t}$ is a random element of $\M_{1}(\Omega)$, the space of probability measures on $\Omega$.	The level-3 entropy functional of Donsker-Varadhan $H:\M_{1}(\Omega) \to [0,+\infty]$ is defined by 
	\begin{equation}\label{5.6}
		\begin{aligned}
			H(Q):= \left\{\begin{array}{cl}
				\mathbb{E}^{\bar{Q}}h_{\mathscr{F}^{0}_{1}} \left(\bar{Q}_{\omega(-\infty,0]};\mathbb{P}_{\omega(0)}\right), &  \text{ if } Q \in \M^{s}_{1}(\Omega), \\ 
				+\infty , & \text{ otherwise},
			\end{array}\right. \end{aligned}
	\end{equation}
	where \begin{itemize}
		\item $\M^{s}_{1}(\Omega)$ is the subspace of $\M_{1}(\Omega)$, whose element are moreover stationary; 
		\item $\bar{Q}$ is the unique stationary extension of $Q \in \M^{s}_{1}(\Omega)$ to $\bar{\Omega} := \C(\R;\L^2(\1))$;  $\mathscr{F}^{s}_{t} = \sigma \{u(r): s \leq r \leq t \}$, for all $\ s,  t \in \R, \ s \leq t$;
		\item $\bar{Q}_{\omega(-\infty,t]}$ is the regular conditional distribution of $\bar{Q}$ knowing $\mathscr{F}^{-\infty}_{t}$;
		\item   $h_{\mathcal{G}}(\varrho,\mu)$ is the usual relative entropy or Kullback information of $\varrho $ w.r.t. $\mu$ restricted to the $\sigma $-field $\mathcal{G} $, is given by 
		\begin{equation*}
			\begin{aligned}h_{\mathcal{G}}(\varrho;\mu):=\left\{ \begin{array}{cl}		\int \frac{\d\varrho}{\d\mu}\big|_{\mathcal{G}} \log \big(\frac{\d\varrho}{\d\mu}\big|_{\mathcal{G}}\big)\d\mu, & \text{ if } \varrho \ll \mu \text{ on } \mathcal{G}, \\
					+\infty, & \text{ otherwise}.
				\end{array}\right. \end{aligned}
		\end{equation*}
	\end{itemize}
	The level-2 entropy functional $\J: \M_{1}(\L^2(\1)) \rightarrow [0,\infty],$ which governs the LDP in our main Theorem \ref{thrm5.1} is 
	\begin{align}\label{5.7}
		\J(\mu) = \inf \{H(Q): Q \in \M^{s}_{1}(\Omega) \text{ and }  Q_{0}= \mu \}, \quad \mu \in \M_{1}(\L^2(\1)),\end{align}
	where $ Q_{0}(\cdot) =Q(u(0) \in \cdot)$ is the marginal law at $t=0$.
	As introduced in \cite{LW2}, we define the restriction of the Donsker-Varadhan entropy to the $\mu$ component, by 
	\begin{equation*}
		\begin{aligned}
			H_{\mu}(Q):= \left\{ \begin{array}{cl}H(Q), &\text{ if } Q_0 \ll \mu,
				\\ \infty , &\text{ otherwise},
			\end{array}\right.
		\end{aligned}
	\end{equation*} and for level-2 entropy functional 
	\begin{equation*}
		\begin{aligned}
			\J_{\mu}(\varrho):= \left\{ \begin{array}{cl}\J(\varrho), & \text{ if }\varrho\ll \mu,
				\\ \infty , & \text{ otherwise },
			\end{array}\right.
		\end{aligned}
	\end{equation*}
	A proof of the following result is available as Lemma 3.1, \cite{MG2}, and hence we omit it here.
	\begin{lemma}\label{lem4.5} 
		For our system $\J(\varrho)<\infty \Rightarrow \varrho \ll \mu $. Moreover, $\J=\J_{\mu}$ on $\M_1(\L^2(\1))$ and $[\J=0]=\{\mu\}$.
	\end{lemma}
	\subsection{Exponential estimates for the solution} In this subsection, we prove a crucial exponential estimate for the solution $u(\cdot)$ to the SGBH equation, which will be helpful to establish the LDP results.
	We need the following result to discuss about the proof of Proposition \ref{prop4.10}. First recall the finite dimensional Galerkin approximation, that is, 
	\begin{equation}\label{5.10}
		\left\{
		\begin{aligned}
			\d  u_n(t)&=\{-\nu\A u_n(t)-\alpha\B_n (u_n(t))+\beta \c_n(u_n(t))\}\d t + \G_n \d \W(t), \; t\in(0,T), \\
			u_n(0)&= x_n:=\Pi_nx,
		\end{aligned}
		\right.
	\end{equation}
	and it satisfies the following a-priori energy estimate: 
	\begin{align}\label{5.11}\nonumber
		&\E\bigg[\sup_{t\in[0,T]}\|u_n(t)\|_{\L^2}^2+\nu\int_{0}^{T}\|\partial_{\xi} u_n(s)\|_{\L^2}^2\d s +\beta\int_{0}^{T} \|u_{n}(s)\|_{\L^{2(\delta+1)}}^{2(\delta+1)}\d s \bigg] \\
		&\leq C(\|x\|_{\L^2}^2+\Tr(\G\G^*)T).
	\end{align}
	\begin{lemma}\label{lem5.12}
		Let $u_n(\cdot) \text{ and } u(\cdot)$ be the solutions of the systems \eqref{5.10} and \eqref{2.2}, respectively. Then, we have 
		\begin{align}\label{5.12}\nonumber
			&\|u_n(t)\|_{\L^2}^2+2\nu\int_{0}^{t}\|\partial_{\xi} u_n(s)\|_{\L^2}^2\d s +\beta\int_{0}^{t} \|u_{n}(s)\|_{\L^{2(\delta+1)}}^{2(\delta+1)}\d s \\&\xrightarrow{a.s.} \|u(t)\|_{\L^2}^2+2\nu\int_{0}^{t}\|\partial_{\xi} u(s)\|_{\L^2}^2\d s +\beta\int_{0}^{t} \|u(s)\|_{\L^{2(\delta+1)}}^{2(\delta+1)}\d s,
		\end{align} for all $t\in[0,T]$.
	\end{lemma}
	\begin{proof}
		Using the Banach-Alaoglu theorem in \eqref{5.11}, we can extract a subsequence $\{u_n\}$ (for simplicity still denoting by $\{u_n\}$) such that 
		\begin{equation}\label{5.13}
			\left\{
			\begin{aligned}
				u_n&\xrightarrow{w^*}u\ \text{ in }\mathrm{L}^2(\Omega;\mathrm{L}^{\infty}(0,T;\L^2(\1))),\\
				u_n&\xrightarrow{w}u\ \text{ in }\mathrm{L}^2(\Omega;\mathrm{L}^{2}(0,T;\H_0^1(\1))),\\
				u_n&\xrightarrow{w}u\ \text{ in }\L^{2(\delta+1)}(\Omega;\L^{2(\delta+1)}(0,T;\L^{2(\delta+1)}(\1))),
			\end{aligned}\right.
		\end{equation}where $u(\cdot)$ denotes the strong solution of the system \eqref{2.2}. Since the solution of \eqref{5.10} is unique,  the whole sequence converges to $u(\cdot)$.
		A calculation similar to \eqref{Expectation} yields 
		\begin{align}\label{5.16}\nonumber
			&\E\bigg[\|u_n(t)\|_{\L^2}^2+2\nu\int_{0}^{t}\|\partial_{\xi} u_n(s)\|_{\L^2}^2\d s +\frac{\beta}{2}\int_{0}^{t} \|u_{n}(s)\|_{\L^{2(\delta+1)}}^{2(\delta+1)}\d s\bigg] \\&\leq \|x_{n}\|^{2}_{\L^2}  +(\Tr(\G_{n}\G^{*}_{n})+C(\beta,\delta) )t,
		\end{align}where the constant $C(\beta,\delta)=\big(\frac{2}{\beta(\delta+1)}\big)^{\frac{1}{\delta}}\frac{\delta}{\delta+1}$.
		Similarly, for the strong solution $u(\cdot)$ of system \eqref{2.2}, we get 
		\begin{align}\label{5.17}\nonumber
			&\E\bigg[\|u(t)\|_{\L^2}^2+2\nu\int_{0}^{t}\|\partial_{\xi} u(s)\|_{\L^2}^2\d s +\frac{\beta}{2}\int_{0}^{t} \|u(s)\|_{\L^{2(\delta+1)}}^{2(\delta+1)}\d s\bigg] \\&\leq \|x\|^{2}_{\L^2}  +(\Tr(\G\G^{*})+C(\beta,\delta))t ,
		\end{align}for all $t\in[0,T]$. From \eqref{5.16} and \eqref{5.17}, we obtain 
		\begin{align}\label{5.18}\nonumber
			\bigg|&\E\bigg[\|u_n(t)\|_{\L^2}^2+2\nu\int_{0}^{t}\|\partial_{\xi} u_n(s)\|_{\L^2}^2\d s +\frac{\beta}{2}\int_{0}^{t} \|u_{n}(s)\|_{\L^{2(\delta+1)}}^{2(\delta+1)}\d s\bigg] \nonumber\\&\quad-\E\bigg[\|u_n(t)\|_{\L^2}^2+2\nu\int_{0}^{t}\|\partial_{\xi} u_n(s)\|_{\L^2}^2\d s +\frac{\beta}{2}\int_{0}^{t} \|u_{n}(s)\|_{\L^{2(\delta+1)}}^{2(\delta+1)}\d s\bigg]\bigg|\nonumber\\& \leq \big|\|x_{n}\|^{2}_{\L^2}-\|x\|^{2}_{\L^2}+\Tr(\G_{n}\G^{*}_{n})t-\Tr(\G\G^{*})t\big| \nonumber\\& \leq \big|\|x_{n}\|_{\L^2}-\|x\|_{\L^2}\big|\big|\|x_{n}\|_{\L^2}+\|x\|_{\L^2}\big|+\big|\Tr(\G_{n}\G^{*}_{n})t-\Tr(\G\G^{*})t\big|.
		\end{align}It is easy to deduce that 
		\begin{align*}
			\big|\|x_{n}\|_{\L^2}-\|x\|_{\L^2}\big| \leq \|x_n-x\|_{\L^2}=\bigg(\sum_{j=n+1}^{\infty}|(x,e_j)|^2\bigg)^{\frac{1}{2}}\to 0 \text{ as } n\to \infty,
		\end{align*} and 
		\begin{align*}
			\Tr(\G_n\G_n^*) -\Tr(\G\G^*) = \Tr (\Pi_n\G\G^*-\G\G^*)\leq \|\Pi_n-\I\|_{\mathcal{L}(\L^2(\1))}\Tr(\G\G^*) \to 0 \text{ as } n \to \infty.
		\end{align*}Passing $n\to\infty$ in \eqref{5.18} by using the above convergences, we find 
		\begin{align*}
			&\E\bigg[\|u_n(t)\|_{\L^2}^2+2\nu\int_{0}^{t}\|\partial_{\xi} u_n(s)\|_{\L^2}^2\d s +\frac{\beta}{2}\int_{0}^{t} \|u_{n}(s)\|_{\L^{2(\delta+1)}}^{2(\delta+1)}\d s\bigg]\\& \to \E\bigg[\|u(t)\|_{\L^2}^2+2\nu\int_{0}^{t}\|\partial_{\xi} u(s)\|_{\L^2}^2\d s +\frac{\beta}{2}\int_{0}^{t} \|u(s)\|_{\L^{2(\delta+1)}}^{2(\delta+1)}\d s\bigg],
		\end{align*}as $n\to\infty$. From the above convergence, one can extract an a.s. convergent  subsequence $\{u_{n_k}\}$ of $\{u_n\}$, that is, 
		\begin{align*}
			&\|u_{n_k}(t)\|_{\L^2}^2+2\nu\int_{0}^{t}\|\partial_{\xi} u_{n_k}(s)\|_{\L^2}^2\d s +\frac{\beta}{2}\int_{0}^{t} \|u_{n_k}(s)\|_{\L^{2(\delta+1)}}^{2(\delta+1)}\d s\\& \xrightarrow{a.s.} \|u(t)\|_{\L^2}^2+2\nu\int_{0}^{t}\|\partial_{\xi} u(s)\|_{\L^2}^2\d s +\frac{\beta}{2}\int_{0}^{t} \|u(s)\|_{\L^{2(\delta+1)}}^{2(\delta+1)}\d s \text{ as } k\to\infty,
		\end{align*} for all $t\in[0,T]$. The  above convergence holds for the original sequence   $\{u_n(\cdot)\}$, since $u_n(\cdot),\; u(\cdot)$ are the unique solutions of \eqref{5.10} and \eqref{2.2}, respectively. 
	\end{proof}
	\begin{proposition}\label{prop4.10}
		For any $0<\lambda_0 < \frac{\pi^2\nu}{2\|\Q\|_{\mathcal{L}(\L^2(\1))}}$, where $\|Q\|_{\mathcal{L}(\L^2(\1))}$ is the norm of $\Q$ as an operator in $\L^2(\1)$ and for any $x\in\L^2(\1)$, the process $u(\cdot)$ satisfies the following estimates:
		\begin{align}\label{5.19}\nonumber
			\E^{\x}\bigg\{\exp\bigg(\lambda_0\|u(t)\|_{\L^2}^2&+\lambda_0\nu \int_{0}^{t}\|\partial_{\xi}u(s)\|_{\L^2}^2\d s +\frac{\lambda_0\beta}{2}\int_{0}^{t}\|u(s)\|_{\L^{2(\delta+1)}}^{2(\delta+1)}\d s\bigg) \bigg\} \\& \leq e^{\lambda_0\|x_0\|_{\L^2}^2+\lambda_0t(\Tr (\Q)+C)},
		\end{align}where the constant $C=\big(\frac{2}{\beta(\delta+1)}\big)^{\frac{1}{\delta}}\frac{\delta}{\delta+1}$. In particular, the following estimates hold
		\begin{align}\label{5.20}
			\E^{x}\big\{\exp\big(\lambda_0\|u(t)\|_{\L^2}^2\big)\big\}&\leq e^{\lambda_0\|x_0\|_{\L^2}^2+\lambda_0t(\Tr (\Q)+C)},\\
			\label{5.21}
			\E^{x}\bigg\{\exp\bigg(\lambda_0\nu \int_{0}^{t}\|\partial_{\xi}u(s)\|_{\L^2}^2\d s \bigg)\bigg\} &\leq e^{\lambda_0\|x_0\|_{\L^2}^2+\lambda_0t(\Tr (\Q)+C)},\\ \label{5.22}
			\E^{x}\bigg\{\exp\bigg(\lambda_0\beta\int_{0}^{t}\|u(s)\|_{\L^{2(\delta+1)}}^{2(\delta+1)}\d s \bigg)\bigg\}&\leq e^{\lambda_0\|x_0\|_{\L^2}^2+\lambda_0t(\Tr (\Q)+C)}.
		\end{align}
	\end{proposition}
	\begin{proof}
		First we establish the result for the finite dimensional system \eqref{5.10}, and then  we pass the limit as $n\to\infty$. 	From \eqref{5.15}, we have 	\begin{align}\label{5.23}
			&\|u_{n}(t)\|^{2}_{\L^2}+ 2\nu \int_{0}^{t} \|\partial_{\xi} u_{n}(s)\|^{2}_{\L^2}\d s+ \frac{\beta}{2} \int_{0}^{t} \|u_{n}(s)\|_{\L^{2(\delta+1)}}^{2(\delta+1)}\d s\nonumber\\&\leq  \|x_{n}\|^{2}_{\L^2}  +(\Tr(\G_{n}\G^{*}_{n})+ C)t+2\int_{0}^{t}	( \G_{n} \d \W(s), u_{n}(s)) ,
		\end{align}$\P$-a.s. 	Let us define
		\begin{align*}
			Z_n(t) :=	\|u_{n}(t)\|^{2}_{\L^2}+ \nu \int_{0}^{t} \|\partial_{\xi}u_{n}(s)\|^{2}_{\L^2}\d s+ \frac{\beta}{2} \int_{0}^{t} \|u_{n}(s)\|_{\L^{2(\delta+1)}}^{2(\delta+1)}\d s.
		\end{align*} Then from \eqref{5.23}, we find
		\begin{align*}
			Z_n(t)\leq \|x_n\|_{\L^2}^2-\nu\int_{0}^{t}\|\partial_{\xi} u_n(s)\|_{\L^2}^2+ (\Tr(\G_{n}\G^{*}_{n})+C)t +2\int_{0}^{t}	( \G_{n} \d \W(s),u_n(s))
		\end{align*}$\P$-a.s., for all $t\in[0,T]$. Applying It\^o's formula to the process $e^{\lambda_0Z_n(t)}$, using the chain rule, we obtain
		\begin{align}\label{5.24}\nonumber
			\d (e^{\lambda_0Z_n(t)})&= e^{\lambda_0Z_n(t)}\bigg[\lambda_0\d Z_n(t)+\frac{\lambda_0}{2}\d[Z_n,Z_n]_t\bigg] \nonumber\\& \leq \lambda_0e^{\lambda_0Z_n(t)} \big[-\nu\|\partial_{\xi}u_n(t)\|_{\L^2}^2+\Tr(\Q_n)+C+2\lambda_0\|\G^*u_n(t)\|_{\L^2}^2\big]\d t\nonumber\\&\quad +2\lambda_0e^{\lambda_0Z_n(t)}(\G_n\d\W(t),u_n(t)).
		\end{align}The following inequalities are easy to obtain: \begin{align}\label{5.25}
			\Tr(\Q_n)\leq \Tr(\Q),\; \|x_n\|_{\L^2}\leq \|x\|_{\L^2}.
		\end{align}Also we have
		\begin{align}\label{5.26}
			\|\G_n^*u_n\|_{\L^2}^2\leq \|\G^*_{n}\|^{2}_{\mathcal{L}(\L^2(\1))} \|u_{n}(t)\|^{2}_{\L^2} \leq \|\Q\|_{\mathcal{L}(\L^2(\1))}\|u_{n}(t)\|^{2}_{\L^2},
		\end{align} and by Poincar\'e's inequality, we get 
		\begin{align}\label{5.27}
			\|u_{n}(t)\|^{2}_{\L^2}\leq \frac{1}{\pi^2} \|\partial_{\xi} u_{n}(t)\|^{2}_{\L^2}.\end{align}
		Using \eqref{5.25}-\eqref{5.27} in \eqref{5.24}, we obtain 
		\begin{align*}
			\d (e^{\lambda_0Z_n(t)})&\leq \lambda_0e^{\lambda_0Z_n(t)}\bigg[-\nu\|\partial_{\xi}u_n(t)\|_{\L^2}^2+\Tr(\Q)+C+\frac{2\lambda_0}{\pi^2}\|\Q\|_{\mathcal{L}(\L^2(\1))}\|\partial_{\xi}u_{n}(t)\|^{2}_{\L^2}\bigg]\\&\quad+2\lambda_0e^{\lambda_0Z_n(t)}(\G_n\d\W(t),u_n(t)).
		\end{align*}Integrating the above inequality from $0$ to $t$, we get
		\begin{align}\label{5.28}\nonumber
			e^{\lambda_0Z_n(t)} &\leq e^{\lambda_0\|x\|_{\L^2}^2}+\lambda_0(\Tr(\Q)+C)\int_{0}^{t}e^{\lambda_0Z_n(s)}\d s\\ &\quad+\lambda_0\bigg[e^{\lambda_0Z_n(t)}\int_{0}^{t}\bigg(\frac{2\lambda_0}{\pi^2}\|\Q\|_{\mathcal{L}(\L^2(\1))}-\nu\bigg)\|\partial_{\xi}u_n(t)\|_{\L^2}^2\d s\bigg]\nonumber\\&\quad+2\lambda_0e^{\lambda_0Z_n(t)}(\G_n\d\W(t),u_n(t)).
		\end{align}Taking expectation on both side, we deduce
		\begin{align*}
			\E^{x}\big(e^{\lambda_0Z_n(t)}\big) &\leq e^{\lambda_0\|x\|_{\L^2}^2}+\lambda_0(\Tr(\Q)+C)\int_{0}^{t}\E^{x}\big(e^{\lambda_0Z_n(s)}\big)\d s\\&\quad+\lambda_0\E^{x}\bigg[e^{\lambda_0Z_n(t)}\int_{0}^{t}\bigg(\frac{2\lambda_0}{\pi^2}\|\Q\|_{\mathcal{L}(\L^2(\1))}-\nu\bigg)\|\partial_{\xi}u_n(t)\|_{\L^2}^2\d s\bigg], 
		\end{align*} since the final term in \eqref{5.28} is a local martingale hence the expectation is zero. 
		Let us now choose $0<\lambda_0 < \frac{\pi^2\nu}{2\|\Q\|_{\mathcal{L}(\L^2(\1))}}$, so that the third term in the right hand side of the above inequality is  negative and  we get
		\begin{align*}
			\E^{x}\big(e^{\lambda_0Z_n(t)}\big)\leq e^{\lambda_0\|x\|_{\L^2}^2}+\lambda_0(\Tr(\Q)+C)\int_{0}^{t}\E^{x}\big(e^{\lambda_0Z_n(s)}\d s\big).
		\end{align*}An application of  Gronwall's lemma in above inequality provides\begin{align*}
			\E^{x}\big(e^{\lambda_0Z_n(t)}\big) \leq e^{\lambda_0\|x\|_{\L^2}^2}e^{\lambda_0t(\Tr(\Q)+C)}.
		\end{align*} Letting $n\to\infty$ in the above inequality for $0<\lambda_0 < \frac{\pi^2\nu}{2\|\Q\|_{\mathcal{L}(\L^2(\1))}}$, with the help of Lemma \ref{lem5.12},   we obtain for any $t\in[0,T]$,
		\begin{align*}
			\E^{x}\big(e^{\lambda_0Z(t)}\big) \leq e^{\lambda_0\|x\|_{\L^2}^2}e^{\lambda_0t(\Tr(\Q)+C)},
		\end{align*}and the proof  is completed.
	\end{proof}

	\subsection{Uniform upper bound for the $\tau$- topology}
	To prove the upper bound \eqref{5.3} in our main Theorem \ref{thrm5.1},  we use  the  criterion of  hyper-exponential recurrence established in  Theorem 2.1, \cite{LW1}, for a general polish space $\mathrm{E}$. The following result is a slight extension of the result in \cite{LW1}, to a uniform LDP over a non-empty family of initial measures (cf. \cite{MG1}). To use this result we require two properties of the associated semigroup, strong Feller and irreducibility.
	\begin{lemma}[Theorem 2.1, \cite{LW1}, or Lemma 6.1, \cite{MG2}]\label{lem5.1}
		For a subset $K$ in $\L^2(\1)$, let us define $\tau_k:=\inf\{t\geq 0: u(t)\in K\}$ and  $\tau^{(1)}_{K}:= \inf \{t \geq 1;u(t) \in K\}$. If for any $\lambda>0$, there exists a compact subset $K$ in $\L^2(\1)$ such that \begin{align}\label{6.1}
			\sup_{\varrho\in \mathcal{M}_{ \lambda_0,R}}\E^{\varrho}[e^{\lambda\tau_k}]<\infty,
		\end{align} and 
		\begin{align}\label{6.2}
			\sup_{x\in K}\E^{x}[e^{\lambda\tau_k^{(1)}}]<\infty, 
		\end{align}then $[\J\leq a]$ is $\tau$-compact for every $a\in\R^+,$ and the upper bound \eqref{5.3} uniform on $\mathcal{M}_{ \lambda_0,R}$ for the $\tau$-topology holds true.
	\end{lemma}
	To prove the upper bound \eqref{5.3}, it is enough to show the estimates \eqref{6.1} and \eqref{6.2} holds for our model. For that we choose a compact subset $K\subset \L^2(\1)$ given by 
	\begin{align}\label{6.3}
		K:= \big\{x \in \H_0^1(\1); \|\partial_{\xi}x\|_{\L^2}\leq M \big\},
	\end{align}where $M$ is the finite real number, which will be fixed later. Using the definition of occupation measure for $n\geq 2$, we obtain 
	\begin{align*}
		\mathbb{P}_{\varrho}\left\{\tau^{(1)}_{K}>n\right\}\leq 	\mathbb{P}_{\varrho}\left\{L_{n}(K)\leq \frac{1}{n}\right\} = 	\mathbb{P}_{\varrho}\left\{L_{n}(K^{c})\geq 1- \frac{1}{n}\right\}.
	\end{align*} For the set $K$ defined in \eqref{6.3}, an application of Markov's inequality yields \begin{align*}
		L_{n}(K^{c}) \leq \frac{1}{M^{2}} L_{n}\left(\|\partial_{\xi} x\|^{2}_{\L^2(\1)}\right).\end{align*}
	For any fixed real number $\lambda_0$ such that $0<\lambda_0 <\frac{\pi^2\nu}{2\|\Q\|_{\mathcal{L}(\L^2(\1))}},$ using Markov's inequality, we obtain 
	\begin{align}\label{6.4}\nonumber
		\P_{\varrho}\left\{\tau_K^{(1)}\geq n\right\}&  \leq  \P_{\varrho}\left\{L_n\left(\|\partial_{\xi}x\|_{\L^2}^2\right)\geq M^2\left(1-\frac{1}{n}\right)\right\} \nonumber\\& 
		\leq \P_{\varrho}\left\{\frac{\nu\lambda_0}{n}\int_{0}^{n}\|\partial_{\xi}u(s)\|_{\L^2}^2\d s\geq \nu\lambda_0M^2\bigg(1-\frac{1}{n}\bigg)\right\}\nonumber\\& 
		\leq \exp\bigg(-n\nu\lambda_0M^2\bigg(1-\frac{1}{n}\bigg)\bigg)\E^{\varrho}\bigg\{\exp\bigg(\nu\lambda_0\int_{0}^{n}\|\partial_{\xi}u(s)\|_{\L^2}^2\d s\bigg)\bigg\}.
	\end{align}For any initial measure $\3\in \M_1(\L^2(\1))$, integrating the exponential estimate \eqref{5.21} w.r.t. $\3(\d x)$, we find
	\begin{align*}
		\E^{\3}\bigg\{\exp\bigg(\lambda_0\nu \int_{0}^{n}\|\partial_{\xi}u(s)\|_{\L^2}^2\d s \bigg)\bigg\} \leq e^{\lambda_0t(\Tr (\Q)+C)}\3(e^{\lambda_0\|\cdot\|_{\L^2}^2}).
	\end{align*}Substituting the above inequality in \eqref{6.4}, we get
	\begin{align}\label{6.5}
		\P_{\varrho}(\tau_K^{(1)}\geq n) \leq \3(e^{\lambda_0\|\cdot\|_{\L^2}^2})e^{-nC_1\lambda_0}, \ \text{ for all }\  n\geq 2,
	\end{align} where the constant $C_1:=\frac{M^2}{2}-\Tr(\Q)-C(\beta,\delta)$ and $C(\beta,\delta)=\big(\frac{2}{\beta(\delta+1)}\big)^{\frac{1}{\delta}}\frac{\delta}{\delta+1}$.
	Fix $\lambda>0$. Using integration by parts formula and \eqref{6.5}, we deduce
	\begin{align*}
		\mathbb{E}^{\varrho}[ e^{\lambda\tau^{(1)}_{K}}] &= 1+ \int_{0}^{+\infty} \lambda e^{\lambda t}\mathbb{P}_{\varrho}(\tau^{(1)}_{K}>t)\d t\leq 1+\sum_{n=0}^{\infty}  \lambda e^{\lambda n}\mathbb{P}_{\varrho}(\tau^{(1)}_{K}>n) \\ &
		\leq e^{2\lambda} + \sum_{n\geq 2} \lambda e^{\lambda(n+1)}\mathbb{P}_{\varrho}(\tau^{(1)}_{K}>n) \\ & \leq e^{2\lambda} \bigg(1+\lambda\varrho (e^{\lambda_{0}\|\cdot\|^{2}})\sum_{n\geq 2}e^{-n(\lambda_{0}C_1-\lambda)} \bigg).
	\end{align*} Using definition \eqref{6.3} of the subset $K$, we can choose the constant $M$ appearing in the definition \eqref{6.3} of $K$ such that $\lambda_{0}C_1-\lambda\geq 1$. Also note that for any $x\in K$, we can use Poincar\'e inequality as $\|x\|_{\L^2}^2\leq  \frac{\|\partial_{\xi}x\|_{\L^2}^2}{\pi^2}\leq \frac{M^2}{\pi^2}$. Taking the supremum over the set $\{\3=\delta_{x}\; x\in K\}$, we find 
	\begin{align*}
		\sup_{x \in K}	\mathbb{E}^{\varrho}[ e^{\lambda\tau^{(1)}_{K}}] \leq  e^{2\lambda} \bigg(1+\lambda e^{\frac{\lambda_0M^2}{\pi^2}}\sum_{n\geq 2}e^{-n(\lambda_{0}C_1-\lambda)} \bigg) <\infty,
	\end{align*} and hence \eqref{6.3} holds. We can obtain \eqref{6.3}  by the same procedure. From the definition of $\tau_k$, we have $\tau_k\leq \tau_k^{(1)}$ and hence one can compute that
	\begin{align*}
		\sup_{\3\in\mathcal{M}_{\lambda_0,R}}\E^{\3}[ e^{\lambda\tau_K}] \leq \sup_{\3\in\mathcal{M}_{\lambda_0,R}}\E^{\3}[ e^{\lambda\tau_K^{(1)}}] \leq e^{2\lambda} \bigg(1+\lambda R\sum_{n\geq 2}e^{-n(\lambda_{0}C_1-\lambda)} \bigg) <\infty,
	\end{align*}which finishes the proof. \\
	\begin{proof}[Proof of Theorem \ref{thrm5.1}]
		We have proved \eqref{6.1} and \eqref{6.2} of Lemma \ref{lem5.1}, which  provide the good uniform upper  and lower bounds of the large deviations, that is, part (i) to (iii) of Theorem \ref{thrm5.1}.
		The first part of \eqref{5.4}, that is, $\J(\varrho)<\infty \implies \varrho \ll \mu$ is given in Lemma \ref{lem4.5}.
		The second part in \eqref{5.4}, that is,   for $\varrho \in \M_1(\L^2(\1))$ with $\J(\varrho)<\infty,\; \varrho(\|\partial_{\xi} x\|_{\L^2}^2)<\infty$ can be established in a similar way as in the proof of Theorem 1.1, \cite{MG2}. 
	\end{proof}
	\begin{proof}[Proof of Corollary \ref{cor4.2}]
		The exponential estimate in Proposition \ref{prop4.10} is sufficient to extend the LDP of Theorem \ref{thrm5.1}, for the unbounded functionals and its consequences. The proof will be on the similar lines as  in the works \cite{MG2, MG1}, etc.
	\end{proof}
	
	\begin{remark}
		As discussed in \cite{GDG}, one can consider the following SGBH equation perturbed by multiplicative  (or correlated)  random force also: 
		\begin{equation}\label{22}
			\left\{
			\begin{aligned}
				\d  u(t)&=\{-\nu\A u(t)-\alpha\B (u(t))+\beta \c(u(t))\}\d t + g(u(t))\d \W(t), \ t\in(0,T), \\
				u(0)&= x,
			\end{aligned}
			\right.
		\end{equation}
		where $g:\L^2(\1)\to[a,b]$ is Lipschitz continuous, $0<a<b<\infty$. The analysis of such problems will be carried out in a future work. One can also consider  SGBH equation perturbed by $\alpha$-stable noise and establish ergodicity results as discussed in \cite{RWJXLX,LXu}, etc. This problem will also be considered in a future work. 
	\end{remark}

	\medskip\noindent
	{\bf Acknowledgments:} The first author would like to thank Ministry of Education, Government of India - MHRD for financial assistance. M. T. Mohan would  like to thank the Department of Science and Technology (DST), India for Innovation in Science Pursuit for Inspired Research (INSPIRE) Faculty Award (IFA17-MA110).  
	
	\medskip\noindent
	{\bf Data availability:} 
	Data sharing not applicable to this article as no datasets were generated or analysed during the current study.

\end{document}